\documentclass[12pt]{amsart}
\usepackage{amssymb, eucal, amsfonts, amsmath, xypic,latexsym}

\newtheorem{lem}{Lemma}
\newtheorem{thm}{Theorem}
\newtheorem{prop}{Proposition}
\newtheorem{rem}{Remark}
\newtheorem{cor}{Corollary}

\newtheorem{defn}{Definition}
\renewcommand{\t}{\mathfrak{t}}
\newcommand{\C}{\mathbb{C}}
\newcommand{\Z}{\mathbb{Z}}
\newcommand{\N}{\mathbb{N}}

\newcommand{\K}{\Bbbk}

\newcommand{\g}{\mathfrak{g}}
\newcommand{\cc}{\mathfrak{c}}
\newcommand{\gl}{\mathfrak{gl}}
\newcommand{\HH}{\mathfrak{H}}
\newcommand{\NN}{\mathfrak{N}}

\newcommand{\h}{\mathfrak{k}}
\newcommand{\li}{\mathfrak{l}}
\newcommand{\z}{\mathfrak{z}}
\newcommand{\mm}{\mathfrak{m}}
\newcommand{\Oo}{\mathcal{O}}

\newcommand{\PP}{\mathcal{P}_\epsilon}
\newcommand{\ii}{\textnormal{\textbf{i}}}
\newcommand{\jj}{\textnormal{\textbf{j}}}

\newcommand{\rank}{\text{rank}}

\newcommand{\End}{\text{End}}

\newcommand{\ad}{\text{ad}}
\newcommand{\Ker}{\text{Ker}}
\newcommand{\im}{\text{Im}}

\newcommand{\Ad}{\text{Ad}}

\newcommand{\spn}{\text{span}}
\newcommand{\Lie}{\text{Lie}}

\newcommand{\Ind}{\textnormal{Ind}}
\newcommand{\ra}{\rightarrow}
\begin{document}
\title{Derived subalgebras of centralisers and finite $W$-algebras}
\author{Alexander Premet and Lewis Topley}
\thanks{\nonumber{\it Mathematics Subject Classification} (2000 {\it revision}).
Primary 17B35, 17B20. Secondary 17B63.}
\keywords{Simple Lie algebra, nilpotent element, primitive ideal, finite $W$-algebra}
\address{School of Mathematics, University of Manchester, Oxford Road,
M13 9PL, UK} \email{Alexander.Premet@manchester.ac.uk}
\email{Lewis.Topley@uea.ac.uk} \maketitle

\begin{abstract}
Let $\g=\Lie(G)$ be  the Lie algebra of a simple algebraic group $G$
over an algebraically closed field of characteristic $0$. Let $e$ be
a nilpotent element of $\g$ and let $\g_e=\Lie(G_e)$ where $G_e$
stands for the stabiliser of $e$ in $G$. For $\g$ classical, we give
an explicit combinatorial formula for the codimension of
$[\g_e,\g_e]$ in $\g_e$ and use it to determine those $e\in\g$ for
which the largest commutative quotient $U(\g,e)^{\rm ab}$ of the
finite $W$-algebra $U(\g,e)$ is isomorphic to a polynomial algebra.
It turns out that this happens if and only if $e$ lies in a unique
sheet of $\g$. The nilpotent elements with this property are called
{\it non-singular} in the paper. Confirming a recent conjecture of
Izosimov we prove that a nilpotent element $e\in\g$ is non-singular
if and only if the maximal dimension of the geometric quotients
$\mathcal{S}/G$, where $\mathcal{S}$ is a sheet of $\g$ containing
$e$, coincides with the codimension of $[\g_e,\g_e]$ in $\g_e$ and
describe all non-singular nilpotent elements in terms of partitions.
We also show that for any nilpotent element $e$ in a classical Lie
algebra $\g$ the closed subset of ${\rm Specm}\, U(\g,e)^{\rm ab}$
consisting of all points fixed by the natural action of the
component group of $G_e$ is isomorphic to an affine space. Analogues
of these results for exceptional Lie algebras are also obtained and
applications to the theory of primitive ideals are given.
\end{abstract}
\maketitle

\bigskip

\section{Introduction and preliminaries}
\subsection{}
Let $\K$ be an algebraically closed field of characteristic $0$ and
let $G$ be a simple algebraic group of adjoint type over $\K$. Given
an element $x$ in the Lie algebra $\g={\rm Lie}(G)$ we write $G_x$
for the (adjoint) stabiliser of $x$ in $G$ and denote by $\g_x$ the
Lie algebra of $G_x$. It is well known that $\g_x$  coincides with
the centraliser of $x$ in $\g$.

Let $U(\g)$ for the universal enveloping algebra of $\g$ and denote
by $\mathcal{X}$ the set of all primitive ideals of $U(\g)$. By the
PBW theorem, the graded algebra associated with the canonical
filtration of $U(\g)$ is isomorphic to the symmetric algebra $S(\g)$
which we identify with $S(\g^*)$ by using the Killing form on $\g$.
Using Commutative Algebra we then attach to $I\in \mathcal{X}$ two
important invariants: the associated variety ${\rm VA}(I)$ and the
associated cycle ${\rm AC}(I)$. The variety ${\rm VA}(I)$ is the
zero locus in $\g$ of the $G$-stable ideal ${\rm gr}(I)$ of
$S(\g^*)$ and ${\rm AC}(I)$ is a formal linear combination
$\sum_{i=1}^l m_i[\mathfrak{p}_i]$ where
$\mathfrak{p}_1,\ldots,\mathfrak{p}_l$ are the minimal primes of
$S(\g^*)$ over ${\rm Ann}_{S(\g^*)}\,{\rm gr}(U(\g)/I)$ and
$m_1,\ldots,m_l$ are their multiplicities; see
\cite[Section~9]{Ja04} where notation is slightly different. Since
the variety ${\rm VA}(I)$ is irreducible by Joseph's theorem
\cite{Jo3} and hence coincides with the Zariski closure of a
nilpotent orbit $\Oo \subset\g$, we have that ${\rm AC}(I)=m_I[J]$
where $m_I\in\N$ and $J=\sqrt{{\rm gr}(I)}$, a prime ideal of
$S(\g^*)$. The positive integer $m_I$ is sometimes referred to as
the {\it multiplicity} of $\Oo$ in the primitive quotient $U(\g)/I$
and abbreviated as ${\rm mult}_{\Oo}(U(\g)/I)$. It is well known
that if $\Oo=\{0\}$ then $I$ coincides with the annihilator in
$U(\g)$ of a finite dimensional irreducible $\g$-module $V$, the
radical $J=\sqrt{{\rm gr}(I)}$ identifies with the ideal
$\bigoplus_{i>0} S^i(\g^*)$  and $m_I=(\dim \, V)^2$.
\subsection{} From now on we let $e$ be a nonzero nilpotent element of $\g$ and
include it into an $\mathfrak{sl}_2$-triple $\{e,h,f\}\subset \g$.
Let $U(\g,e)$ be the finite $W$-algebra associated with the pair
$(\g,e)$, a non-commutative filtered deformation of the coordinate
algebra $\K[e+\g_f]$ on the Slodowy slice $e+\g_f$ regarded with its
Slodowy grading. Recall that $U(\g,e)=\big({\rm End}_\g\,
Q_e\big)^{\rm op}$ where $Q_e$ stands for a generalised
Gelfand--Graev $\g$-module associated with $e$; see \cite{Sasha1},
\cite{gg} for more detail. By a result of Skryabin, proved in the
appendix to \cite{Sasha1}, the right $U(\g,e)$-module $Q_e$ is free
and for any irreducible $U(\g,e)$-module $V$ the $\g$-module
$Q_e\otimes_{U(\g,e)} V$ is irreducible. As a consequence, the
annihilator $I_V:={\rm Ann}_{U(\g)}\big(Q_e\otimes_{U(\g,e)} V\big)$
is a primitive ideal of $U(\g)$.

Let $\Oo$ be the adjoint $G$-orbit of $e$ and define
$\mathcal{X}_{\Oo}:=\{I\in\mathcal{X}\,|\,\,{\rm
VA}(I)=\overline{\Oo}\}$. By \cite{Sasha2},
$I_V\in\mathcal{X}_{\Oo}$ for any finite dimensional irreducible
$U(\g,e)$-module $V$, whilst \cite{Lo2}, \cite{Sasha3} and \cite{gi}
show that any primitive ideal $I\in \mathcal{X}_{\Oo}$ has the form
$I_W$ for some finite dimensional irreducible $U(\g,e)$-module $W$.
As explained in \cite{Sasha3}, there is a natural action of the
component group $\Gamma=G_e/G_e^\circ$ on the set ${\rm
Irr}\,U(\g,e)$ of all isoclasses of finite dimensional irreducible
$U(\g,e)$-modules. It is straightforward to check that the primitive
ideal $I_W$ depends only on the isoclass of $W$ and so one can speak
of a primitive ideal $I_{[W]}$ where $[W]$ is the isoclass of $W$ in
${\rm Irr}\,U(\g,e)$; see \cite[Corollary~4.1]{Sasha3}, for
instance. In \cite{Lo5} Losev showed that
\begin{eqnarray}\label{Losev}{\rm mult}_{\Oo}(U(\g)/I_W)=
[\Gamma:\Gamma_W]\cdot(\dim\,W)^2\end{eqnarray} where $\Gamma_W$
denotes the stabiliser of the isoclass $[W]$ in $\Gamma$.
Furthermore, confirming a conjecture of the first-named author he
proved in {\it loc.\,cit.} that the equality $I_{[W]}=I_{[W']}$
holds for $[W],[W']\in {\rm Irr}\,U(\g,e)$ if and only if
$[W']=\,^{\!\gamma}[W]$ for some $\gamma\in \Gamma$. In particular,
this means that $\dim\,W$ is an intrinsic invariant of the primitive
ideal $I=I_W\in\mathcal{X}_{\Oo}$.

By Goldie's theory, for any $I\in\mathcal{X}$ the prime Noetherian
ring $U(\g)/I$ embeds into a full ring of fractions. The latter ring
is prime Artinian and hence isomorphic to the matrix algebra ${\rm
Mat}_n(\mathcal{D}_I)$ over a skew-field $\mathcal{D}_I$ called the
{\it Goldie field} of $U(\g)/I$. The positive integer $n=n_I$
coincides with the Goldie rank of $U(\g)/I$ which is often
abbreviated as ${\rm rk}(U(\g)/I)$.

Recall that a primitive ideal $I$ is called {\it completely prime}
if $U(\g)/I$ is a domain. It is well known that this happens if and
only if ${\rm rk}(U(\g)/I)=1$. Classifying the completely prime
primitive ideals of $U(\g)$ is an old-standing classical (and much studied) problem of
Lie Theory. In general, it remains open outside type $\sf A$
although many important partial results can be found in \cite{BV},
\cite{Bo1}, \cite{BJ}, \cite{RBr}, \cite{Jo0}, \cite{Jo4},
\cite{Lo3}, \cite{Lo1}, \cite{McG}, \cite{Moe}, \cite{Moe1} and
references therein. If $I=I_V\in\mathcal{X}_{\Oo}$, where
$[V]\in{\rm Irr}\,U(\g,e)$, then the main result of \cite{Sasha4}
states that the number $$q_I:=\frac{\dim\,V}{{\rm
rk}(U(\g)/I)}$$ is an integer, and it is also proved in {\it
loc.\,cit.} that $q_I=1$ if the Goldie field $\mathcal{D}_I$ is
isomorphic to the skew-field of fractions of a Weyl algebra.  The
integrality of $q_I$ implies  that $I_V$ is completely prime
whenever $\dim\,V=1$ (this fact also follows from results of
M{\oe}glin \cite{Moe1} and Losev \cite{Lo2}).

Obviously, $I=I_V$ is completely prime if and only if $q_I=\dim\,V$.
If $\Gamma=\{1\}$ then combining (\ref{Losev}) with Joseph's results
on Goldie-rank polynomials \cite{Jo1} (as exposed in
\cite[12.7]{Ja}) it is straightforward to see that the scale factor
$q_I$ takes the same value on coherent families of primitive ideals
in $\mathcal{X}_{\Oo}$; see \cite[5.3]{Lo4} for more detail. It
seems likely that this holds without any assumption on $\Gamma$ and
the entire set $\{q_I\,\colon\,\,I\in\mathcal{X}\}$ is finite. (By a
coherent family of primitive ideals we mean any subset $\{I(w\cdot
\mu)\colon\,\mu\in\Lambda^+\}$ of $\mathcal{X}_{\Oo}$ with $\mu$ and
$w$ satisfying the assumptions of \cite[12.7]{Ja}.) We mention for
completeness that outside type $\sf A$ there are examples of
completely prime primitive ideals $I\in\mathcal{X}_{\Oo}$ for which
$q_I>|\Gamma|$ (see \cite[Remark~4.3]{Sasha4}), but it is proved in
\cite{Lo4} for $\g$ classical (and conjectured for $\g$ exceptional)
that $q_I=1$ whenever the central character of $I$ is integral.
\subsection{}
In this paper we begin a systematic investigation of those
$I\in\mathcal{X}_{\Oo}$ for which ${\rm mult}_{\Oo}(U(\g)/I)=1$; we
call such primitive ideals {\it multiplcity free}. For $\g$
classical we impose no assumptions on $e$, but for $\g$ exceptional
we shall assume that the orbit $\Oo$ is induced in the sense of
Lusztig--Spaltenstein from a nilpotent orbit in a proper Levi
subalgebra of $\g$. The remaining case of rigid (i.e. non-induced)
orbits in exceptional Lie algebras is dealt with in \cite{Pr}.
As we explained earlier, any multiplicity free primitive ideal is
completely prime, but the converse may not always be true outside
type $\sf A$.

Let $\mathcal{S}_1,\ldots, \mathcal{S}_t$ be all sheets of $\g$
containing $\Oo$. For $1\le i\le t$, set
$r_i=\dim\,\mathcal{S}_i-\dim\,\Oo$, the rank of $\mathcal{S}_i$,
and define $$r(e):=\max_{1\le i\le t}\,r_i.$$ Let
$\cc_e=\g_e/[\g_e,\g_e]$. Since any $1$-dimensional torus of $G_e$
and any unipotent element $u=\exp(\ad\,n)$ with $n\in\g_e$ act
trivially on $\cc_e$, it is straightforward to see that the adjoint
action of $G_e$ on $\g_e$ induces the trivial action of the
connected  group $G_e^\circ$ on $\cc_e$ and hence gives rise to a
natural action of $\Gamma$. We denote by $\cc_e^\Gamma$ the
corresponding fixed point space, i.e. the set of all $x\in\cc_e$
such that $\gamma(x)=x$ for all $\gamma\in\Gamma$. We define
$$c(e):=\dim(\cc_e),\qquad\quad
c_\Gamma(e):=\dim(\cc_e^\Gamma).$$

Let $U(\g,e)^{\rm ab}=U(\g,e)/I_c$ where $I_c$ is the two-sided
ideal of $U(\g,e)$ generated by all commutators $u\cdot v-v\cdot u$
with $u,v\in U(\g,e)$. Our assumption on $\Oo$ in conjunction with
\cite{RBr}, \cite{Lo2} and \cite{GRU} guarantees that $I_c$ is a
proper ideal of $U(\g,e)$; see \cite{Sasha3} for more detail. We
denote by $\mathcal{E}$ the maximal spectrum of the finitely
generated commutative $\K$-algebra $U(\g,e)^{\rm ab}$. This affine
variety parametrises the $1$-dimensional representations of
$U(\g,e)$ and is acted upon by the the component group $\Gamma$ (it
is known that $\Gamma$ acts on $U(\g,e)^{\rm ab}$ by algebra
automorphisms). We denote by $\mathcal{E}^\Gamma$ the corresponding
fixed point set which consists of all $\eta\in\mathcal{E}$ such that
$\gamma(\eta)=\eta$ for all $\gamma\in\Gamma$. Let $I_\Gamma$ be the
ideal of $U(\g,e)^{\rm ab}$ generated by all $\phi-\phi^\gamma$ with
$\phi\in U(\g,e)^{\rm ab}$ and $\gamma\in\Gamma$. It is
straightforward to see that $\mathcal{E}^\Gamma$ coincides with the
zero locus of $I_\Gamma$ in $\mathcal{E}$. We define
$U(\g,e)_\Gamma^{\rm ab}:=\,U(\g,e)^{\rm ab}/I_\Gamma$.

It follows from (\ref{Losev}) that $I=I_V$ is multiplicity free if
and only if $\dim\,V=1$ and $\Gamma_V=\Gamma$. Thus, in order to
classify the multiplicity-free primitive ideals in
$\mathcal{X}_{\Oo}$ we need  to determine the variety
$\mathcal{E}^\Gamma$. This problem is important as solving it could
eventually lead us to a complete description of {\it all}
quantisations of nilpotent orbits; see  to \cite{Moe1} and
\cite[Theorem~1.1]{Lo1} for precise statements.

Thanks to \cite[Theorem~1.2]{Sasha3} we know that
$\dim\,\mathcal{E}=r(e)$ and the number of irreducible components of
$\mathcal{E}$ is greater than or equal to $t$. Thus, the variety
$\mathcal{E}$ is irreducible only if $e$ lies in a unique sheet of
$\g$. For $\g=\mathfrak{sl}_n$, this condition is satisfied for any
nilpotent element $e$ and \cite[Corollary~3.2]{Sasha3} states that
$U(\mathfrak{sl}_n,e)^{\rm ab}$ is a polynomial algebra in $r(e)$
variables. Our first main result is a generalisation of that to all
Lie algebras of classical types. We call an element $a\in \g$ {\it
non-singular} if it lies in a unique sheet of $\g$. If
$\dim\,\g_a=m$ and $\g^{(m)}= \{x\in\g\colon\,\dim\,\g_x=m\}$, a
locally closed subset of $\g$, then it follows from the smoothness
of sheets of classical Lie algebras (proved by Im Hof in \cite{Im})
that $a$ is non-singular if and only if $a$ is a smooth point of the
quasi-affine variety $\g^{(m)}$ (hence the name).
\begin{thm}\label{A}
If $e$ is a nilpotent element in a classical Lie algebra $\g$, then
the following are equivalent:
\begin{itemize}
\item[(i)\,] $e$ is non-singular;

\smallskip

\item[(ii)\,] $c(e)=r(e)$;

\smallskip

\item[(iii)\,] $U(\g,e)^{\rm ab}$ is isomorphic to a polynomial algebra in $r(e)$ variables.
\end{itemize}
\end{thm}
The equivalence of the first two statements of Theorem~\ref{A} was
conjectured by Izosimov  \cite{Iz} for all elements in a classical
Lie algebra $\g$. In Remark~\ref{izosim}, we use the
Jordan--Chevalley decomposition in $\g$ to show that his conjecture
is an immediate consequence of Theorem~\ref{A}.

Although the polynomiality of $U(\g,e)^{\rm ab}$ occurs rather
infrequently outside type $\sf A$, the algebras $U(\g,e)^{\rm
ab}_\Gamma$ exhibit a much more uniform behaviour:
\begin{thm}\label{B}
If $e$ is any nilpotent element in a classical Lie algebra $\g$,
then $U(\g,e)_\Gamma^{\rm ab}$ is isomorphic to a polynomial algebra
in $c_\Gamma(e)$ variables. In particular $\mathcal{E}^\Gamma$ is a
single point if and only if $c_\Gamma(e)=0$.
\end{thm}
\noindent As an obvious corollary of Theorem~\ref{B} we deduce that
the variety $\mathcal{E}^\Gamma$ is isomorphic to an affine space
for any nilpotent element in a classical Lie algebra and hence is
irreducible.
\subsection{}
In order to prove Theorems~\ref{A} and \ref{B} we have to look very
closely at the centralisers of nilpotent elements in classical Lie
algebras. A link between completely prime primitive ideals and
centralisers of nilpotent elements originates in the fact that for
any nilpotent element $e\in\g$ the finite $W$-algebra $U(\g,e)$ is a
filtered deformation of the universal enveloping algebra $U(\g_e)$;
see \cite{Sasha2} and \cite{BGK}.

Suppose $\g$ is one of $\mathfrak{so}_N$ or $\mathfrak{sp}_N$. It is
well known that to any nilpotent element $e\in \g$ one can attach a
partition $\lambda\in\mathcal{P}_\epsilon(N)$ where $\epsilon=1$ if
$\g=\mathfrak{so}_N$ and $\epsilon=-1$ if $\g=\mathfrak{sp}_N$.
Recall that a partition $\lambda=(\lambda_1, \ldots, \lambda_n)$ of
$N$ with $\lambda_1\ge \cdots\ge \lambda_n\ge 1$ is in
$\mathcal{P}_\epsilon(N)$ if there is an involution $i\mapsto i'$ on the set
of indices $\{1,\ldots, n\}$ satisfying $i'\in\{i-1,i,i+1\}$
such that $\lambda_{i'}=\lambda_i$ and $i'=i$ if and
only if $\epsilon(-1)^{\lambda_i}=-1$ for all $i$. We call a pair of
indices $(i,i+1)$ with $1\le i<n$ a {\it 2-step} of $\lambda$ if
$i'=i$, $(i+1)'=i+1$ and
$\lambda_{i-1}\ne\lambda_i\ge
\lambda_{i+1}\ne\lambda_{i+2}$ where
our convention is that $\lambda_i=0$ for $i\in\{0,n+1\}$. We denote
by $\Delta(\lambda)$ the set of all 2-steps of $\lambda$ and set
 $$s(\lambda):=\,\textstyle{\sum}_{i=1}^n\lfloor
(\lambda_i-\lambda_{i+1})/2\rfloor.$$ We call $\lambda$ {\it exceptional}
if $\h$ has type $\sf D$ and there exists a $k< n$ such that the parts
$\lambda_{k},\lambda_{k+1}$ are odd and the parts $\lambda_i$
with $i\not\in \{k, k+1\}$ are all even.

It should be mentioned that for any $(i,i+1)\in\Delta(\lambda)$ the
integers $\lambda_i$ and $\lambda_{i+1}$ have the same parity. If
$(i,i+1)\in \Delta(\lambda)$ and $i>1$ (resp. $i=1$), then we call
$\lambda_{i-1}$ and $\lambda_{i+2}$ (resp. $\lambda_3$) the {\it
boundary} of $(i,i+1)$. We say that a 2-step $(i,i+1)$ is {\it good}
if its boundary and $\lambda_i$ have the opposite parity.
\begin{thm}\label{C}
Let $\g$ be one of $\mathfrak{so}_N$ or $\mathfrak{sp}_N$, where
$N\ge 2$, and let $e$ be a nilpotent element of $\g$ associated with
a  partition $\lambda\in\mathcal{P}_\epsilon(N)$. Then the following
hold:
\begin{itemize}
\item[(i)\,] $c(e)=s(\lambda)+|\Delta(\lambda)|$;

\smallskip

\item[(ii)\,] $c_\Gamma(e)=s(\lambda)$ unless $\g=\mathfrak{so}_N$ and
$\lambda\in\mathcal{P}_1(N)$ is exceptional, in which case $c_\Gamma(e) = s(\lambda)+1$;
\smallskip

\item[(iii)\,] $e$ is non-singular if and only if
all 2-steps of $\lambda$ are good.
\end{itemize}
\end{thm}
\noindent
 If $\lambda\in\mathcal{P}_1(N)$ is  exceptional, then it is immediate
 from the definitions that $|\Delta(\lambda)|=1$ and the only 2-step of $\lambda$ is good.
 Therefore, any nilpotent element $e\in\g$ associated with $\lambda$ is
 non-singular. It is also straightforward to see that any such $e$ is a Richardson element of $\g$.

For a nilpotent element $e$ associated with a partition
$\lambda\in\mathcal{P}_\epsilon(N)$, we give an explicit
combinatorial formula for the number $r(e)$; see
Corollary~\ref{newz}. It involves the notion of a {\it good
2-cluster} of $\lambda$ introduced in Subsection~\ref{3.3}.

\subsection{} Now suppose that $\g$ is an exceptional Lie algebra. In this case our
results are less complete because we have to exclude the following seven induced orbits:

\smallskip

\begin{center}
{\sc Table~0. Unresolved cases.}
\end{center}
\begin{center}

\medskip

\begin{tabular}{|c|c|c|c|c|c|c|}
\hline
${\sf F_4}$ & ${\sf E_6}$ & ${\sf E_7}$& ${\sf E_8}$ & ${\sf E_8}$ &${\sf E_8}$ &${\sf E_8}$
\\
\hline
 ${\sf C_3(a_1)}$&
 ${\sf A_3+A_1}$ &  ${\sf D_6(a_2)}$   & ${\sf E_6(a_3)+A_1}$ &    ${\sf D_6(a_2)}$ & ${\sf E_7(a_2)}$
&${\sf E_7(a_5)}$\\
\hline
\end{tabular}
\end{center}

\medskip

Using \cite[pp.~440--445]{C} one observes that all orbits listed in Table~0 are non-special.
\begin{thm}\label{D}
Let $\g$ be an exceptional Lie algebra and suppose that $e$ is an induced nilpotent element of $\g$.
Then the following hold:
\begin{itemize}
\item[(i)\,] $\mathcal{E}^\Gamma\ne \emptyset$.

\smallskip

\item[(ii)\,] If $e$ is not listed in the first six columns of Table~0 and lies in a
single sheet of $\g$, then $U(\g,e)^{\rm ab}$ is isomorphic to a polynomial algebra in $c(e)$ variables.

\smallskip

\item[(iii)\,] If $e$ is not listed in Table~0, then $U(\g,e)^{\rm ab}_\Gamma$ is isomorphic to a
polynomial algebra in $c_\Gamma(e)$ variables.
\end{itemize}
The numbers $c(e)$ and $c_\Gamma(e)$ are listed in
the last two columns of Tables~1--6.
\end{thm}

Curiously, there are instances where for an induced element $e$ the
variety $\mathcal{E}^\Gamma$ is a single point. For $\g$ exceptional
there are four such cases (two in type ${\sf E}_7$ and two in type
${\sf E_8}$) and for $\g$ classical this occurs when $e$ is
associated with a partition $\lambda\in\mathcal{P}_\epsilon(N)$ for
which $\lambda_i-\lambda_{i+1}\in\{0,1\}$ for all $i$ (we call such
partitions {\it almost rigid}). The nilpotent elements from the four
orbits in types ${\sf E_7}$ and ${\sf E_8}$ have already appeared in
the literature under three different names: $p$-{\it compact}, {\it
compact} and {\it reachable}; see \cite{BB}, \cite{EG}, \cite{Pa},
\cite{Ya10}, \cite{deG2}. It is worth mentioning that almost rigid
and exceptional partitions in $\mathcal{P}_\epsilon(N)$ also played
a special role in Namikawa's work \cite{Na} on
$\mathbb{Q}$-factorial terminalisations of nilpotent orbit closures
in classical Lie algebras.

In proving Theorem~\ref{D} we rely heavily on results of de Graaf
\cite{deG2} and Lawther--Testerman \cite{LT1} obtained by
computational methods. It seems plausible that the algebra
$U(\g,e)^{\rm ab}_\Gamma$ is reduced and the variety
$\mathcal{E}^\Gamma$ is equidimensional in all cases, but to prove
this for the orbits listed in Table~0 one would have to use
different methods (a computational approach in the spirit of
\cite{GRU} would certainly do the trick).
\subsection{} The traditional way to classify the completely prime
ideals $I\in\mathcal{X}_\Oo$ parallels Borho's classification of the
sheets of $\g$; see \cite{Bor}. Here one aims to show that if the
orbit $\Oo$ is induced from a rigid orbit $\Oo_0$ in a Levi
subalgebra $\li$ of $\g$, then the majority of $I$ as above can be
obtained as the annihilators in $U(\g)$ of (not necessarily
irreducible) induced $\g$-modules
$${\rm Ind}_{\mathfrak p}^{\g}(E):=U(\g)\otimes_{U({\mathfrak p})}E,$$ where
$\mathfrak{p}=\li\oplus\mathfrak{n}$ is a parabolic subalgebra of
$\g$ with nilradical $\mathfrak n$ and $E$ is an irreducible
$\mathfrak{p}$-module with $\mathfrak{n}\cdot E=0$ such that the
annihilator $I_0:={\rm Ann}_{U(\li)}\, E$ is a completely prime
primitive ideal of $U(\li)$ with ${\rm VA}(I_0)= \overline{\Oo}_0$.
The ideals
$$I(\mathfrak{p},E):={\rm Ann}_{U(\g)}\big({\rm Ind}_{\mathfrak p}^\g(E)\big)$$ are referred to as {\it induced}.
It should be mentioned that $I(\mathfrak{p},E)$ does not have to be
primitive and completely prime, in general, but this holds under the
additional assumption that $I_0$ is completely prime thanks to
Conze's theorem \cite{Co} and the Dixmier--M{\oe}glin equivalence
\cite[8.5.7]{Di}. It is well known that $I(\mathfrak{p},E)$
coincides with the largest two-sided ideal of $U(\g)$ contained in
the left ideal $U(\g)(\mathfrak{n}+I_0)$ and hence depends only on
$\mathfrak p$ and $I_0$; see \cite[10.4]{BGR}. We shall sometimes
use a more flexible notation $\mathfrak{I}_{\mathfrak p}^\g(I_0)$
when referring to $I(\mathfrak{p},E)$

Motivated by the natural desire to keep things simple, one wants
{\it all} completely prime primitive ideals  in $\mathcal{X}_{\Oo}$
to be induced, but since this fails outside type $\sf A$ one must
find a way to determine the non-induced ones. This is, of course,
the hardest part of the problem and the main reason why the
classification remains open outside type $\sf A$; see \cite{BJ} for
more detail.

Fortunately, this issue does not arise for the multiplicity-free primitive ideals.
The following is the main result of this paper:

\begin{thm}\label{E} Let $I\in\mathcal{X}_{\Oo}$ be a multiplicity-free primitive
ideal associated with an induced nilpotent orbit $\Oo\subset \g$. If
$\g$ is exceptional assume further that $\Oo$ is not listed in
Table~0. Then there exists a proper parabolic subalgebra $\mathfrak
p$ of $\g$ with a Levi subalgebra $\li$ and a rigid nilpotent orbit
$\Oo_0$ in $\li$ such that $\Oo$ is induced from $\Oo_0$ and
$I=I({\mathfrak p},E)$ where $E$ is an irreducible
$U(\mathfrak{p})$-module with the trivial action of the nilradical
of $\mathfrak p$. Moreover, the primitive ideal $I_0={\rm
Ann}_{U(\li)\,} E$ is completely prime and ${\rm
VA}(I_0)=\overline{\Oo}_0$.
\end{thm}
Theorem~\ref{E} can be regarded as a generalisation of M{\oe}glin's
theorem \cite{Moe} on completely prime primitive ideals of
$U(\mathfrak{sl}_n)$. From the main body of the paper one can obtain
more information on the parabolic subalgebra $\mathfrak p$ and the
$\mathfrak p$-module $E$. It is quite possible that Theorem~\ref{E}
holds for all induced orbits in $\g$ and this would follow (by the
same argument) if the variety $\mathcal{E}^\Gamma$ turned out to be
irreducible for all orbits listed in Table ~0.

\medskip

\noindent {\bf Acknowledgement.} The authors would like to thank
S.~Goodwin, W.~de Graaf, R.~Lawther, I.~Losev, A.~Moreau, R.~Tange,
D.~Testerman and O.~Yakimova for useful discussions and e-mail
correspondence on the subject of this paper. We are also thankful to the anonymous referee for careful reading, thoughtful suggestions, and pointing out several inaccuracies 
in the first version of this paper.
\section{The derived subalgebra of a centraliser}
\subsection{A basis for centralisers in classical Lie algebras}\label{1.1}

Let $\K$ be an algebraically closed field of any characteristic
$\neq 2$. Fix $N \geq 2$ and denote by $V$ an $N$-dimsnional vector
space over $\K$. In this section we denote by $G$ the algebraic
group $GL(V)$ with Lie algebra $\g = \Lie(G) = \gl(V)$ and let
$\Psi=(\,\cdot\,,\, \cdot\,)$ be a symmetric or skew-symmetric
non-degenerate bilinear form on $V$ with values in $\K$, so that
$(u,v) = \epsilon (v,u)$ for all $u, v\in V$ where $\epsilon =
\pm1$. Choose a basis for $V$ to identify $\gl(V)$ with $\gl_N$ and let $J$ be the matrix associated
to $\Psi$ with respect to that basis. If $X$ is an endomorphism of $V$ then $X^\top$ denotes
the transpose of $X$. There is a Lie algebra automorphism $\sigma :
\g \to \g$ of order 2 taking $X \in \g $ to $-J^{-1} X^\top J$ which
is independent of our choice of basis. Then $\sigma$ induces a
$\Z_2$-grading $\g = \g _0 \bigoplus \g_{1}$. Make the notation $\h
= \g_0$. If $\epsilon = 1$ then $\h$ is an orthogonal algebra, and
if $\epsilon = -1$ then $\h$ is a symplectic algebra. In either case
$\g_{1}$ is an $\h$-module. Let $K$ denote the connected component
of the associated orthogonal or symplectic group.

The conjugacy classes of nilpotent elements in $\g$ are in one to
one correspondence with ordered partitions of $N$: to a partition
$\lambda = (\lambda_1,..., \lambda_n)$ of $N$ with $\lambda_1 \geq
\cdots \geq \lambda_n\ge 1$ we associate the $G$-orbit of the
nilpotent element in Jordan normal form with Jordan block sizes
$\lambda_1,...,\lambda_n$.  Let $e \in \h$ be a nilpotent element
with Jordan block sizes $\lambda_1 \geq \cdots \geq \lambda_n$.
Since $\h$ acts naturally on $V$ we may decompose $V$ uniquely into
minimal $e$-stable subspaces $V = \bigoplus_{i=1}^n V[i]$, and shall
call these $V[i]$ the Jordan block spaces of $e$ in $V$. Since $e$
restricts to a regular nilpotent endomorphism on each $V[i]$, there
exist vectors $\{w_i\}$ such that $\{e^s w_i : 1 \leq i \leq n, 0
\leq s < \lambda_i\}$ forms a basis for $V$. When dealing with
partitions $\lambda$ as above we always assume that $\lambda_0=0$
and $\lambda_i=0$ for all $i>n$.

The following condition on the Jordan block sizes can be found in
\cite[Theorem~1.4]{Ja04}, for example. The final statement follows
from \cite[Theorem~5.1.6]{CM}.
\begin{lem}\label{nilpotents}
The $w_i \in V$ can be chosen so that there exists an involution $i
\mapsto i'$ on the set $\{1,...,n\}$ such that
\begin{enumerate}
\item{$\lambda_i = \lambda_{i'}$ for all $i = 1,...,n$}

\smallskip

\item{$(V[i], V[j]) = 0$ if $i \neq j'$}

\smallskip

\item{$i=i'$ if and only if $\epsilon(-1)^{\lambda_i} = -1$}
\end{enumerate}
\end{lem}
The lemma states that for a nilpotent element in a symplectic Lie
algebra each Jordan block of odd dimension can be paired with a
different Jordan block of the same dimension; in an orthogonal
algebra each Jordan block of even dimension can be paired with a
different Jordan block of the same dimension; and that this pairing
is involutory. Renumbering the vectors $w_i$ if necessary we may
(and will) assume from now on that $$i'\in\{i-1,i,i+1\}\ \ \mbox{for
all }\ 1\le i\le n.$$ As an immediate consequence of this convention
we have that $j'>i'$ whenever $1\le i<j\le n$ and $j\ne i'$.
 Following \cite{CM} we denote by $\mathcal{P}_\epsilon(N)$ the set of partitions
 of $N$ which are associated to nilpotent elements of $\h$ (ie. fulfilling the parity
 conditions of Lemma \ref{nilpotents}).

If $\mathcal L$ is a Lie algebra and $x\in\mathcal{L}$ the we write
$\mathcal{L}_x$ for the centraliser of $x$ in $\mathcal L$. Since
$\sigma(e) = e$, the centraliser of $e$ in $\g$ is $\sigma$-stable,
inducing a decomposition $\g_e = \h_e \oplus (\g_e)_{1}$ where
$(\g_e)_{1} = (\g_{1})_e$ is a $\h_e$-module. Thanks to
\cite[Theorems~2.5, 2.6]{Ja04} we may identify $\h_e$ with
$\Lie(K_e)$. We shall normalise the basis for $V$. Let $\{w_i\}$ be
chosen in accordance with the above
 and fix $1 \leq i \leq n$, $0 < s$. We have $(e^{\lambda_i-1} w_i, e^sw_{i'}) = (-1)^s
 (e^{\lambda_i - 1 + s}w_i, w_{i'})$ and $e^{\lambda_i - 1 + s}w_i = 0$ so
 $e^{\lambda_i-1}w_i$ is orthogonal to all $e^sw_{i'}$ with $s>0$. There is a
 (unique up to scalar) vector $v \in V[i]$ which is orthogonal to all $e^s w_{i'}$ for
 $s < \lambda_i - 1$. This $v$ does not lie in $\im(e)$ for otherwise it would be othogonal to
 all of $V[i] + V[i']$. This is not possible since the restriction of $\Psi$ to $V[i] + V[i']$ is
 non-degenerate. It does no harm to replace $w_i$ by $v$ and normalise according to the rule
\begin{eqnarray*}
(w_i, e^{\lambda_i - 1} w_{i'}) = 1 & \text{ whenever
} i \leq i'
\end{eqnarray*}
With respect to this basis the matrix of the restriction of $\Psi$
to $V[i] + V[i']$ is antidiagonal with entries $\pm 1$.

Let $\xi \in \mathfrak{g}_e$. Then $\xi (e^s w_i) = e^s (\xi w_i)$
showing that $\xi$ is determined by its action on the $w_i$. If we
define
\begin{eqnarray*}
\xi_i^{j,s} w_k = \left\{ \begin{array}{ll}
         e^sw_j & \mbox{ if $i=k$}\\
        0 & \mbox{ otherwise}\end{array} \right.
\end{eqnarray*}
and extend the action to $\{e^s w_i\}$ by the requirement that
$\xi_i^{j,s}$ is linear and centralises $e$ then
\begin{eqnarray}\label{gebasis}
\{\xi_i^{j,\lambda_j - 1 - s} : 1 \leq i,j \leq n , 0 \leq s < \min(\lambda_i, \lambda_j)\}
\end{eqnarray}
forms a basis for $\mathfrak{g}_e$ ; see \cite{Ya06}, for example.
Our next aim is to describe a basis for $\h_e$. The following
approach is implicit in \cite{Ya06}, however we shall recover the
details for the reader's convenience. Since $\sigma : \g_e \to \g_e$
is an involution the maps $\xi + \sigma(\xi)$, with $\xi \in \g_e$,
span $\h_e$. Thanks to (\ref{gebasis}) we may define $\zeta_i^{j,s}
= \xi_i^{j,\lambda_j - 1 - s} + \sigma(\xi_i^{j,\lambda_j - 1-s})$
and conclude that $\{\zeta_i^{j,s} : 1 \leq i,j \leq n , 0 \leq s <
\min(\lambda_i, \lambda_j)\}$ is the required spanning set for
$\h_e$. This leaves us with two immediate tasks: evaluate
$\sigma(\xi_i^{j,\lambda_j-1-s})$ and determine the linear relations
between the $\zeta_i^{j,s}$. Using the fact that $\zeta_i^{j,s}$ is
skew self-adjoint with respect to $\Psi$ we deduce that
\begin{eqnarray}\label{sigmaaction}
\sigma(\xi_i^{j,\lambda_j - 1 - s}) = \varepsilon_{i,j,s}\xi_{j'}^{i', \lambda_i - 1 - s}
\end{eqnarray}
where $\varepsilon_{i,j,s}$ is defined by the relationship
$(e^{\lambda_j-1-s}w_j,e^sw_{j'}) = - \varepsilon_{i,j,s}(w_i,
e^{\lambda_i-1} w_{i'})$. This requires a little calculation. We now
have made the notation
\begin{eqnarray*}
\zeta_i^{j,s} = \xi_i^{j,\lambda_j-1-s} +  \varepsilon_{i,j,s}\xi_{j'}^{i', \lambda_i - 1 - s}.
\end{eqnarray*}
We make further notation
$$\varpi_{i \leq j} = \left\{ \begin{array}{ll}
         1 & \mbox{if $i \leq j$}\\
        -1 & \mbox{if $i > j$}\end{array} \right. $$
and comparing with Lemma \ref{nilpotents} we see that $\varpi_{i\leq
i'} \varpi_{i'\leq i} = \epsilon(-1)^{\lambda_i-1}$, which shall
prove useful in some later calculations. The next lemma settles the
question of which linear relations exist between the maps
$\zeta_i^{j,s}$. The proof may be found in \cite{Top}.
\begin{lem}\label{spanningdetails}
The following are true:
\begin{enumerate}
\item{$\varepsilon_{i,j,s} = (-1)^{\lambda_j - s} \varpi_{i\leq i'} \varpi_{j \leq j'}$;}

\smallskip

\item{$\varepsilon_{i,j,s} = \varepsilon_{j',i',s}$;}

\smallskip

\item{The only linear relations amongst the $\zeta_i^{j,s}$ are those of the
form $\zeta_i^{j,s} = \varepsilon_{i,j,s} \zeta_{j'}^{i',s}$.}
\end{enumerate} 
\end{lem}

Thanks to the above lemma we may refine a basis from the spanning
set of vectors $\{\zeta_i^{j,s}\}$ by removing any zero elements and
excluding precisely one of the pair $(\zeta_i^{j,s},
\zeta_{j'}^{i',s})$ when these vectors are non-zero. With this in
mind define
\begin{eqnarray*}
H&:=& \{ \zeta_i^{i,s} : i< i' , 0 \leq s < \lambda_i \} \cup
\{ \zeta_i^{i,s} : i = i', 0 \leq s < \lambda_i, \lambda_i - s \text{ even}\};\\
N_0 &:=& \{ \zeta_i^{i',s} : i \neq i', 0 \leq s < \lambda_i,  \lambda_i - s \text{ odd} \};\\
N_1 &:=& \{\zeta_i^{j,s} : i < j \neq i', 0 \leq s < \lambda_j\},
\end{eqnarray*}
and
\begin{eqnarray*}
\HH &:=& \spn (H);\\
\NN_0 &:=& \spn (N_0);\\
\NN_1 &:=& \spn (N_1).
\end{eqnarray*}
If $U_0$ and $U_1$ are subspaces of $V$ then $\End(U_0, U_1)$ shall
denote the space of all linear maps $U_0 \ra U_1$. We consider
$\End(U_0, U_1)$ to be a subspace of $\End(V)$ under the natural
embedding induced by the inclusions of $U_0$ and $U_1$ into $V$.
\begin{lem}\label{subbasis}
The set $H\sqcup N_0 \sqcup N_1$ forms a basis for $\h_e$.
Furthermore we have the following characterisation of the three
spaces:
\begin{enumerate}
\item{$\HH$ is precisely the subspace of $\h_e$ which preserves each Jordan block
space $V[i]$: $$\HH = \h_e \cap(\oplus_{i}\emph{\End}(V[i]));$$}
\item{$\NN_0$ is precisely the subspace of $\h_e$ which ``interchanges'' $V[i]$
and $V[i']$ for $i \neq i'$ and annihilates $V[i]$ for $i=i'$:
$$\NN_0 = \h_e \cap(\oplus_{i\neq i'} \emph{\End}(V[i], V[i']));$$}
\item{$\NN_1$ is the subspace of $\h_e$ which does neither of the
above: $$\NN_1 = \h_e \cap(\oplus_{i} (\oplus_{j\not\in\{i,i'\}} \emph{\End}(V[i], V[j]))).$$}
\end{enumerate}
\end{lem}
\begin{proof}
First we show that all elements of $H\sqcup N_0 \sqcup N_1$ are
non-zero. Clearly $\zeta_i^{j,s} = 0$ if and only if
$\xi_i^{j,\lambda_j - 1-s} = -\varepsilon_{i,j,s}
\xi_{j'}^{i',\lambda_i-1-s}$. For this we require that $i = j'$ and
$\varepsilon_{i,j,s} = -1$. For $i=j'$ we must have $i = i' = j$ or
$i \neq i' = j$. In the first case, $\varepsilon_{i,j,s} =
(-1)^{\lambda_j-s}$ which equals $-1$ only if $\lambda_i-s$ is odd.
But the maps $\zeta_i^{i,s}$ do not occur in $H$ when $i=i'$ and
$\lambda_i-s$ is odd. In the second case $\varepsilon_{i,j,s} =
(-1)^{\lambda_i-1-s}$ which equals $-1$ only if $\lambda_i-s$ is
even. However, the maps $\zeta_i^{i',s}$ do not occur in $N_0$ when
$i\neq i'$ and $\lambda_i-s$ is even.

Next observe that when $\zeta_i^{j,s}\neq 0$ exactly one of the two
maps $\zeta_i^{j,s}$ and $\zeta_{j'}^{i',s}$ occurs in $H\sqcup
N_0\sqcup N_1$, thus showing this set to be a basis by
Lemma~\ref{spanningdetails}(3). The three characterisations are
clear upon inspection of the definitions of the sets $H, N_0$ and
$N_1$.
\end{proof}

\subsection{Decomposing $\h_e$}\label{sub2}
It is our intention to decompose $[\h_e, \h_e]$ into subspaces. In
order to do so we must first decompose $\HH$ and $\NN_1$. Let
\begin{eqnarray*}
&\HH_0 := \spn\{\zeta_i^{i,s}\in \HH : \lambda_i - s \text{ even} \}\\
&\HH_1 := \spn\{\zeta_i^{i,s} \in\HH : \lambda_i - s \text{ odd} \}
\end{eqnarray*}
so that $\HH = \HH_0 \bigoplus \HH_1$. The space $\HH_0$ can be
further decomposed as $\bigoplus_{m=1}^{\lfloor\lambda_1/2\rfloor}
\HH_0^{m}$ where $$\HH_0^{m} := \spn\{\zeta_i^{i,\lambda_i - 2m} \in \HH: 1
\leq i \leq n\}.$$ Next we must decompose each $\HH_0^{m}$ into
subspaces $\HH_{0,j}^m$ for $j\ge 1$.

Fix $0 < m \leq \lfloor \lambda_1/2 \rfloor$, put $a_{1,m}:=1$ and
let $1=a_{1,m}< a_{2,m}<\cdots< a_{t(m), m}\le n+1$ be the set of
all integers such that $$\lambda_{a_{j,m}-1} - \lambda_{a_{j,m}}
\geq 2m, \qquad 2\le j\le t(m).$$ For $1 \leq j < t(m)$ we define
$$\HH_{0,j}^m := \spn \{\zeta_i^{i,\lambda_i - 2m}\in \HH : a_{j,m} \leq i
< a_{j+1, m}\}$$ and set $$\HH_{0,t(m)}^m := \spn
\{\zeta_i^{i,\lambda_i - 2m} \in \HH : a_{t(m),m} \leq i < n+1\}.$$
\begin{lem}
The following are true:
\begin{enumerate}\label{bits}
\item{If $\lambda_{a_{t(m),m}} < 2m$ then $\HH_{0,t(m)}^m = \{0\}$;}

\smallskip

\item{$\HH_0^m = \bigoplus_{j=1}^{t(m)} \HH_{0,j}^m$.}
\end{enumerate}
\end{lem}
\begin{proof}
If $a_{t(m),m} = n+1$ then certainly $\HH_{0,t(m)}^m = 0$, so assume
not. If $\lambda_{a_{t(m),m}} < 2m$ then the ordering $\lambda_1
\geq \cdots \geq \lambda_n$  implies that $\lambda_i - 2m < 0$ for
all $i \geq a_{t(m),m}$. Then $\zeta_i^{i,\lambda_i-2m} = 0$ for all
$\zeta_i^{i,\lambda_i - 2m} \in \HH_{0,t(m)}^m$ proving (1). The
choice of $m$ (and the fact that $a_{1,m} = 1$) ensures that
$\bigoplus_{l=1}^{t(m)} \HH_{0,j}^m = \spn\{\zeta_i^{i,\lambda_i-
2m} \in \HH : 1\leq i \leq n\} = \HH_0^m$. Hence (2).
\end{proof}

It should be noted that if $i \neq i'$ then
$\varepsilon_{i,i,\lambda_i - 2m} = 1$ by
Lemma~\ref{spanningdetails}. In this case $\zeta_i^{i,\lambda_i -
2m} = \zeta_{i'}^{i',\lambda_{i'} - 2m}$ by the same lemma. In
order to overcome this notational problem and concisely refer to a
basis for $\HH_{0,j}^m$ it shall be convenient to use an indexing
set slightly different from $\{1,...,n\}$. Extend the involution $i
\mapsto i'$ to all of $\Z$ by the rule $i = i'$ for $i> n$ or $i <
1$. We adopt the convention $\lambda_i = 0$ for all $i > n$ or $i <
1$ which immediately implies $\zeta_i^{i,s} = 0$ for any such $i$.
We shall index our maps and partitions by the set $\Z/\sim$ where
$i\sim j$ if $i = j'$. We denote by $[i]$ the class of $i$ in $\Z
/\sim$. We have $\lambda_i = \lambda_{i'}$ for all $i$ so we may
introduce the notation $\lambda_{[i]}$. As was observed a moment
ago, $\zeta_i^{i,\lambda_i - 2m} = \zeta_{i'}^{i',\lambda_{i'} -
2m}$. Hence we may also use the notation $\zeta_{[i]}^{[i],
\lambda_{[i]} - 2m}$. Furthermore, since $i' \in \{i-1, i, i+1\}$ we
have a well defined order on $\Z /\sim$ inherited from $\Z$: let
$[i] \leq [j]$ if $i \leq j$. As a result there exists a unique
isomorphism of totally ordered sets $\psi : (\Z / \sim) \rightarrow
\Z$ with $\psi([1]) = 1$. Using this isomorphism we define analogues
of addition and subtraction $+ , - \colon\, (\Z / \sim) \times \Z
\to (\Z / \sim)$ by the rules
\begin{eqnarray*}
&[i] + j := \psi^{-1}(\psi(i) + j)\\
&[i] - j := \psi^{-1}(\psi(i) - j)
\end{eqnarray*}
To clarify, $[i] + 1$ is the class in $(\Z/\sim)$ succeeding $[i]$
and $[i]-1$ is that class preceding $[i]$ in the ordering.

For $1\leq j < t(m)$, Lemma~\ref{spanningdetails}(3) yields that the
set $$\big\{ \zeta_{[i]}^{[i],\lambda_{[i]} - 2m} \in \HH:\, [a_{j,m}] \leq
[i] < [a_{j+1, m}]\big\}$$ is a basis for $\HH_{0,j}^m$. Using this
basis we may describe an important hyperplane $\HH_{0,j}^{m,+}$ of
$\HH_{0,j}^m$. First we define the augmentation map
$\HH_{0,j}^m\twoheadrightarrow \K$ by sending
$\zeta_{[i]}^{[i],\lambda_{[i]} - 2m}$ to $1$ for all  $[a_{j,m}]
\leq [i] < [a_{j+1, m}]$ and extending to $\HH_{0,j}^m$ by
$\K$-linearity. Let $\HH_{0,j}^{m,+}$ denote the kernel of this map.
It was noted in Lemma \ref{bits} that $\HH_{0,t(m)}^m$ might be
zero. If this is not the case then a basis for $\HH_{0,t(m)}^m$ is
the span of those $\zeta_{[i]}^{[i], \lambda_{[i]} - 2m}$ which are
non-zero with $[a_{t(m),m}] \leq [i] \leq [n]$. Using this basis we
can define the augmentation map $\HH_{0,t(m)}^m \twoheadrightarrow
\K$ and hyperplane $\HH_{0,t(m)}^{m,+}$ of $\HH_{0,t(m)}^m$ in a
similar fashion. Make the notation $$\HH_0^+\,:=\,
\sum_{m=1}^{\lfloor \lambda_1/2 \rfloor}
\Big(\textstyle{\bigoplus}_{j=1}^{t(m)-1}  \HH_{0,j}^{m,+}+
\HH_{0,t(m)}^m\Big)\subseteq \HH_0.$$

Before we continue we must decompose $\NN_1$ into a direct sum of
two subspaces. We shall need the following definition, first stated in
the introduction.
 \begin{defn}\label{de} Given $\lambda=(\lambda_1,\ldots,\lambda_n)\in\mathcal{P}_\epsilon(N)$ we denote by $\Delta(\lambda)$ the set of all pairs 
 $(i,i+1)$ with $1\le i<n$ such that
$i'=i$, $(i+1)'=i+1$ and
$\lambda_{i-1}\ne\lambda_i\ge
\lambda_{i+1}\ne\lambda_{i+2}$. If $(i,i+1)\in \Delta(\lambda)$ then the pair will be called {\rm a $2$-step} of
 $\lambda$. If $i>1$ and $(i,i+1)$ is a $2$-step of $\lambda$ then  $\lambda_{i-1}$ and
 $\lambda_{i+2}$ are referred to as {\rm the boundary of} $(i,i+1)$. If
 $(1,2)\in\Delta(\lambda)$ then $\lambda_3$ is referred to as {\rm the boundary of} $(1,2)$
(if $n=2$ then $\lambda_3=0$ by convention).
 \end{defn}
Here and throughout we adopt the convention that $\lambda_0 = \lambda_{n+1} = 0$.
Take note that if $(n-1,n) \in \Delta(\lambda)$ then $\lambda_{n-2}$ and $\lambda_{n+1}=0$
form the boundary of $(n-1, n)$. We define $\NN_1^-$ to be the span of the basis vectors $\zeta_i^{i+1, \lambda_{i+1}-1} \in N_1$
such that $(i, i+1) \in \Delta(\lambda)$ and we let $\NN_1^+$ be the complement to $\NN_1^-$ in $\NN_1$
which is spanned by the remaining basis vectors $\zeta_i^{j,s} \in N_1$.

\subsection{Decomposing $[\h_e, \h_e]$}
It is the intention of this section to decompose $[\h_e, \h_e]$ into
a finite collection of those subspaces of $\h_e$ defined in the
previous section. Our calculations shall be quite explicit and
depend principally upon the following.
\begin{lem}\label{product} For all indices
$i,j,s$ and $k,l,r$
$$[\zeta_i^{j,s}, \zeta_k^{l,r}] = \delta_{il} \zeta_k^{j,r+s - (\lambda_i - 1)} -
\delta_{jk} \zeta_i^{l,r+s - (\lambda_j - 1)}+ \varepsilon_{k,l,r}\big(\delta_{k,i'}
\zeta_{l'}^{j,r+s - (\lambda_i - 1)} - \delta_{j,l'} \zeta_i^{k',r+s - (\lambda_j - 1)}\big).$$
\end{lem}
The proof is a short calculation which we leave to the reader. The
following proposition shall be central in the process of decomposing
$[\h_e, \h_e]$.

\begin{prop}\label{subs}
The following inclusions hold
\begin{eqnarray*}
[\HH,\HH] = \{0\} ,& [\HH, \NN_0]\subseteq \NN_0, & [\HH, \NN_1] \subseteq \NN_1, \\
& [ \NN_0, \NN_0 ] \subseteq \HH, & [ \NN_0, \NN_1 ] \subseteq \NN_1
\end{eqnarray*}
Furthermore, for any two elements $\zeta_i^{j,s}, \zeta_k^{l,r} \in
N_1$ the commutator $[\zeta_i^{j,s}, \zeta_k^{l,r}]$ lies in either
$\HH$, $\NN_0$ or $\NN^+_1$. More precisely $$[\zeta_i^{j,s},
\zeta_k^{l,r}] \in\, \left\{ \begin{array}{llll}
         \NN^+_1 & \ \mbox{ if $i=l$ or $j=k$;}\\
         \NN_0 \text{ or } \NN_1^+& \ \mbox{ if $k = i'$ or $j = l'$ but not both};\\
         \HH & \ \mbox{ if $k = i'$ and $j = l'$;}\\
         0 & \ \mbox{ otherwise.}\end{array} \right.$$
\end{prop}
\begin{proof}
We shall call on the characterisations of $\HH, \NN_0$ and $\NN_1$
given in Lemma \ref{subbasis}. Thanks to \cite[Theorem~1]{Ya06} we
have $\HH = \h \cap (\g_e)_\alpha$ where $\alpha$ is a certain regular
element of $\g_e^\ast$. By \cite[1.11.7]{Di} the stabaliser
$(\g_e)_\alpha$ is abelian, hence $[\HH, \HH] = 0$. The elements of
$\NN_0$ are characterised by the fact that they exchange the spaces
$V[i]$ and $V[i']$ with $i \neq i'$. Therefore the elements of
$[\HH, \NN_0]$ must exchange them also, implying $[\HH, \NN_0]
\subseteq \NN_0$. Each $\zeta_i^{j,s} \in \NN_1$ transports $V[i]$
to $V[j]$ and $V[j']$ to $V[i']$. Therefore $[\HH, \zeta_i^{j,s}]$
does likewise and $[\HH, \NN_1] \subseteq \NN_1$. Since each element
of $\NN_0$ exchanges the spaces $V[i]$ and $V[i']$ with $i\neq i'$
and annihilates all $V[i]$ with $i=i'$ the commutator space $[\NN_0,
\NN_0]$ must stabalise all $V[i]$, hence be contained in $\HH$. The
inclusion $[\NN_0, \NN_1] \subseteq \NN_1$ is best checked using
Lemma~\ref{product}. Let $i \neq i'$ and $l > k \neq l'$. Then
$[\zeta_i^{i',s} \zeta_k^{l,r}]$ is nonzero only if $i = l$ or $i' =
k$. Our restrictions on $i,l$ and $k$ ensure that these two
possibilities are mutually exclusive. In the first case
$$[\zeta_i^{i',s}, \zeta_k^{l,r}] = \zeta_k^{l', r + s -
(\lambda_i-1)} - \varepsilon_{k,l,r} \zeta_l^{k', r+s-
(\lambda_i-1)}$$ which lies in $\NN_1$. The second case is very
similar.

We now consider the final claim. Suppose $j > i\neq j'$ and $l > k
\neq l'$. By Lemma~\ref{product} the bracket $[\zeta_i^{j,s},
\zeta_k^{l,r}]$ is only nonzero when one or more of the following
equalities hold: $i=l, j=k, i'=k, j'=l$. We shall consider these
four possibilities one by one. Since the bracket is anticommutative
the reasoning in the case $i=l$ is identical to the case $j=k$ and
so we need to consider only the first of these two possibilities. If
$i=l$ then the relations $i' \neq j > i$ and $l > k \neq l'$ ensure
that $j\neq k, i' \neq k$ and $j' \neq l$. Therefore
$[\zeta_i^{j,s}, \zeta_k^{l,r}] = \zeta_k^{j, r+s - (\lambda_i-1)}
\in \NN_1$. In order for this map to lie in $\NN_1^-$ we would
require $j = k+1$, however we have $j > i = l > k$ which makes this
impossible. Thus $[\zeta_i^{j,s}, \zeta_k^{l,r}] \in \NN_1^+$.

By Lemma \ref{spanningdetails} we have $\zeta_i^{j,s} = \pm
\zeta_{j'}^{i',s}$ and $\zeta_k^{l,r} = \pm\zeta_{l'}^{k',r}$ so the
reasoning in case $i=k'$ is identical to the case $j'=l$. Therefore
we need only to consider the first of these two possibilities.
Suppose $i=k'$. Then certainly $i \neq l$ and $j \neq k$. If $j' =
l$ then $$[\zeta_i^{j,s}, \zeta_k^{l,r}] = \varepsilon_{k,l,r}
(\zeta_j^{j,r+s-(\lambda_i - 1)} - \zeta_i^{i,r+s - (\lambda_j-1)})
\in \HH,$$ so assume from henceforth that $j' \neq l$. Then
$$[\zeta_i^{j,s}, \zeta_k^{l,r}] = \varepsilon_{k,l,r}
\zeta_{l'}^{j,r+s - (\lambda_i-1)}.$$ If $j = l$ then the product
lies in $\NN_0$. Assume $j\neq l$. Thanks to the relation
$\zeta_{l'}^{j,r+s - (\lambda_i-1)} = \pm \zeta_{j'}^{l,r+s -
(\lambda_i-1)}$ from Lemma \ref{spanningdetails} we may assume that
$j>l'$, and from here it is easily seen that the product lies in
$\NN_1$. In order for the product to lie in $\NN_1^-$ we require
$\lambda_{l' - 1} \neq \lambda_{l'}$ which implies $\lambda_l < \lambda_i$ since $i = k' <
l$. From the bounds $0 \leq r < \lambda_l$ and $0 \leq s <
\lambda_j$ we deduce that $r + s - (\lambda_i - 1) < \lambda_j - 1$
which confirms that the term $\zeta_{l'}^{j,r+s - (\lambda_i-1)}$
does not lie in $\NN_1^-$.
\end{proof}

\begin{prop}\label{Nprop} The following are true:
\begin{enumerate}
\item{$\NN_0 \subset [\h_e, \h_e]$;}

\smallskip

\item{$\NN_1 \cap [\h_e, \h_e] = \NN_1^+$.}
\end{enumerate}
\end{prop}
\begin{proof}
Assume $i \neq i'$ and $\lambda_i - s$ is odd. We
have $\varepsilon_{i,i,s} = (-1)^{\lambda_i-s}$ so
$$[\zeta_i^{i',s},\zeta_i^{i,\lambda_{i}-1}] = \zeta_i^{i',s} -
\varepsilon_{i,i,\lambda_i-1}\zeta_i^{i',s} = 2\zeta_i^{i',s} \in
[\h_e, \h_e].$$ Since char$(\K) \neq 2$ we get $\NN_0 = [\HH, \NN_0]
\subseteq [\h_e, \h_e]$. This proves Part 1.

For the sake of clarity we shall divide the proof of Part 2 of the
current proposition into subsections \emph{(i), (ii),..., (ix).} In
Parts \emph{(i),...,(v)} we demonstrate that $\NN_1^+ \subseteq
[\h_e, \h_e]$ by showing that if $\zeta_i^{j,s} \in N_1$ is amongst
the basis vectors spanning $\NN_1^+$ then some
multiple of $\zeta_i^{j,s}$ may be found as a product of two basis
elements in $\h_e$. Recall that these vectors are defined
to be those for which either $(i, i+1) \notin \Delta(\lambda)$, or for which
$j \neq i+1$, or for which $s < \lambda_{j}-1$.
In Parts \emph{(vi),...,(viii)} we show that the
reverse inclusion holds by noting that $\NN_1\cap [\h_e,\h_e]$ is
the sum of those products $[\zeta_i^{j,s}, \zeta_k^{l,r}]$ which lie
in $\NN_1$, and showing that all such products actually lie in
$\NN_1^+$. For the remainder of the proof we shall fix $l > k \neq
l'$ so that $\min(\lambda_k , \lambda_l) = \lambda_l$.

\emph{(i) If $l \neq l'$ or $k \neq k'$, then $\zeta_k^{l,r} \in
[\h_e, \h_e]$ for $r = 0,1,..., \lambda_l-1$:} Suppose first that $l
\neq l'$. We have
$$[\zeta_l^{l,\lambda_l-1}, \zeta_k^{l,r}] = \zeta_k^{l, r} \in [\h_e, \h_e]$$
whence we obtain $\zeta_k^{l,r} \in [\h_e, \h_e]$ for $r = 0,1,..., \lambda_l-1$.
Now suppose $k \neq k'$. Then
$$[\zeta_k^{k,\lambda_k-1}, \zeta_k^{l,r}] = - \zeta_k^{l, r} \in [\h_e, \h_e]$$
so that $\zeta_k^{l,r} \in [\h_e, \h_e]$ for all $r = 0,1,..., \lambda_l-1$

\emph{(ii) If $l'= l$ and $k = k'$ then $\zeta_k^{l,r} \in [\h_e,
\h_e]$ for $r=0,1,...,\lambda_l -2$:} With $l$ and $k$ as above
$$[\zeta_l^{l,\lambda_{l}-2}, \zeta_k^{l,r}] = \zeta_k^{l,r -1} -
\varepsilon_{k,l,r} \zeta_{l'}^{k', r-1} \in [\h_e, \h_e].$$ By Part
3 of Lemma \ref{spanningdetails} this final expression is $(1 -
\varepsilon_{k,l,r} \varepsilon_{k,l,r-1}) \zeta_k^{l,r-1}$. Since
$\varepsilon_{k,l,r} = (-1)^{\lambda_l-r}$ this expression actually
equals $2\zeta_k^{l,r-1}$. Allowing $r$ to run from $0$ to
$\lambda_l-1$ we obtain the desired result.

\emph{(iii) If $l' = l$, $k = k'$ and $k \neq j-1$ then
$\zeta_k^{l,r} \in [\h_e, \h_e]$ for $r=0,1,...,\lambda_l-1$:} We
may assume there exists $j$ fulfilling $l > j > k$. Then $k \neq l$
and $k' \neq j \neq l'$ so that $$[\zeta_j^{l,r},
\zeta_k^{j,\lambda_j-1}] =  \zeta_k^{l,r}\in [\h_e \h_e].$$

\emph{(iv) If $l = l'$, $k = k'$ and either $\lambda_k = \lambda_{k-1}$ or $\lambda_l
= \lambda_{l+1}$, then $\zeta_k^{l,r} \in [\h_e, \h_e]$ for $r = 0, 1,...,
\lambda_l - 1$:} First suppose that $\lambda_k = \lambda_{k-1}$. Since $k = k'$ we
have $k-1 = (k-1)'$ so that
$$[\zeta_{k-1}^{l,r}, \zeta_{k-1}^{k,\lambda_k-1}] =
\varepsilon_{k-1, k, \lambda_k -1} \zeta_{k}^{l,r}\in [\h_e,
\h_e]$$ for $r = 0, 1,..., \lambda_l - 1$.

Next suppose that $\lambda_l = \lambda_{l+1}$. Since $l=l'$ we have $l+1 = (l+1)'$
and so
$$[\zeta_{k}^{l+1,\lambda_{l+1}-1} \zeta_{l}^{l+1,r}] =
-\varepsilon_{l, l+1,r} \zeta_k^{l, r}\in [\h_e, \h_e]$$ for $r =
0, 1,..., \lambda_l - 1$.

\emph{(v) $\NN_1^+ \subseteq [\h_e, \h_e]$: }
This follows by  combining the deductions of Parts \emph{(i) - (iv)}.

\emph{(vi) $[\HH, \NN_1] \subseteq \NN_1^+$:} We continue to fix $l
> k \neq l'$. The bracket $[\zeta_i^{i,s}, \zeta_k^{l,r}]$ is nonzero
only if $i = k$ or $i = l$. Assume $i = l$ (the case $i = k$ is
similar). Then $[\zeta_i^{i,s}, \zeta_k^{l,r}] =
\zeta_k^{l,r+s-(\lambda_i - 1)}$ which lies either in $\NN_1^-$ or
$\NN_1^+$. In order for $\zeta_k^{l,r+s-(\lambda_i - 1)}$ to lie in
$\NN^-$ we must have $l=l'$. But in that case $i = i'$ and so
$\lambda_i - s$ must be even by the definition of $\HH$. In
particular, $s \leq \lambda_i - 2$ and $r + s - (\lambda_i-1) \leq r
- 1 < \lambda_l - 1$ forcing $[\zeta_i^{i,s}, \zeta_k^{l,r}] \in
\NN_1^+$.

\emph{(vii) $[\NN_0, \NN_1] \subseteq \NN_1^+$:} The product
$[\zeta_i^{i',s}, \zeta_k^{l,r}]$ with $i\neq i'$ is nonzero only if
$i = l$ or $i' = k$. Our restrictions on $i,l$ and $k$ ensure that
these two possibilities are mutually exclusive. In the first case
$$[\zeta_i^{i',s}, \zeta_k^{l,r}] = \zeta_k^{l', r + s -
(\lambda_i-1)} - \varepsilon_{k,l,r} \zeta_l^{k', r+s-
(\lambda_i-1)} = (1 - \varepsilon_{k,l,r}
\varepsilon_{k,l',r+s-(\lambda_i-1)})
\zeta_k^{l',r+s-(\lambda_i-1)}.$$ If
$\zeta_k^{l',r+s-(\lambda_i-1)}\in\NN_1^-$ then $l = l'$ by the
definition of $\NN_1^-$. But then $i = l=l'$ yields $i = i'$
contrary to our assumptions. We deduce that
$\zeta_k^{l',r+s-(\lambda_i-1)}\in\NN_1^+.$

Now consider the case $i' = k$. A calculation similar to the above
gives $$[\zeta_i^{i',s}, \zeta_k^{l,r}] = (\varepsilon_{k,l,r}
\varepsilon_{k,l,r+s - (\lambda_i - 1)} -
1)\zeta_i^{l,r+s-(\lambda_i-1)}.$$ Since $i\ne i'$ we see as before
that the right hand side lies in $\NN_1^+$, hence \emph{(vii)}.

\emph{(viii) $\NN_1 \cap [\NN_1, \NN_1] \subseteq \NN_1^+$:} This
follows immediately from the last statement of Proposition
\ref{subs}.

\emph{(ix) $\NN_1\cap [\h_e, \h_e] = \NN_1^+$:} By \emph{(v)} we
know that $\NN_1^+ \subseteq \NN_1\cap [\h_e, \h_e] $. By
Proposition \ref{subs}, $\NN_1 \cap [\h_e, \h_e]$ is equal to the
span of those products $[\zeta_i^{j,s}, \zeta_k^{l,r}]$ which lie in
$\NN_1$. By that same proposition and Parts \emph{(vi) - (viii)} we
see that every product $[\zeta_i^{j,s}, \zeta_k^{l,r}]$ which lies
in $\NN_1$ actually lies in $\NN_1^+$. The claim follows.
\end{proof}

\begin{prop}\label{Hprop} The following are true:
\begin{enumerate}
\item{$\HH_1 \subset [\h_e, \h_e]$;}

\smallskip

\item{$\HH_0 \cap [\h_e, \h_e] = \HH_0^+$.}
\end{enumerate}
\end{prop}
\begin{proof}
$\HH_1$ has a basis consisting of vectors $\zeta_i^{i,s}$ with $i <
i'$ and $\lambda_i - s$ odd. Fix such a choice of $i$ and $s$, and
choose $r$ such that $\lambda_i - r$ is odd. By Lemma \ref{product},
we have that $$[\zeta_{i'}^{i, s}, \zeta_{i}^{i',r}] = (1 +
\varepsilon_{i,i',r}) (\zeta_i^{i,s+r-(\lambda_i-1)} -
\zeta_{i'}^{i',s+r-(\lambda_i - 1)}).$$ Since
$\varepsilon_{i,i,r+s-(\lambda_i-1)} =
(-1)^{\lambda_i-(s+r-(\lambda_i-1))} = (-1)^{(\lambda_i - s) +
(\lambda_i-r) + 1} = -1$ it follows that
$\zeta_{i'}^{i',s+r-(\lambda_i-1)} = -
\zeta_i^{i,s+r-(\lambda_i-1)}$ by Part 3 of Lemma
\ref{spanningdetails}. Also $\varepsilon_{i, i', r} =
(-1)^{\lambda_i - r + 1} = 1$. Therefore
\begin{eqnarray*}
[\zeta_{i'}^{i, s}, \zeta_{i}^{i',r}] = 4 \zeta_i^{i,s+r-(\lambda_i-1)}
\end{eqnarray*}
which is nonzero since char$(\K)\neq 2$. We make the observation
that the above expression lies in $\HH_1$ for any choice of $r$ and
$s$ with $\lambda_i -r$ and $\lambda_i - s$ both odd. Taking $r =
\lambda_i -1$ we obtain $\zeta_i^{i,s} \in \HH \cap [\h_e, \h_e]$.
Since $\HH_1$ is spanned by those $\zeta_i^{i,s}$ such that $i < i'$
and $\lambda_i - s$ is odd we have $\HH_1 \subseteq [\h_e, \h_e]$.
This completes (1).

\smallskip

For the sake of clarity we shall divide the proof of Part 2 of the
current proposition into subsections \emph{(i), (ii), ..., (vii)}.
The approach is much the same as Part 2 of Proposition \ref{Nprop}.
In Parts \emph{(i),...,(iv)} we show that a spanning set for
$\HH_0^+$ may be found in $[\h_e, \h_e]$ and in the subsequent Parts
\emph{(v), (vi), (vii)} we demonstrate that any product
$[\zeta_i^{j,s}, \zeta_k^{l,r}]$ which lies in $\HH_0$ actually lies
in $\HH_0^+$.

\emph{(i) The subspace $\HH_0 \cap [\NN_1,\NN_1]$ is spanned by all
$\zeta_{[j]}^{[j], \lambda_{[j]} - 2m} - \zeta_{[i]}^{[i],
\lambda_{[i]} - 2m} $ such that $[1] \leq [i] < [j] \leq [n]$ and
$\lambda_i - \lambda_j < 2m < \lambda_j + \lambda_i$:} Indeed by
Proposition \ref{subs} we see that $\HH \cap [\NN_1, \NN_1]$ is
spanned by commutators $[\zeta_i^{j,s}, \zeta_{i'}^{j',r}]$ with
$[j] > [i]$. In turn $$[\zeta_i^{j,s}, \zeta_{i'}^{j',r}] =
\varepsilon_{i', j', r} [\zeta_i^{j,s}, \zeta_{j}^{i,r}] =
\varepsilon_{i', j', r}(\zeta_j^{j,r+s-(\lambda_i-1)} -
\zeta_i^{i,r+s-(\lambda_j-1)}).$$ The reader will notice that
$$[\zeta_i^{j,s}, \zeta_{j}^{i,r}] \in \left\{ \begin{array}{ll}
        \HH_1 & \mbox{if $\lambda_i + \lambda_j - (r+s) -1$ odd};\\
        \HH_0 & \mbox{if $\lambda_i + \lambda_j - (r+s)-1$ even}.\end{array} \right.$$ As a
        consequence $\HH_0 \cap [\NN_1, \NN_1]$  is spanned all
        $[\zeta_i^{j,s}, \zeta_{j}^{i,r}]$ with $[1] \leq [i] <  [j] \leq [n]$ and
        $0 \leq s,r < \lambda_i, \lambda_i + \lambda_j - (r+s) -1\text{ even}.$ If we
        pick $[1] \leq [i] <  [j] \leq [n]$ and $0 \leq s,r < \lambda_i$ such that
        $\lambda_i + \lambda_j - (r+s) -1 = 2m$ then we have $$[\zeta_i^{j,s},
        \zeta_{j}^{i,r}] = \varepsilon_{i',j',r} \big(\zeta_{[j]}^{[j], \lambda_{[j]} - 2m} -
        \zeta_{[i]}^{[i], \lambda_{[i]} - 2m}\big).$$ The constraints placed on $s$ and $r$
        are equivalent to $\lambda_i - \lambda_j < 2m < \lambda_i + \lambda_j$, and \emph{(i)} follows.

\emph{(ii) $\HH_0 \cap [\h_e, \h_e] = \HH_0 \cap [\NN_1, \NN_1]$:}
By Proposition \ref{subs} we see that $$\HH \cap [\h_e, \h_e] =
[\NN_0, \NN_0] + (\HH \cap [\NN_1, \NN_1])$$ whereas our observation
in Part 1 of the current proposition shows that $[\NN_0, \NN_0]
\subseteq \HH_1$. Since $\HH = \HH_0 \oplus \HH_1$ and $$\HH \cap
[\NN_1, \NN_1]=(\HH_0 \cap [\NN_1, \NN_1])\textstyle{\bigoplus}
(\HH_1 \cap [\NN_1, \NN_1])$$ by our discussion in $(i)$ we obtain
$\HH_0 \cap [\h_e, \h_e] = \HH_0 \cap [\NN_1, \NN_1]$.

\emph{(iii) Each spanning vector from (i) lies in a unique
$\HH_{0,l}^m$, in particular we have that $\HH_0 \cap [\NN_1, \NN_1]
= \bigoplus_{l,m} (\HH_{0,l}^m \cap [\NN_1, \NN_1])$:} Fix $m$ in
the appropriate range and suppose $1\leq i < a_{t(m),m}$ We claim
that if $[j] > [i]$ then each $\zeta_{[j]}^{[j], \lambda_{[j]} - 2m}
- \zeta_{[i]}^{[i], \lambda_{[i]} - 2m} \in \HH_0 \cap [\NN_1,
\NN_1]$ lies in $\HH_{0,l}^m$ where $l$ is the unique integer
fulfilling $[a_{l, m}] \leq [i] < [a_{l+1,m}]$. It will suffice to
show that given $i,j,l$ and $m$ as above we have $[j] <
[a_{l+1,m}]$. To see this, suppose that $[j] \geq [a_{l+1,m}]$. Then
by our choice of $a_{l+1,m}$ we have $\lambda_{a_{l+1,m}-1} -
\lambda_{a_{l+1,m}} \geq 2m$ which implies $\lambda_i - \lambda_j
\geq 2m$ contrary to the restriction $\lambda_i - \lambda_j < 2m$
noted in the statement of \emph{(i)}. We conclude that $[a_{l,m}]
\leq [i] < [j] < [a_{l+1,m}]$ and that the corresponding spanning
vector lies in $\HH_{0,l}^m$. In case $a_{t(m),m} \leq i$ we have
$\zeta_{[j]}^{[j], \lambda_{[j]} - 2m} - \zeta_{[i]}^{[i],
\lambda_{[i]} - 2m} \in \HH_{0,t(m)}^m$ by definition. Thus we have
shown that the spanning vectors  of $\HH_0 \cap  [\NN_1, \NN_1]$
each lie in some $\HH_{0,l}^m$, as claimed.

\emph{(iv) The inclusion $\HH_{0,l}^{m,+} \subseteq \HH_{0,l}^m \cap
[\NN_1, \NN_1]$ holds for all $l$ and $m$:} Suppose $1
\leq i < a_{t(m),m}$. Since $\lambda_{a_{t(m),m} -1} -
\lambda_{a_{t(m),m}} \geq 2m$ we know that $\lambda_{a_{t(m),m}-1}
\geq 2m$ and so $\lambda_i \geq 2m$. It follows that
$\zeta_{[i]}^{[i], \lambda_{[i]} - 2m} \neq 0$ for all such $i$. Fix
$[i]$ with $[a_{l,m}] < [i] < [a_{l+1,m}]$. By our choice of
integers $\{a_{1,m},...,a_{t(m),m}\}$ we know that $\lambda_{[i]-1}
- \lambda_{[i]} < 2m$ and since $\lambda_{[i]-1}, \lambda_{[i]} \geq
\lambda_{a_{t(m),m}} \geq 2m$ we have $\lambda_{[i]-1} +
\lambda_{[i]} > 2m$. By these remarks, using \emph{(i)}, it follows
that $\zeta_{[i]-1}^{[i]-1, \lambda_{[i]-1} - 2m} -
\zeta_{[i]}^{[i], \lambda_{[i]} - 2m}$ is a nonzero element of
$\HH_{0,l}^m \cap [\NN_1, \NN_1]$. These vectors span all of
$\HH_{0,l}^{m,+}$ so \emph{(iv)} follows for $l < t(m)$.

The argument for $l = t(m)$ is similar. Let $k = \max\{ i: \lambda_i
\geq 2m\}$. Then $\zeta_i^{i,\lambda_i - 2m} \neq 0 $ if and only if
$i \leq k$ so $\HH_{0,t(m)}^m =
\spn\{\zeta_{[i]}^{[i],\lambda_{[i]}-2m} : [a_{t(m),m}] \leq [i]
\leq [k]\}$. Fix  $[i]$ with $[a_{t(m),m}] < [i] \leq [k]$. By our
choice of integers $\{a_{1,m},...,a_{t(m),m}\}$ we know that
$\lambda_{[i]-1} - \lambda_{[i]} < 2m$ and by our choice of $k$ we
have $\lambda_{[i]-1} + \lambda_{[i]} > 2m$. The argument now
concludes exactly as above.

\emph{(v) The equality $\HH_{0,l}^m \cap [\NN_1, \NN_1] =
\HH_{0,l}^{m,+}$ holds for all $1\leq l < t(m)$:} The discussion in
\emph{(iii)} confirms that $\HH_{0,l}^m \cap [\NN_1, \NN_1]$ is
spanned by all $\zeta_{[j]}^{[j], \lambda_{[j]} - 2m} -
\zeta_{[i]}^{[i], \lambda_{[i]} - 2m}$ with $[a_{l,m}] \leq [i] <
[j] < [a_{l+1,m}], \lambda_i - \lambda_j < 2m < \lambda_j +
\lambda_i$. This space is clearly contained in $\HH_{0,l}^{m,+}$.
Now \emph{(v)} follows from \emph{(iv)}.

\emph{(vi) $\HH_{0,t(m)}^m \cap [\NN_1, \NN_1] = \HH_{0,t(m)}^m$:}
First we note that $\HH_{0,t(m)}^{m,+} \subseteq \HH_{0,t(m)}^m \cap
[\NN_1, \NN_1]$ by \emph{(iv)}. If $\lambda_{a_{t(m),m}} < 2m$ then
$\HH_{0,t(m)}^m = 0$ by Part 1 of Lemma \ref{bits} and the statement
holds trivially. So assume $\lambda_{a_{t(m),m}} \geq 2m$ and let $k
= \max\{ i : \lambda_i \geq 2m\}$. Then $\HH_{0,t(m)}^m$ is spanned
by all $\zeta_{[i]}^{[i],\lambda_{[i]}-2m}$ with  $[a_{t(m),m}] \leq
[i] \leq [k]\}$. We claim that $[k]+1\leq [n]$. If not then $[k] =
[n]$ which implies that $ \lambda_{k} - \lambda_{k+1} = \lambda_k
\geq 2m$ forcing $k+1 \in \{a_{1,m},...,a_{t(m),m}\}$. However $k+1
> a_{t(m),m}$ and $a_{1,m} \leq \cdots \leq a_{t(m),m}$. This
contradiction confirms the claim. By the very same reasoning we know
that $\lambda_{[k]} - \lambda_{[k]+1} = \lambda_k - \lambda_{k+1} <
2m$ and the inequality $[k]+1\leq n$ gives us $\lambda_{[k]+1} > 0$
which in turn implies $\lambda_{[k]} + \lambda_{[k]+1} > 2m$. By
\emph{(i)} and \emph{(iii)} we have $\zeta_{[k]+1}^{[k]+1,
\lambda_{[k]+1}-2m} - \zeta_{[k]}^{[k],\lambda_{[k]} - 2m} \in
\HH_{0,t(m)}^m$. Since $\lambda_{k+1} < 2m$ we know that
$\zeta_{[k]+1}^{[k]+1, \lambda_{[k]+1}-2m} = 0$. Since
$\zeta_{[k]}^{[k],\lambda_{[k]} - 2m} \notin \HH_{0,t(m)}^{m,+}$ and
$\HH_{0,t(m)}^{m,+}$ has codimension 1 in $\HH_{0,t(m)}^m$ statement
\emph{(vi)} follows.

\emph{(vii) $\HH_0 \cap [\h_e, \h_e] = \HH_0^+$:} By \emph{(ii)} and
\emph{(iii)} we have $$\HH_0 \cap [\h_e, \h_e] =
\textstyle{\bigoplus}_{l,m} (\HH_{0,l}^m \cap [\NN_1, \NN_1]).$$ The
proposition now follows from \emph{(v)} and \emph{(vi)}.
\end{proof}

\begin{thm}\label{derived}
The derived subalgebra $[\h_e, \h_e]$ coincides with $\NN_0
\bigoplus \NN_1^+ \bigoplus \HH_0^+\bigoplus \HH_1.$
\end{thm}
\begin{proof}
The sum of the above subspaces is direct by construction. By
Proposition \ref{subs} we know that $[\h_e, \h_e]$ is the sum of the
three spaces $$[\h_e, \h_e] = (\NN_0 \cap [\h_e, \h_e]) + (\NN_1
\cap [\h_e, \h_e]) + (\HH \cap [\h_e, \h_e]).$$ By Proposition
\ref{Nprop} we have that $(\NN_0 \cap [\h_e, \h_e]) + (\NN_1 \cap
[\h_e, \h_e]) = \NN_0 + \NN_1^+$. By Proposition \ref{Hprop} using
the fact that $\HH = \HH_0 \oplus \HH_1$ we have $\HH \cap [\h_e,
\h_e] = \HH_1 + \HH_0^+$. The theorem follows.
\end{proof}

 \subsection{A combinatorial formula for $\dim \h_e^{\text{ab}}$} 

 As a corollary to the previous theorem we obtain an expression for the dimension of
 the maximal abelian quotient $\h_e^\text{ab} := \h_e/[\h_e, \h_e]$. Given a
 partition $\lambda=(\lambda_1,\ldots,\lambda_n)\in\mathcal{P}_\epsilon(N)$ we have 
 defined $\Delta(\lambda)$ to be the set of pairs $(i, i+1)$ with $1\leq i < n,  i'=i, (i+1)'=i+1$ and $\lambda_{i-1}
 \neq \lambda_i \geq \lambda_{i+1} \neq \lambda_{i+2}$; see Definition~1. Recall that the elements of $\Delta(\lambda)$ are referred to as {\it 2-steps}. Now set
$$s(\lambda) := \textstyle{\sum}_{i=1}^{n} \lfloor(\lambda_{i} - \lambda_{i+1})/2
\rfloor.$$
Note that if $(i,i+1)\in\Delta(\lambda)$ then $\epsilon(-1)^{\lambda_i} =
\epsilon(-1)^{\lambda_{i+1}}  = -1$ and recall our convention that $\lambda_0=0$ and $\lambda_i=0$ for all $i>n$.
We may now state and prove the formula for $\dim \,\h_e^\text{ab}$.
\begin{cor}\label{expression} Let $\h$ be one of the classical Lie algebras $\mathfrak{so}_N$
or $\mathfrak{sp}_N$ where $N\ge 2$ and suppose that char$(\K)\ne 2$. Then
 $\dim \h_e^{\emph{\text{ab}}} = s(\lambda) + |\Delta(\lambda)|$ for any
 nilpotent element $e=e(\lambda)\in\h$.
 \end{cor}
 \begin{proof}
 Recall that $\h_e = \HH \bigoplus \NN_0 \bigoplus \NN_1$, that $\NN_1 =
 \NN_1^- \bigoplus \NN_1^+$, and that $\HH = \HH_0 \bigoplus \HH_1$ with
 $\HH_0^+ \subseteq \HH_0$. By Theorem \ref{derived} we have that
 $\h_e^\text{ab} \cong (\NN_1/\NN_1^+) \bigoplus (\HH_0/\HH_0^+)$ as
 vector spaces. We claim that $\dim(\NN_1/\NN_1^+) = |\Delta(\lambda)|$ and
 that $\dim(\HH_0/\HH_0^+) = s(\lambda)$, from whence the theorem shall follow.
 First of all observe that $\dim(\NN_1/\NN_1^+) = \dim(\NN_1^-)$. By Part 3 of
 Lemma~\ref{spanningdetails} the maps $\zeta_{i-1}^{i,\lambda_i-1}$ spanning
 $\NN_1^-$ are all linearly independent.
Out last remark in Subsection~\ref{sub2} defines the set $\NN_1^{-}$ to be the space spanned
by $$N_1^{-}:=\{\zeta_{i}^{i+1,\lambda_{i+1} - 1} : (i,i+1) \in \Delta(\lambda)\}.$$
The map $(i,i+1) \mapsto \zeta_i^{i+1, \lambda_{i+1}-1}$ is clearly a bijection $\Delta(\lambda) \leftrightarrow N_1^-$.
We conclude that $\dim(\NN_1/\NN_1^+) = \dim(\NN_1^-) = |\Delta(\lambda)|$.

We must now show that $\dim(\HH_0/ \HH_0^+) = s(\lambda)$. Observe
that $\HH_0 = \bigoplus_{l,m} \HH_{0,l}^m$ (Part 2 of Lemma
\ref{bits}) and that each $\HH_{0,l}^{m,+}$ has codimension $1$ in
$\HH_{0,l}^m$. Furthermore if $l < t(m)$ then $\HH_{0,l}^m \neq 0$.
We conclude that $\dim(\HH_0/ \HH_0^+) = |\mathcal{D}|$ where
$$\mathcal{D} = \{(l,m) : 1\leq l \leq t(m)-1,\, 1\leq m\leq \lfloor
\lambda_1/2 \rfloor\} .$$

On the other hand, $s(\lambda) = |\mathcal{D}'|$ where
$$\mathcal{D}' = \{(i,m) \in \{2,...,n+1\}\times \{1,...,\lfloor
\lambda_1/2 \rfloor\}: \lambda_{i-1} - \lambda_{i} \geq 2m \}.$$ If
we construct a bijection from $\mathcal{D}$ to $\mathcal{D}'$ then
the result follows. Define a map from $\mathcal{D}$ to
$\{2,...,n+1\}\times \{1,...,\lfloor \lambda_1/2 \rfloor\}$ by the
rule $$(i,m) \longmapsto (a_{i+1,m}, m).$$ By the definition of the
integers $\{a_{1,m}, a_{2,m}, ..., a_{t(m),m}\}$ it is a well
defined injection into $\mathcal{D}'$. Fix $1\leq m \leq \lfloor
\lambda_1/2 \rfloor$. Since $\lambda_0 = 0$ and $\lambda_1 \geq
\cdots \geq \lambda_n$, we have $a_{1,m} = 1$ and
$\{a_{2,m},...,a_{t(m),m}\}$ is the set of all integers $i$ with
$2\le i\le n+1$ and $\lambda_{i-1} - \lambda_{i} \geq 2m$. Thus the
map is surjective and $\dim(\HH_0/ \HH_0^+) = s(\lambda)$.
\end{proof}
\begin{rem}\label{R1}
{\rm If $\g=\mathfrak{sl}_N$ where $N\ge 2$ and $e$ is a nilpotent
element of $\g$ corresponding to a partition
$(\lambda_1,\ldots,\lambda_n)$ of $N$ then $\dim \g_e^{\rm ab}=\dim
\z(\g_e)=\lambda_1-1$. This follows, for instance, from results of
\cite{Ya10}. If $e$ is a nilpotent element in a classical Lie
algebra $\h$ of type other than $\sf A$ then it may happen that
$\h_e^{\rm ab}$ and  $\z(\h_e)$ have different dimensions.}
\end{rem}
\noindent{\bf Example~1.} To illustrate Corollary~\ref{expression}
we consider the special case where $\h=\mathfrak{so}_4$. This Lie
algebra has type ${\sf D_2}\cong {\sf A_1}\times {\sf A_1}$ and is
isomorphic to a direct sum of two copies of $\mathfrak{sl}_2$.
Therefore $\h$ has three nonzero nilpotent orbits: the orbits
containing root vectors $e_1$ and $e_2$ of the two simple ideals of
$\h$ and the regular nilpotent orbit containing $e_1+e_2$. It is
immediate that $\h_{e_1}\cong \h_{e_2}\cong \mathfrak{sl}_2\oplus\K$
whilst $\h_{e_1+e_2}$ is abelian and has dimension $2$. In
particular, $\dim \h_{e_1}^{\rm ab}=\dim \h_{e_2}^{\rm ab}=1$ and
$\dim \h_{e_1+e_2}^{\rm ab}=2$.

On the combinatorial side, the set $\mathcal{P}_1(4)$ contains only
two nontrivial partitions, namely, $\lambda=(3,1)$ and $\mu=(2,2)$.
Since $\h$ is of type $\rm D$ and the partition $(2,2)$ has even
parts only, there are two nilpotent orbits in $\h$ attached to it
(they are permuted by an outer automorphism of $\h$ and assigned the
Roman numerals $I$ and $I\!I$). It is straightforward to see that
our root vectors $e_1$ and $e_2$ correspond to the partition $\mu$
whereas $e_1+e_2$ is attached to $\lambda$. Since $(1,2)$ is the
only $2$-step of $\lambda$ we get $|\Delta(\lambda)|=1$ and
$s(\lambda)=\lfloor(3-1)/2\rfloor=1$. So $\dim \h_e^{\rm ab}=1+1=2$
by Corollary \ref{expression}. On the other hand,
$\Delta(\mu)=\emptyset$ and
$s(\mu)=\lfloor(2-2)/2\rfloor+\lfloor(2-0)/2\rfloor=1$ yielding
$\dim \h_{e_1}^{\rm ab}=\dim \h_{e_2}^{\rm ab}=0+1=1$. This agrees
with our earlier deductions.

\smallskip

\noindent{\bf Example~2.} Now suppose that $\h=\mathfrak{so}_6$, a
Lie algebra of type ${\sf D_3}\cong {\sf A_3}$. In this case
$\h\cong\mathfrak{sl}_4$. The Lie algebra $\mathfrak{sl}_4$ has four
nonzero nilpotent orbits which correspond to the partitions $(4)$,
$(3,1)$, $(2,2)$ and $(2,1,1)$. Using Remark \ref{R1} we see that
$\dim \h_e^{\rm ab}$ equals $3$, $2$, $1$ and $1$ in the respective
cases.

On the other hand, the set $\mathcal{P}_1(6)$ contains four
nontrivial partitions $\mu$, namely, $(5,1)$, $(3,3)$, $(3,1,1,1)$
and $(2,2,1,1)$ and the corresponding nilpotent orbits of $\h$ are
associated with  the partitions $(4)$, $(3,1)$, $(2,2)$ and
$(2,1,1)$ when regarded as elements of $\mathfrak{sl}_4$. Since
$|\Delta(\mu)|=1$ if $\mu$ is one of $(5,1)$, $(3,3)$ or $(2,2,1,1)$
and $\Delta(\mu)=\emptyset$ if $\mu=(3,1,1,1)$, applying Corollary
\ref{expression} yields that $\dim \h_e^{\rm ab}$ equals $3$, $2$,
$1$ and $1$ in the respective cases. This agrees with our earlier
deductions.

 \section{Applications to the theory of sheets in classical Lie algebras}
 \subsection{The Kempken-Spaltenstein algorithm}\label{2.1}
Let $G$ be a simple algebraic group over $\K$ and $m \in \N$. We
recall that a sheet of the Lie algebra $\g=\Lie(G)$ is an
irreducible component of the locally closed set
$$\g^{(m)} := \{x \in \g : \dim \g_x = m\}.$$
Let $\mathcal{N}(\g)$ denote the the variety of all nilpotent
elements in $\g$. It is well known that every sheet of $\g$ contains
a unique nilpotent orbit; see \cite[5.8]{BKr}. However, outside type
$\sf A$ the sheets are not disjoint and a given nilpotent orbit of
$\g\not\cong\mathfrak{sl}_N$ may lie in several different sheets.

Crucial for the theory of sheets in semisimple Lie algebras is the
notion of a rigid element (such elements were termed {\it original}
by Borho). An element $x \in \mathcal{N}(\g)$ is called {\it rigid}
if the ajoint $G$-orbit of $x$ coincides with a sheet of $\g$. Any
rigid element of $\g$ is necessarily nilpotent.

Let $\mathfrak{l}$ be a Levi subalgebra of $\g$. The centre
$\z(\mathfrak{l})$ is a toral subalgebra of $\g$ and for any
$z\in\z(\mathfrak{l})$ the centraliser $\g_z$ contains
$\mathfrak{l}$. We denote by $\z(\mathfrak{l})_{\rm reg}$ the set of
all $z\in\z(\mathfrak{l})$ for which $\g_z=\mathfrak{l}$. This is a
non-empty Zariski open subset of $\z(\mathfrak{l})$. Given a
nilpotent element $e_0\in[\mathfrak{l},\mathfrak{l}]$ we define
$\mathcal{D}(\mathfrak{l},e_0)$ to be the $G$-stable set $({\rm
Ad}\,G)\big(e_0+\z(\mathfrak{l})_{\rm reg}\big)$ and we call
$\mathcal{D}(\mathfrak{l},e_0)$ a {\it decomposition class} of $\g$.

Every sheet $\mathcal{S}$ of $\g$ is a $G$-stable subset of $\g$
locally closed and irreducible in the Zariski topology of $\g$. By a
classical result of Borho \cite{Bor} every sheet is a finite union
of decomposition classes and contains a unique Zariski open such
class. Furthermore a decomposition class
$\mathcal{D}(\mathfrak{l},e_0)$ contained in $\mathcal S$ is open in
$\mathcal S$ if and only if $e_0$ is rigid in $\mathfrak l$; see
\cite[3.7]{Bor}. Conversely every decomposition class
$\mathcal{D}(\mathfrak{l},e_0)$ with $e_0$ rigid in $\mathfrak{l}$
is Zariski open in a unique sheet of $\g$. Furthermore, the unique nilpotent orbit in that sheet  is obtained from $e_0$ by Lusztig-Spaltenstein induction. This result of Borho
gives us a very transparent way to parametrise the sheets of $\g$. 

If $\mathcal{S}$ is a sheet of $\g$ and
$\mathcal{D}(\mathfrak{l},e_0)$ is its open decomposition class then
$\dim \z(\mathfrak{l})$ is called the {\it rank} of $\mathcal S$ and
abbreviated as ${\rm rk}(\mathcal{S})$. This notion is important as
it enables us to determine the dimension of $\mathcal{S}$. Indeed
suppose $\mathcal{S}\subset\g^{(m)}$. Since the morphism $$G\times
\big(e_0+\z(\mathfrak{l}_{\rm reg})\big)\longrightarrow
\,\mathcal{S},\quad\ (g,x)\mapsto (\Ad\, g)x,$$ is dominant, it
follows from the theorem on dimensions of the fibres of a morphism
and the theory of induced conjugacy classes that
$$\dim \mathcal{S}\,=\,\dim \g-m+{\rm rk}(\mathcal{S});$$
see \cite{LS} and \cite{Bor} for more detail.

In this section we deal with sheets in classical Lie algebras and we
keep the notation introduced in Section~1. We shall be discussing
the properties of various different nilpotent orbits in various
different classical Lie algebras simultaneously. In order to
distinguish between the various orbits we shall often appeal to
their associated partitions.

Recall from Section~1 the set $\mathcal{P}_\epsilon(N)$ of
partitions of $N$ associated with the nilpotent elements of $\h$.
Given $e\in\mathcal{N}(\h)$ we denote by $\lambda(e)$ the partition
in $\mathcal{P}_\epsilon(N)$ corresponding to $e$. If
$\lambda=(\lambda_1,\ldots,\lambda_n)\in\mathcal{P}_\epsilon(N)$
then we write  $e(\lambda)$ for any element in $\mathcal{N}(\h)$
whose Jordan block sizes (arranged as in Lemma~\ref{nilpotents}) are
$\lambda_1,\lambda_2,...,\lambda_n$. The map $e\mapsto \lambda(e)$
indices a surjection from the orbit set $\mathcal{N}(\h)/K$ onto
$\mathcal{P}_\epsilon(N)$. The fibres of this surjction are
singletons unless $\g$ is of type $\rm D$ and all parts of $\lambda$
are even. In the latter case the fibre consists of two nilpotent
orbits permuted by an outer automorphism of $\h$ and the two  orbits
in the fibre are traditionally assigned the Roman numerals $I$ and
$I\!I$. Since the centralisers of all elements lying in the fibres
of the above surjection are isomorphic as abstract Lie algebras, the
notation $e(\lambda)$ is unambiguous and will cause no confusion.

The following classification of rigid elements in
$\mathcal{N}(\h)$ was given by Kempken and Spaltenstein:
\begin{thm}{\rm (See \cite{K}, \cite{Sp}.)}\label{rigids}
Let $\lambda = (\lambda_1, ..., \lambda_n)\in
\mathcal{P}_\epsilon(N)$. Then $e(\lambda)\in\mathcal{N}(\h)$ is
rigid if and only if
\begin{itemize}
\item{$\lambda_i - \lambda_{i+1} \in \{0,1\}$ for all $1\leq i\leq n$;}
\item{the set $\{ (i,i+1) \in \Delta(\lambda) : \lambda_i = \lambda_{i+1}\}$ is empty.}
\end{itemize}
\end{thm}
In the above we observe the convention $\lambda_0 = 0$ and
$\lambda_i = 0$ for $i > n$. Note that $(i,i+1) \in \Delta(\lambda)$
implies $\lambda_i - \lambda_{i+1}$ is even by
Lemma~\ref{nilpotents}. Therefore the two conditions for
$e(\lambda)$ rigid together imply $\Delta(\lambda) = \emptyset$ and
we may replace second criterion for rigidity with this apparently
stronger condition. Using our results on the derived subalgebra of
$\h_e$ we recover a result of Yakimova first proven in
\cite[Theorem~12]{Ya10}.
\begin{cor}\label{rrr}
$[\h_{e(\lambda)}, \h_{e(\lambda)}] = \h_{e(\lambda)}$ if and only
if $e(\lambda)$ is rigid.
\end{cor}
\begin{proof}
Evidently $[\h_{e(\lambda)}, \h_{e(\lambda)}] = \h_{e(\lambda)}$ if
and only if $\dim(\h_{e(\lambda)}^\text{ab}) = 0$. Now apply
Corollary \ref{expression} and Theorem \ref{rigids}.
\end{proof}

In view of Theorem~\ref{rigids} we have a well defined notion of a
rigid partition in $\mathcal{P}_\epsilon(N)$ and we denote the set
of all such partitions by $\mathcal{P}_\epsilon^\ast(N)$. Relying on
results of \cite{K} and \cite{Sp} Moreau describes an algorithm
\cite{Mor} which takes $\lambda \in \mathcal{P}_\epsilon(N)$ and
returns an element of $\mathcal{P}_\epsilon^\ast(M)$ for some $M
\leq N$. In this section we also follow \cite{K} and \cite{Sp} and
present an extended version of Moreau's algorithm which will be used
later to determine when a nilpotent element of $\h$ lies in a single
sheet and to confirm a conjecture made by Izosimov in \cite{Iz}.

Throughout the following $\ii$ shall denote a finite sequence of
integers between $1$ and $n$. The procedure is as follows: the
algorithm commences with input $\lambda = \lambda^{\emptyset} \in
\mathcal{P}_\epsilon(N)$ where $\emptyset$ denotes the empty
sequence. At the $l^{\text{th}}$ iteration, the algorithm takes
$\lambda^{\ii} \in \mathcal{P}_\epsilon(N - 2\sum_{j=1}^{l-1} i_j)$
where $\ii = (i_1, ..., i_{l-1})$ and returns $\lambda^{\ii'} \in
\mathcal{P}_\epsilon(N- \sum_{j=1}^l i_j)$ where $\ii' =
(i_1,...,i_{l-1}, i_l)$ for some $i_l$. If the output
$\lambda^{\ii'}$ is a rigid partition then the algorithm terminates
after the $l^{\text{th}}$ iteration with output $\lambda^{\ii'}$. We
shall now explicitly describe the $l^{\text{th}}$ iteration of the
algorithm. If after the $(l-1)^\text{th}$ iteration the input
$\lambda^{\ii}$ is not rigid then the algorithm behaves as follows.
Let $i_l$ denote any index in the range $1\leq i \leq n$ such that
either of the following occurs:
\\ \\ \textbf{Case 1:}\qquad\, $\lambda^{\ii}_{i_l} \geq \lambda^{\ii}_{i_l+1}+2$.\\

\noindent
\textbf{Case 2:} \qquad $(i_l,i_l+1) \in \Delta(\lambda^{\ii})\, \text{ and }\,
\lambda_{i_l}^\ii = \lambda^\ii_{i_l+1}.$\\

\noindent Note that no integer $i_l$ will fulfil both of these
criteria. If $\ii = (i_1,..., i_{l-1})$ then define $\ii' =
(i_1,...,i_{l-1}, i_l)$. For Case 1 the algorithm has output
$$\lambda^{\ii'} = (\lambda^{\ii}_1 - 2, \lambda^{\ii}_2 - 2, ...,
\lambda^{\ii}_{i_l}-2, \lambda^{\ii}_{i_l+1},..., \lambda^{\ii}_n)$$
whilst for Case 2 the algorithm has output
$$\lambda^{\ii'} = (\lambda^{\ii}_1-2, \lambda^{\ii}_2-2,...,\lambda^{\ii}_{i_l-1} -2,
\lambda^{\ii}_{i_l} - 1, \lambda^{\ii}_{i_l+1} - 1, \lambda^{\ii}_{i_l+2},...,\lambda^{\ii}_n).$$

In what follows we shall often refer to the algorithm just described
as the {\it KS algorithm} (after Kempken and Spaltenstein). Due to
its definition and the classification of rigid partitions the KS
algorithm certainly terminates after a finite number of steps. In
the hope of avoiding any confusion we shall use `Case' when
referring to Case 1 or Case 2 of the algorithm, and we shall use
`case' to refer to a particular situation. We shall say that a
sequence $\ii = (i_1, i_2, ..., i_l)$ is an \emph{admissible sequence
for} $\lambda$ if Case 1 or Case 2 occurs at the point $i_k$ for the
partition $\lambda^{(i_1,...,i_{k-1})}$ for each $k=1,...,l$. We
shall use the notation $|\ii|$ to denote the length of such a
sequence. An admissible sequence $\ii$  for $\lambda$ shall be
called \emph{maximal admissible for} $\lambda$ if neither Case 1 nor Case 2
occurs for any index $i$ between 1 and $n$ for the partition
$\lambda^{\ii}$. If a sequence $\ii = (i_1,...,i_l)$ is admissible
for $\lambda$ and  $1\leq j \leq l+1$ then we shall use the notation
$\ii_j = (i_1,...,i_{j-1})$. Clearly the sequence $\ii_j$ is
admissible for $\lambda$  for any $1 \leq j\leq l+1$. By convention
the empty sequence is admissible for any
$\lambda\in\mathcal{P}_\epsilon(N)$.
\begin{lem}\label{maxadmiss}
Let \emph{$\ii$} be an admissible sequence for $\lambda$. Then $\bf
i$ is maximal admissible if and only if $\lambda^{\text{\emph{\bf
i}}}$ is a rigid partition.
\end{lem}
\begin{proof} In view of Theorem~\ref{rigids}
this follows from the definition of maximal admissible sequences.
\end{proof}
\begin{rem}\label{remark1}\,
{\rm {\rm(i)}\ Rather that defining $i_l$ to be any index between
$1$ and $n$ such that Case 1 or Case 2 occurs, Moreau's algorithm in
\cite{Mor} defines $i_l$ to be the \emph{smallest} such index. This
discrepancy ensures that her algorithm is deterministic (the outcome
does not depend upon a choice of indices $i_1, i_2, i_3,...$). In a
sense, being non-deterministic is an advantage of the KS algorithm
and we shall see later that it has enough power to reach and pin
down {\rm all} sheets of $\h$ containing a given nilpotent element.

\smallskip

\noindent {\rm (ii)}\  The KS algorithm is transitive in the
following sense: if $\ii$ is an admissible sequence for $\lambda$
and $\jj$ is an admissible sequence for $\lambda^{\ii}$ then $(\ii,
\jj)$ is an admissible sequence for $\lambda$, where $(\ii, \jj)$
denotes the concatenation of the two sequences $\ii$ and $\jj$.
Furthermore $\lambda^{(\ii, \jj)} = (\lambda^\ii)^\jj$.}
\end{rem}

\subsection{Non-singular partitions and preliminaries of the algorithm}
Before placing the algorithm into the geometric context for which it
was intended we shall discuss it combinatorially. We start by
introducing a combinatorial notion related with the notion of a boundary of $\lambda\in{\mathcal P}_\epsilon(N)$; see Definition~\ref{de}.

A 2-step $(i, i+1)\in\Delta(\lambda)$ is said to be {\it good} if $\lambda_i$ and the boundary of $(i,i+1)$
have the opposite parity. It is worth mentioning that if $(i,i+1)$ is a good 2-step with $i>1$ then both $\lambda_{i-1}$ and $\lambda_{i+2}$ must have the same parity. If a 2-step $(i,i+1)\in\Delta(\lambda)$ is not good then we say that it is {\it bad}. We note that $(i,i+1)$ is a
bad $2$-step of $\lambda$ if and only if either $i>1$ and
$\lambda_{i-1} - \lambda_i \in 2\N$ or $\lambda_{i+1} -
\lambda_{i+2} \in 2\N$.

We call a partition $\lambda\in\mathcal{P}_\epsilon(N)$
\emph{singular} if it has a bad 2-step. Naturally if all 2-steps of
$\lambda$ are good then we call $\lambda$ \emph{non-singular}. In
the next section we shall interpret these singular and non-singular
partitions in geometric terms. In particular we shall show that
singular partitions correspond precisely to the nilpotent singular
points on the varieties $\h^{(m)}$, hence their name.

\begin{figure}[htb]
\setlength{\unitlength}{0.017in}
\begin{center}
\begin{picture}(250,80)(0,0)

\put(0,0){\line(0,1){80}}
\put(0,80){\line(1,0){40}}
\put(40,80){\line(0,-1){20}}
\put(40,60){\line(1,0){40}}
\put(80,60){\line(0,-1){40}}
\put(80,20){\line(1,0){20}}
\put(100,20){\line(0,-1){20}}
\put(100,0){\line(-1,0){100}}

\qbezier[24](0,60),(20,60),(40,60)
\qbezier[48](0,40),(40,40),(80,40)
\qbezier[48](0,20),(40,20),(80,20)

\put(20,0){\line(0,1){80}}
\put(40,0){\line(0,1){60}}
\put(60,0){\line(0,1){60}}
\put(80,0){\line(0,1){20}}

\put(150,0){\line(0,1){80}}
\put(150,80){\line(1,0){20}}
\put(170,80){\line(0,-1){40}}
\put(170,40){\line(1,0){40}}
\put(210,40){\line(0,-1){20}}
\put(210,20){\line(1,0){40}}
\put(250,20){\line(0,-1){20}}
\put(250,0){\line(-1,0){100}}

\qbezier[12](150,60),(160,60),(170,60)
\qbezier[12](150,40),(160,40),(170,40)
\qbezier[36](150,20),(180,20),(210,20)

\put(170,0){\line(0,1){40}}
\put(190,0){\line(0,1){40}}
\put(210,0){\line(0,1){20}}
\put(230,0){\line(0,1){20}}
\end{picture}
\end{center}
\caption{The Young diagrams of two singular partitions in
$\mathcal{P}_1(15)$ and $\mathcal{P}_{-1}(10)$. The bad 2-steps are
$(3,4)$ and $(2,3)$, respectively.}\label{pikcha_A}
\end{figure}
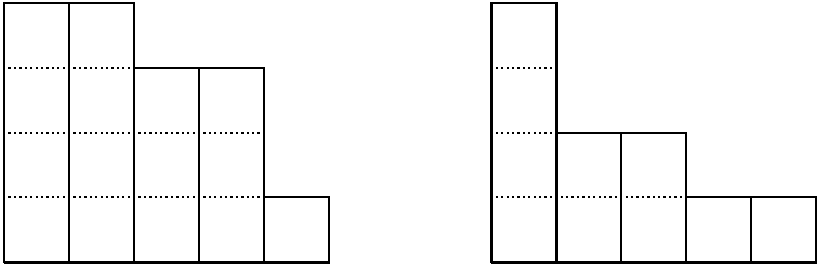

We now collect some elementary lemmas about the behaviour of the
algorithm. For the remnant of the subsection we assume that $\lambda
\in \mathcal{P}_\epsilon(N)$ has the standard ordering
$\lambda_1\geq \cdots \geq \lambda_n$.
\begin{lem}\label{deltain2}
Suppose $\emph{\ii} = (i)$ is a sequence of length $1$. If Case 2
occurs for $\lambda$ at index $i$ then $\Delta(\lambda^{\emph{\ii}})
= \Delta(\lambda) \setminus \{(i,i+1)\}$. Furthermore, if $(i,i+1)$ is a
good $2$-step of $\lambda$ then $s(\lambda^\emph{\ii})=s(\lambda).$
\end{lem}
\begin{proof}
We shall suppose that there is a 2-step $$(j,j+1) \in
\Delta(\lambda)\setminus \big( \Delta(\lambda^{\ii})\cup\{(i,i+1)\}\big)$$
and derive a contradiction. Observe that if $j < i-2$ (resp. $j>
i+2$) then for $k \in\{ j-1, j, j+1, j+2\}$ we have that
$\lambda_{k}^{\ii}=\lambda_{k}-2$ (resp.
$\lambda_{k}^{\ii}=\lambda_{k}$). So $(j,j+1) \in \Delta(\lambda)$ if and
only if $(j,j+1) \in \Delta(\lambda^{\ii})$. It remains to show that if $j
= i\pm1$ or $j = i \pm 2$ and $(j,j+1) \in \Delta(\lambda)$ then $(j,j+1) \in
\Delta(\lambda^{\ii})$. If $j = i\pm1$ and $(j,j+1) \in \Delta(\lambda)$
then $\lambda_i \neq \lambda_{i+1}$ contradicting the fact that Case
2 occurs for $\lambda$ at index $i$.

Suppose $j = i-2$. Then $(j,j+1), (j+2,j+3) \in \Delta(\lambda)$ and hence
$\lambda_{j+1} \neq \lambda_{j+2}$ and $(j+1)'=j+1$, $(j+2)'=j+2$.
As a consequence $\lambda_{j+1} - \lambda_{j+2}$ is even implying
that $\lambda_{j+1} - \lambda_{j+2}\geq 2$ and  $\lambda_{j+1}^{\ii}
\neq \lambda_{j+2}^{\ii}$. Since for $k\in\{j-1, j , j+1\}$ the
equality $\lambda_{k}^{\ii}=\lambda_k-2$ holds, we conclude that $(j,j+1)
\in \Delta(\lambda^{\ii})$. A similar argument shows that if $j =
i+2$ then $(j,j+1) \in \Delta(\lambda)$ implies $(j,j+1) \in
\Delta(\lambda^{\ii})$. We conclude that $\Delta(\lambda^{\rm{\ii}})
= \Delta(\lambda) \setminus \{(i,i+1)\}$.

Now suppose $(i,i+1)$ is a good 2-step of $\lambda$.  Since
$\lambda_{i+1}-\lambda_{i+2}$ and $\lambda_{i-1}-\lambda_i$ if $i>1$
are odd we have that
$$\big\lfloor(\lambda_{i+1}^\ii-\lambda_{i+2}^\ii)/2\big\rfloor
=\lfloor((\lambda_{i+1}-1)-\lambda_{i+2})/2\big\rfloor=
\lfloor(\lambda_{i+1}-\lambda_{i+2})/2\big\rfloor$$
and
$$\big\lfloor(\lambda_{i-1}^\ii-
\lambda_{i}^\ii)/2\big\rfloor=
\big\lfloor((\lambda_{i-1}-2)-(\lambda_i-1))/2\big
\rfloor=
\big\lfloor(\lambda_{i-1}-\lambda_{i})/2\big\rfloor$$
if $i>1$.
As $\lambda_j^\ii=\lambda_j$ for $j\not\in\{i,i+1\}$ it
follows that  $s(\lambda^\ii)=s(\lambda)$ as claimed.
\end{proof}

\begin{lem}\label{deltain}
If $\emph{\ii}$ is an admissible sequence for $\lambda$ then
$\Delta(\lambda^{\emph{\ii}}) \subseteq \Delta(\lambda)$.
\end{lem}
\begin{proof}
In view of Lemma \ref{deltain2} and Remark \ref{remark1}(ii) it will
suffice to prove the current lemma when $\ii= (i)$ and $i$ is an
index at which Case 1 occurs for $\lambda$. Suppose $(j,j+1) \in
\Delta(\lambda^\ii)$. Then since Case 1 preserves the parity of the
entries of $\lambda$ (that is to say $\lambda_k^\ii \equiv \lambda_k
\mod 2$ for $1\leq k \leq n$), we deduce that $j'=j$ and
$(j+1)'=j+1$. If $j<i$ or $j>i+1$ then $\lambda_{j-1} - \lambda_j =
\lambda_{j-1}^\ii - \lambda_j^\ii$ and  $\lambda_{j+1} -
\lambda_{j+2}=\lambda_{j+1}^\ii - \lambda_{j+2}^\ii$ showing that $(j,j+1)
\in \Delta(\lambda)$ in these cases. If $j=i+1$ then  $\lambda_{j-1}
- \lambda_j = \lambda_{j-1}^\ii - \lambda_j^\ii+2$ and
$\lambda_{j+1} - \lambda_{j+2} = \lambda_{j+1}^\ii -
\lambda_{j+2}^\ii$. Hence $(j,j+1)\in\Delta(\lambda)$. Finally, if $j=i$
then  $\lambda_{j-1} - \lambda_j = \lambda_{j-1}^\ii -
\lambda_j^\ii+2$ and $\lambda_{j+1} - \lambda_{j+2} =
\lambda_{j+1}^\ii - \lambda_{j+2}^\ii$. Thus $(j,j+1)\in\Delta(\lambda)$
in all cases and our proof is complete.
\end{proof}

\begin{lem}\label{goodinherit}
If $(i, i+1)$ is a good 2-step for $\lambda$, $\emph{\ii}$ is an
admissible sequence and $(i,i+1) \in \Delta(\lambda^\emph{\ii})$ then
$(i,i+1)$ is a good 2-step for $\lambda^\emph{\ii}$.
\end{lem}
\begin{proof}
It suffices to prove the lemma when $\ii = (i_1)$ is an admissible
of length 1. If Case 1 occurs at index $i_1$ then $\lambda_j^\ii -
\lambda_{j+1}^\ii \equiv \lambda_j - \lambda_{j+1} \mod 2$ for all
$j$. Since $(i,i+1)$ is good for $\lambda$ it follows that
$\lambda_{i-1}^\ii - \lambda_i^\ii$ is odd (or $i=1$) and
$\lambda_{i+1}^\ii - \lambda_{i+2}^\ii$ is odd, so that $(i,i+1)$ is
a good 2-step for $\lambda^\ii$. Now suppose Case 2 occurs for
$\lambda$ at index $i_1$. We may assume that $i_1 \neq i$. If $i_1 =
i-1$ or $i_1 = i-2$ then $(i_1,i_1+1) \in \Delta(\lambda)$ implies
$\epsilon(-1)^{\lambda_{i-1}} = -1$ and $\lambda_{i-1} - \lambda_i$
is even, contrary to the assumption that the 2-step $(i,i+1)$ is
good for $\lambda$. Similarly, if $i_1 = i+1$ or $i_1 = i+2$ then
$\lambda_{i+1} - \lambda_{i+2}$ is even, contradicting the
assumption that $(i,i+1)$ is good. It follows that $i_1<i-2$ or $i_1
> i+2$, from whence it immediately follows that $(i,i+1)$ is a good
2-step for $\lambda^\ii$.
\end{proof}

\begin{cor}\label{singinherit}
If $\lambda$ is non-singular then $\lambda^{\emph{\ii}}$ is
non-singular for any admissible sequence $\emph{\ii}$.
\end{cor}
\begin{proof}
If $(i,i+1) \in \Delta(\lambda^\ii)$ then $(i,i+1) \in \Delta(\lambda)$ by
Lemma~\ref{deltain}. Since $\lambda$ is non-singular, $(i,i+1)$ is a
good 2-step for $\lambda$. By Lemma~\ref{goodinherit}, $(i,i+1)$ is
good for $\lambda^\ii$.
\end{proof}

\subsection{The length of admissible sequences}\label{3.3}
In this section we shall give a combinatorial formula for the
maximal length of admissible sequences for $\lambda$. The formula
shall be of central importance to our results on sheets. First we
shall need  some further terminology related to partitions
$\lambda=(\lambda_1,\ldots,\lambda_n)\in \mathcal{P}_\epsilon(N)$.

\begin{defn}
A sequence $1\leq i_1 < i_2 < \cdots < i_k < n$ with $k\ge 2$ is
called a \emph{2-cluster} of $\lambda$ whenever $(i_j,i_j+1) \in
\Delta(\lambda)$ and $i_{j+1} = i_j + 2$ for all $j$. Analogous to
the terminology for 2-steps we say that a 2-cluster $i_1,...,i_k$
\emph{has a bad boundary} if either of the following conditions
holds:
\begin{itemize}
\item{$\lambda_{i_1-1} - \lambda_{i_1} \in 2\N$;}
\smallskip
\item{$\lambda_{i_k+1} - \lambda_{i_k+2} \in 2\N$}
\end{itemize}
(if $i_1=1$ the the first condition should be omitted). A \emph{bad
2-cluster} is one which has a bad boundary, whilst a \emph{good
2-cluster} is one without a bad boundary.
\end{defn}
\begin{lem}
A good 2-cluster is maximal in the sense that it is not a proper
subsequence of any 2-cluster.
\end{lem}
\begin{proof}
If $i_1,...,i_k$ is a good 2-cluster then $\lambda_{i_1-1} -
\lambda_{i_1}, \lambda_{i_k+1} - \lambda_{i_k+2} \notin 2\N$. The
fact that $(i_1,i_1+1), (i_k,i_k+1) \in \Delta(\lambda)$ means that
$\epsilon(-1)^{\lambda_{i_1}} = \epsilon(-1)^{\lambda_{i_k+1}} =
-1$. Combining these few observations we get
$\epsilon(-1)^{\lambda_{i_1-1}} = \epsilon(-1)^{\lambda_{i_k+2}} =
1$ and so $(i_1-2,i_1-1) \notin \Delta(\lambda)$ and $(i_k + 2,i_k+3) \notin
\Delta(\lambda)$.
\end{proof}

We introduce the notations:
\begin{eqnarray*}
\Delta_{\rm bad}(\lambda) &:=& \{\text{the bad 2-steps of $\lambda$}\};\\
\Sigma(\lambda) &:=& \{\text{the good 2-clusters of $\lambda$}\};
\end{eqnarray*}
and write
$$z(\lambda) = s(\lambda) + |\Delta(\lambda)| - \big(|\Delta_{\rm bad}(\lambda)| -
|\Sigma(\lambda)|\big).$$
It is immediate from the definitions that $|\Delta_{\rm bad}(\lambda)| \ge
|\Sigma(\lambda)|$ and $|\Delta_{\rm bad}(\lambda)| = |\Sigma(\lambda)|$ if
and only if $\Delta_{\rm bad}(\lambda)=\emptyset$.

\begin{lem}\label{stairslemma}
$|\Sigma(\lambda)| \geq |\Sigma(\lambda^{\emph\ii})|$ for length 1
admissible sequences $\emph\ii = (i)$, unless Case 2 occurs at $i$
and $$i-4, i-2, i, i+2, i+4$$ is a subsequence of a good 2-cluster,
in which case $|\Sigma(\lambda)| = |\Sigma(\lambda^\ii)| - 1$.
\end{lem}
\begin{proof}
We make the notation $\ii =(i)$. In this first paragraph we deal
with the possibility that Case 1 occurs for $\lambda$ at index $i$.
Let us consider some necessary conditions for $\Sigma(\lambda) \neq
\Sigma(\lambda^{\ii})$. We require that $(i-1,i)$ or $(i+1,i+2)$ lie in
$\Delta(\lambda)$, that the 2-steps $(i-1,i)$ or $(i+1,i+2)$ (or
both) constitute a 2-step in a good 2-cluster, and that $\lambda_i -
\lambda_{i+1} = 2$. Let us assume these conditions. If precisely one
of the two pairs $(i-1,i), (i+1,i+2)$ lies in $\Delta(\lambda)$ (we may assume $(i-1,i)
\in \Delta(\lambda)$) then it follows that the good 2-cluster in
question has the form $i_1\leq \cdots \leq i_k = i-1$. But
$\lambda_{i_k + 1} - \lambda_{i_k + 2} = 2$ then implies that the
2-cluster has a bad boundary; a contradiction. It follows that both
$(i-1,i)$ and $(i+1,i+2)$ lie in $\Delta(\lambda)$. Then we have a good
2-cluster $i_1 \leq \cdots \leq i-1 = i_l \leq i_{l+1} = i+1 \leq
\cdots \leq i_k$. However the sequences $i_1,i_2,...,i_{l-1}$ and
$i_{l+2},...,i_{k-1},i_k$ are either of length $\leq 1$, or are bad
2-clusters for $\lambda^\ii$, so $|\Sigma(\lambda)| =
|\Sigma(\lambda^\ii)| + 1$.

Now suppose Case 2 occurs at index $i$. Similar to the previous case
$\Sigma(\lambda)$ is only affected if $(i,i+1)$ is a bad 2-step in a
good 2-cluster. If precisely one of the two pairs $(i-2,i-1)$ and $(i+2,i+3)$ lie in
$\Delta(\lambda)$ (we may assume $(i-2,i-1)\in \Delta(\lambda))$ then such
a 2-cluster will take the form $i_1,...,i_k = i$. If $k > 2$ then
$i_1,...,i_{k-1}$ is a good 2-cluster for $\lambda^\ii$ so that
$|\Sigma(\lambda^\ii)| = |\Sigma(\lambda)|$. If $k = 2$ (we know
$k\geq 2$) then the 2-cluster is eradicated by the iteration of the
algorithm and $|\Sigma(\lambda^\ii)| = |\Sigma(\lambda)| - 1$.

Suppose that both $(i-2,i-1)$ and $(i+2,i+3)$ lie in $\Delta(\lambda)$. Then
$\Sigma(\lambda)$ is unaffected unless $i_1,...,i_j = i,...,i_k$ is
a good 2-cluster, which we shall assume from henceforth. Note that
$j \geq 2$ and $k - j \geq 1$ by assumption. If $j=2$ and $k-j=1$
then the good 2-cluster is no longer present for $\lambda^\ii$ and
$|\Sigma(\lambda)| = |\Sigma(\lambda^\ii)| - 1$. If $j> 2$ and
$k-j=1$ then $i_1,...,i_{j-1}$ is a good 2-cluster for $\lambda^\ii$
and $|\Sigma(\lambda)| = |\Sigma(\lambda^\ii)|$. The situation when
$j=2$ and $k-j>1$ is very similar. In the final case $j > 2$,
$k-j>1$ and $i-4, i-2, i, i+2, i+4$ is a subsequence of a good
2-cluster, as in the statement of the lemma. Here both $i-2j, ...,
i-2$ and $i,i+2,..., i+2k$ are good 2-clusters for $\lambda^\ii$ so
that $|\Sigma(\lambda)| = |\Sigma(\lambda^\ii)| - 1$ as required.
\end{proof}

Before continuing we shall need some notation. We define a
construction which takes $\lambda\in \PP(N)$ to $\lambda^S\in
\mathcal{P}_\epsilon(N-2k)$ for some $k\geq 0$. It is based entirely
on application of the algorithm. The partition $\lambda^S$ is call
\emph{the shell of $\lambda$} and is constructed as follows: for all
$1\leq i \leq n$ we apply Case 1 repeatedly; if $\lambda_{i} -
\lambda_{i+1} \in 2\N$ and if $(i-1,i)$ or $(i+1,i+2)$ lie in
$\Delta(\lambda)$ then apply Case 1 until $\lambda^\ii_{i} -
\lambda^\ii_{i+1} = 2$; if we are not in the previous situation then
apply Case 1 until $\lambda^\ii_{i} - \lambda^\ii_{i+1} \in
\{0,1\}$; finally apply Case 2 at every index $i$ such that
$(i,i+1)$ is a good 2-step. In order to keep the notation consistent
we may regard $S$ as the admissible sequence of indices (chosen in
ascending order) used to construct $\lambda^S$.

Retain the convention $\lambda = (\lambda_1,...,\lambda_n)$ with
$\sum\lambda_i = N$. In order to make use of the shell $\lambda^S$
we shall interest ourselves firstly in the set of partitions which
equal their own shell $\lambda =   \lambda^S$, and secondly in the
relationship between a partition and its shell. It turns out that
certain properties of a partition $\lambda = \lambda^S$ are
controlled by the properties of certain special partitions
constructed from $\lambda$. A \emph{profile} $\mu$ of $\lambda$ is a
partition constructed in the following manner: choose indices
$(j,k)$ with $0 < j \leq k \leq n+1$ such that $i = i'$ for all $j\leq
i < k$, and such that $j-1\neq (j-1)'$ (or $j-1 = 0$) and $k\neq k'$
(or $k = n+1)$. Define $\mu = (\mu_1,...,\mu_{k-j})$ by the rule
$$\mu_i = \lambda_{i + (j-1)} - \lambda_k.$$ If $k < n+1$ then in
order to preserve the condition $i=i'$ we regard $\mu$ as an element
of $\mathcal{P}_{1}(\sum_{i=j}^{k-1} \lambda_i - (k-j)\lambda_k)$. If $k = n+1$ then
$\lambda_{k} = 0$ and we may regard $\mu$ is an element of
$\mathcal{P}_{\epsilon}(\sum_{i=j}^{n} \lambda_i)$. We say that the
profile $\mu$ constructed in this manner is \emph{of type $(j,k)$},
and we include Figure~\ref{pikcha_B} to show what is intended by the
definition.
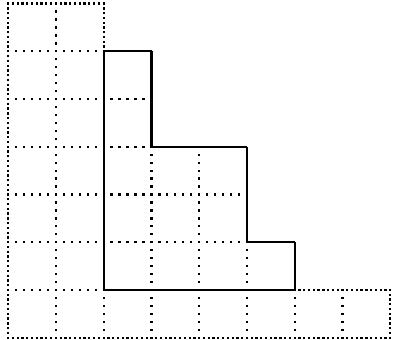
\begin{figure}[htb]
\setlength{\unitlength}{0.025in}
\begin{center}
\begin{picture}(80,70)(0,0)

\qbezier[70](0,0),(0,35),(0,70)
\qbezier[20](0,70),(10,70),(20,70)
\qbezier[10](20,70),(20,65),(20,60)
\qbezier[20](60,10),(70,10),(80,10)
\qbezier[10](80,10),(80,5),(80,0)
\qbezier[80](80,0),(40,0),(0,0)

\put(20,10){\line(0,1){50}}
\put(20,60){\line(1,0){10}}
\put(30,60){\line(0,-1){20}}
\put(30,40){\line(1,0){20}}
\put(50,40){\line(0,-1){20}}
\put(50,20){\line(1,0){10}}
\put(60,20){\line(0,-1){10}}
\put(60,10){\line(-1,0){40}}

\qbezier[42](10,0),(10,35),(10,70)
\qbezier[6](20,0),(20,5),(20,10)
\qbezier[24](30,0),(30,20),(30,40)
\qbezier[24](40,0),(40,20),(40,40)
\qbezier[12](50,0),(50,10),(50,20)
\qbezier[6](60,0),(60,5),(60,10)
\qbezier[6](70,0),(70,5),(70,10)

\qbezier[12](0,10),(10,10),(20,10)
\qbezier[30](0,20),(25,20),(50,20)
\qbezier[30](0,30),(25,30),(50,30)
\qbezier[18](0,40),(15,40),(30,40)
\qbezier[18](0,50),(15,50),(30,50)
\qbezier[12](0,60),(10,60),(20,60)

\end{picture}
\end{center}
\caption{The dotted perimeter represents the Young diagram of the partition
$\lambda=(7,7,6,4,4,2,1,1) \in \mathcal{P}_{-1}(32)$. The solid perimeter represents
the profile of $\lambda$ of type $(3,7)$.}\label{pikcha_B}
\end{figure}

Suppose $\mu$ is a profile of $\lambda$ of type $(j,k)$ and $\ii = (i_1,...,i_l)$ is
an admissible sequence for $\mu$. Then the \emph{$j$-adjust} of $\ii$ is the sequence
$$\jmath(\ii) = (i_1 + (j-1), i_2 + (j-1), ..., i_l + (j-1)).$$
It is clear that $\jmath(\ii)$ is an admissible sequence for $\lambda$.

\begin{prop}\label{reductionprop}
Suppose $\lambda$ is equal to its shell and let $\mu(1), \mu(2),..., \mu(l)$ be a
complete set of distinct profiles for $\lambda$, with $\mu(m)$ of type $(j_m, k_m)$. Then the following hold:
\begin{enumerate}
\item{$z(\lambda) = \sum_{i=1}^l z(\mu(i))$.}

\smallskip

\item{If $\ii(m)$ is an admissible sequence for $\mu(m)$ then
$$(\jmath_1(\ii(1)), \jmath_2(\ii(2)),...,\jmath_l(\ii(l)))$$ is an admissible sequence for $\lambda$,
where this last sequence is obtained by concatenating the sequences $\jmath_m(\ii(m))$.
}
\end{enumerate}
\end{prop}
\begin{proof}
Since $\lambda = \lambda^S$ all differences $\lambda_i-\lambda_{i+1}$ are equal to $0,1,$ or $2$. If
$\lambda_i - \lambda_{i+1} = 2$ then necessarily $(i-1,i)\in \Delta(\lambda)$ or $(i+1,i+2) \in \Delta(\lambda)$.
In either case $i=i'$, $i+1 = (i+1)'$ (or $i = n$) and it follows that there exists a profile of type $(j,k)$
with $j \leq i$ and $i+1 < k$ (or $i < k$ when $i = n$). Then each index $i$ for which $\lambda_i - \lambda_{i+1} = 2$
contributes 1 to $s(\lambda)$ and 1 to $\sum_{j=1}^l s(\mu(j))$ so that $s(\lambda) = \sum_{j=1}^l s(\mu(j))$. The
condition $\lambda = \lambda^S$ also implies that all 2-steps are bad 2-steps so that
$|\Delta(\lambda)| = |\Delta_{\rm bad}(\lambda)|$. Similarly $\mu(m) = \mu(m)^S$ so
$|\Delta(\mu(m))| = |\Delta_{\rm bad}(\mu(m))|$ for all $m$, and it remains to prove
that $|\Sigma(\lambda)| = \sum_{i=1}^l |\Sigma(\mu(i))|$. This follows from the fact that
all good 2-clusters $i_1\leq \cdots \leq i_l$ fulfil $i = i'$ for all $i_1\leq i \leq i_l+1$
so for each such 2-cluster there exists profile of type $(j,k)$ with $j\leq i_1$ and $i_l +1< k$.
Part (1) follows.

The second actually holds even when $\lambda \neq \lambda^S$. For obvious reasons the indices of the distinct
profiles do not overlap, and we may assume that $k_m < j_{m+1}$ for $m=1,...,l-1$. Then for $1 \leq i < l$
we set $\jj(i) = (\jmath_1(\ii(1)),...,\jmath_i(\ii(i)))$ and note that $\lambda_r^{\jj(i)} = \lambda_r$
for all $r \geq j_{i+1}$. Using that $\jmath_{i+1}(\ii(i+1))$ is admissible for $\lambda$ we obtain by
induction that $\jmath_{i+1}(\ii(i+1))$ is an admissible sequence for $\lambda^{\jj(i)}$. By the transitivity
of the algorithm we deduce then that $(\jmath_1(\ii(1)), \jmath_2(\ii(2)),...,\jmath_l(\ii(l)))$ is admissible
for $\lambda$ as required.
\end{proof}

\begin{prop}\label{profileprop}
Let $\lambda = (\lambda_1,..,\lambda_n)$ be a partition and suppose that $i=i'$ for all $1\leq i\leq n$. Then there
exists an admissible sequence for $\lambda$ of length $z(\lambda)$.
\end{prop}
\begin{proof}
A partition $\lambda$ fulfilling $i=i'$ for all $1\leq i\leq n$ contains a good 2-cluster if and only if
$1, 3, 5, ..., n-1$ is good 2-cluster. In this case it is the only good 2-cluster. Suppose that this is the case.
Of course this implies that $n$ is even and $\epsilon = 1$, so $\lambda_n$ is odd. Construct a sequence $\ii$ by
repeatedly applying Case 1 at indices $2i-1$ for $1\leq i \leq \frac{n}{2}$ so that $\lambda_{2i-1} - \lambda_{2i} = 0$
for all such $i$. Then $$|\ii| = \sum_{i=1}^{\frac{n}{2}} \lfloor \frac{\lambda_{2i-1} - \lambda_{2i}}{2} \rfloor.$$
We construct an admissible sequence $\ii'$ by subsequently applying Case 1 at indices $2i$ for $1\leq i\leq n$ so that
$\lambda_{2i} - \lambda_{2i+1} = 2$ for all such $i$. Our sequence $\ii'$ has length $$|\ii'| =
s(\lambda) - (\frac{n}{2}-1).$$ At this point we are able to say precisely what $\lambda^{\ii'}$ looks like.
We have $\lambda^{\ii'} = \lambda^S = (n-1,n-1,n-3,n-3,...,3,3,1,1)$. Finally we obtain $\ii''$ by applying
Case 2 precisely once at each index $2i-1$ for $1\leq i\leq \frac{n}{2}$. The partition $\lambda^{\ii''}$ is
rigid, so $\ii''$ is maximal (Lemma~\ref{maxadmiss}) and $$|\ii''| = s(\lambda) + 1.$$ In order to complete
this part of the proof we must show that $z(\lambda) = s(\lambda) + 1$. Notice that our assumptions on $\lambda$
imply that every 2-step is bad. Therefore $|\Delta(\lambda)| = |\Delta_{\rm bad}(\lambda)|$ and by our original
remarks $z(\lambda) = s(\lambda) + 1$ as required.

Now assume that $\lambda$ has no good 2-clusters. Since $i=i'$ for all $i$ we may apply Case 1 repeatedly at all
indices to obtain a maximal admissible partition. Clearly $|\ii| = s(\lambda)$. Once again all 2-steps are bad so
that $|\Delta(\lambda)| = |\Delta_{\rm bad}(\lambda)|$, and by assumption $|\Sigma(\lambda)| = 0$. Hence
$z(\lambda) =s(\lambda) =  |\ii|$ as promised.
\end{proof}

\begin{thm}\label{zismax} We have that
$$z(\lambda) = \max |\emph\ii|$$ where the maximum is taken over all admissible sequences $\emph\ii$ for $\lambda$.
\end{thm}
\begin{proof}
We begin by showing that $z(\lambda) \geq z(\lambda^{\ii}) + 1$ where $\ii = (i)$ is an admissible sequence
of length 1 for $\lambda$. First assume Case 1 occurs for $\lambda$ at $i$. Then $s(\lambda^\ii) = s(\lambda) - 1$.
Furthermore, if the iteration at $i$ removes a 2-step (ie. if $\lambda_{i} - \lambda_{i+1} = 2$ and either
$(i-1,i) \in \Delta(\lambda)$ or $(i+1,i+2) \in \Delta(\lambda)$ or both) then that 2-step is bad. Therefore
$|\Delta(\lambda)| - |\Delta(\lambda^\ii)| = |\Delta_{\rm bad}(\lambda)| - |\Delta_{\rm bad}(\lambda^\ii)|$.
It remains to be seen that the number of good 2-clusters does not increase as we pass from $\lambda$ to
$\lambda^{\ii}$. This follows from Lemma~\ref{stairslemma}.

Now suppose that Case 2 occurs for $\lambda$ at index $i$. Certainly if $(i,i+1)$ is a good 2-step then
$z(\lambda^\ii) = z(\lambda) - 1$, so we may assume that $(i, i+1)$ is a bad 2-step. Suppose first that
this 2-step has precisely one bad boundary. We may assume that $\lambda_{i-1} - \lambda_i $ is even and
$\lambda_{i+1} - \lambda_{i+2}$ is odd. We can deduce at this point that $s(\lambda^\ii) = s(\lambda) - 1$
and $|\Delta(\lambda^\ii)| = |\Delta(\lambda)| - 1$. If $(i-2,i-1) \notin \Delta(\lambda)$ then
$|\Delta_{\rm bad}(\lambda^\ii)| = |\Delta_{\rm bad}(\lambda)| - 1$. Similarly, if
$(i-2,i-1) \in \Delta(\lambda)$ and $\lambda_{i-3} - \lambda_{i-2}$ is even then
$|\Delta_{\rm bad}(\lambda^\ii)| = |\Delta_{\rm bad}(\lambda)| - 1$. In either of these
two situations the number of good 2-clusters decreases, thanks to Lemma~\ref{stairslemma}.
Hence $z(\lambda) \geq z(\lambda^{\ii}) + 1$ once again. We must now consider the possibility
that $(i-2,i-1) \in \Delta(\lambda)$ and $\lambda_{i-3} - \lambda_{i-2}$ is odd. In this situation
$s(\lambda^\ii) = s(\lambda) - 1$, $|\Delta(\lambda^\ii)| = |\Delta(\lambda)| - 1$ and
$|\Delta_{\rm bad}(\lambda^\ii)| = |\Delta_{\rm bad}(\lambda)| - 2$. Notice that $i-2, i$ is a good 2-cluster
for $\lambda$ but not for $\lambda^\ii$, so that $|\Sigma(\lambda^\ii)| = |\Sigma(\lambda)| - 1$ and
$z(\lambda^\ii) = z(\lambda) - 1$. A similar argument works when $\lambda_{i-1} - \lambda_i$ is odd but
$\lambda_{i+1} - \lambda_{i+2} = 2$.

Now we assume that $(i, i+1)$ is a bad 2-step and that both boundaries are bad. If neither $(i-2,i-1)$ nor $(i+2,i+3)$ lie
in $\Delta(\lambda)$ then $s(-)$ decreases by 2, $|\Delta(-)|$ decreases by 1, and $|\Delta_{\rm bad}(-)|$ decreases by 1
upon passing from $\lambda$ to $\lambda^\ii$. Certainly $|\Sigma(-)|$ may only decrease, by lemma \ref{stairslemma}, and
so $z(\lambda) \geq z(\lambda^\ii) + 1$ in this situation. Now move on and suppose that precisely one of $i-2$ and $i+2$
lie in $\Delta(\lambda)$. We shall examine the case $(i-2,i-1)\in \Delta(\lambda)$, the other being very similar.

When $\lambda_{i-3} - \lambda_{i-2}$ is odd $s(\lambda^\ii) = s(\lambda) - 2$, $|\Delta(\lambda^\ii)| = |\Delta(\lambda)| - 1$
and $|\Delta_{\rm bad}(\lambda^\ii)| = |\Delta_{\rm bad}(\lambda)| - 2$ (since $(i-2,i-1)$ is no longer a bad 2-step after
this iteration). Furthermore $(i,i+1)$ cannot make up a 2-step in a good 2-cluster since $(i+2,i+3) \notin \Delta(\lambda)$
and $\lambda_{i+1} - \lambda_{i+2}$ is even, therefore $|\Sigma(\lambda)|$ remains unchanged. So consider the possibility
that $(i-2,i-1)$ has two bad boundaries: that $\lambda_{i-3} - \lambda_{i-2}$ is even. Then our conclusions are exactly
the same as before, except that $|\Delta_{\rm bad}(\lambda^\ii)| = |\Delta_{\rm bad}(\lambda)| - 1$. In either situation
$z(\lambda^\ii) \geq z(\lambda) - 1$.

Finally we have the situation $(i-2,i-1),(i+2,i+3) \in \Delta(\lambda)$. Once again we must distinguish between the number of bad
boundaries attached to the 2-steps $(i-2,i-1)$ and $(i+2,i+3)$. Suppose that both of these 2-steps have a single bad
boundary (they have at least 1). Then $i - 2, i, i+2$ is a good 2-cluster. It is immediately clear upon passing from
$\lambda$ to $\lambda^\ii$ that $s(\lambda^\ii) = s(\lambda) - 2$, $|\Delta(\lambda^\ii)| = |\Delta(\lambda)| - 1$,
$|\Delta_{\rm bad}(\lambda^\ii)| = |\Delta_{\rm bad}(\lambda)| - 3$, and $|\Sigma(\lambda^\ii)| = |\Sigma(\lambda)| - 1$.
Once again $z(\lambda^\ii) \geq z(\lambda) - 1$ follows. The last two situations to consider are when precisely one of the
two 2-steps $(i-2,i-1)$ and $(i+2, i+3)$ has two bad boundaries, and when both of them have two bad boundaries.

Take the former situation. We may assume that $(i-2, i-1)$ has two
bad boundaries, and $(i+2,i+3)$ has one (the opposite configuration
is similar). Upon iterating the algorithm, $s(\lambda^\ii) =
s(\lambda) - 2$, $|\Delta(\lambda^\ii)| = |\Delta(\lambda)| -1$ and
$|\Delta_{\rm bad}(\lambda^\ii)| = |\Delta_{\rm bad}(\lambda)| - 2$.
By Lemma~\ref{stairslemma}, $|\Sigma(\lambda^\ii)| \leq
|\Sigma(\lambda)|$. In the final case $(i-2,i-1)$ and $(i+2,i+3)$
both have two bad boundaries. The outcome is that $s(\lambda^\ii) =
s(\lambda) - 2$, $|\Delta(\lambda^\ii)| = |\Delta(\lambda)| - 1$ and
$|\Delta_{\rm bad}(\lambda^\ii)| = |\Delta_{\rm bad}(\lambda)| - 1$
both decrease by 1 and by Lemma~\ref{stairslemma} either
$|\Sigma(\lambda^\ii)| = |\Sigma(\lambda)|$ or
$|\Sigma(\lambda^\ii)| = |\Sigma(\lambda)| + 1$.

We have eventually shown that $z(\lambda) \geq z(\lambda^\ii) + 1$. Recall that for any maximal admissible sequence
$\ii$ the partition $\lambda^\ii$ is rigid. Also notice that $z(\lambda) = 0$ for any rigid partition $\lambda$. We
deduce for any maximal admissible sequence $\ii$ of length $l$, that
$$z(\lambda) \geq z(\lambda^{\ii_2}) + 1 \geq z(\lambda^{\ii_3}) + 2 \geq \cdots z(\lambda^{\ii_{l+1}}) + l = l.$$
Here $\ii_k$ denotes $(i_1,...,i_{k-1})$. In order to complete the proof we shall exhibit a maximal admissible sequence
of length $z(\lambda)$. This shall require some reductions.

Notice first that $z(\lambda)$ decreases by 1 at each iteration when
we apply Case 1 in constructing the shell $\lambda^S$. Therefore we
may assume that $\lambda = \lambda^S$. Let $\mu(1), \mu(2),...,
\mu(l)$ be a complete set of distinct profiles for $\lambda$, as in
the statement of Proposition~\ref{reductionprop}. By
Proposition~\ref{profileprop} we know that for each $1\leq m\leq l$
there is an admissible sequence of length $z(\mu(m))$ for $\mu(m)$.
Using Part (2) of Proposition~\ref{reductionprop} we obtain an
admissible sequence for $\lambda$ of length $\sum_{i=1}^l
z(\mu(i))$, and by Part~(1) of the same proposition that length is
equal to $z(\lambda)$. Hence a sequence of the correct length
exists, and the theorem follows.
\end{proof}

The following corollary shall be of some importance to our later work.
\begin{cor}\label{maxseq}
For all $\lambda \in \mathcal{P}_\epsilon(N)$ the following hold:
\begin{enumerate}
\item{$c(\lambda) \geq z(\lambda)$;}

\smallskip

\item{$c(\lambda) = z(\lambda)$ if and only if $\lambda$ is non-singular.}
\end{enumerate}
\end{cor}
\begin{proof}
Part (1) follows from the fact that $|\Delta_{\rm bad}(\lambda)| > |\Sigma(\lambda)|$ for all partitions
$\lambda$. For Part (2) we observe that $|\Delta_{\rm bad}(\lambda)| - |\Sigma(\lambda)| = 0$ if and only
if $\lambda$ is non-singular.
\end{proof}
\section{A Geometric Interpretation of the Algorithm}
We would like to characterise the non-singular partitions in geometric terms. This characterisation  shall
be given in the corollary to the next theorem. The remainder of this section shall be spent preparing to
prove that theorem. The symmetric group $\mathfrak{S}_l$ acts on the set of sequences in $\{1,...,n\}$ of
length $l$  by the rule $\sigma (i_1,...,i_l) = (i_{\sigma(1)},...,i_{\sigma(l)})$.
Let $$\Phi_\lambda := \{ \text{the maximal admissible sequences for }\lambda \} /\sim$$
where $\ii \sim \mathbf{j}$ if $\ii$ and $\jj$ have equal length $l$ and are $\mathfrak{S}_{l}$ conjugate.
What follows is the main theorem of this section.
\begin{thm}\label{nosheets}
The following are true for any $\lambda\in\mathcal{P}_\epsilon(N)$:
\begin{enumerate}
\item{$e(\lambda)$ lies in $|\Phi_\lambda|$ distinct sheets;}

\smallskip

\item{$|\Phi_\lambda| = 1$ if and only if $\lambda$ is non-singular.}
\end{enumerate}
\end{thm}

The next corollary explains our choice of terminology.
\begin{cor}\label{izosim} Suppose $\lambda\in\mathcal{P}_\epsilon(N)$ and
$\dim \h_{e(\lambda)}=m$.
Then the following are equivalent:
\begin{enumerate}
\item{the partition $\lambda$ is non-singular;}

\smallskip

\item{$c(\lambda) = z(\lambda)$;}

\smallskip

\item{$e(\lambda)$ lies in a unique sheet;}
\end{enumerate}
If the base field $\K$ has characteristic $0$ or
\emph{char}$(\K)=p\gg 0$ then $(1)$, $(2)$ and $(3)$ hold if and
only if $e(\lambda)$ is a non-singular point on the quasi-affine
variety $\h^{(m)}$.
\end{cor}
\begin{proof}
Statements (1), (2) and (3) are equivalent by Theorems \ref{zismax},
Corollary~\ref{maxseq} and Theorem~\ref{nosheets}. Now suppose that
the characteristic of $\K$ is either zero or char$(\K)=p\gg0$.  Then
Im Hof proved in \cite[Chapter 6]{Im} that all sheets of $\h^{(m)}$
are smooth algebraic varieties. (Im Hof assumes in {\it op.\,cit.}
that char$(\K)=0$, but his arguments extend easily to the case where
char$(\K)$ is sufficiently large). In view of our discussion in
Subsection~\ref{2.1}, Im Hof's result implies that all irreducible
components of $\h^{(m)}$ are smooth algebraic varieties. In this
situation it follows from \cite[Chapter II, \S2, Theorem 6]{Sh} that
$e=e(\lambda)$ is a non-singular point of the algebraic variety
$\h^{(m)}$ if and only if $e$ belongs to a unique irreducible
component of $\h^{(m)}$. This completes the proof.
\end{proof}

We shall now assemble all of the necessary information required to
prove Theorem~\ref{nosheets}. We start by recalling some facts
regarding sheets and induced orbits.  A short survey of these topics
can be found in \cite{Mor}. For a full discussion in the
characteristic zero case see \cite{CM} or \cite{TY}. Since every
sheet of $\h$ contains a dense decomposition class we have the
following:.
\begin{thm}{\rm (See \cite{Bor}.)}\label{classsheets}
There is a 1-1 correspondence between the set of sheets of $\h$ and
the $K$-conjugacy classes of pairs $(\li, \Oo_\li)$ where $\li$ is a
Levi subalgebra of $\h$ and $\Oo_\li$ is a rigid nilpotent orbit in
$\li$.
\end{thm}
We shall say that a sheet $\mathcal{S}$ of $\h$ \emph{has data}
$(\li, \Oo_\li)$ if $\mathcal{S}$ is identified with $(\li,
\Oo_\li)$ under the above correspondence. In view of our discussion
in Subsection~\ref{2.1} this means that $\mathcal S$ contains an
open decomposition class of the form $\mathcal{D}(\li,e_0)$ with
$e_0\in\Oo_\li$.

Let $\mathfrak{p} = \mathfrak{l} \oplus \mathfrak{n}$ be a parabolic
subalgebra of $\h$ with $\li$ a Levi subalgebra of $\h$ and
$\mathfrak{n}$ the nilradical of $\mathfrak{p}$. Let $\Oo_\li$ be a
nilpotent orbit in $\li$. Since the orbit set $\mathcal{N}/K$ is
finite there exists a unique nilpotent orbit in $\h$ which meets the
irreducible quasi-affine variety $\Oo_\li + \mathfrak{n}\subset
\mathcal{N}(\h)$ in a dense open subset. This orbit, denoted by {\rm
Ind}$_\li^\h(\Oo_\li)$, is said to be {\it induced} from the orbit
$\Oo_\li$.

We record three pieces of information regarding induced orbits.
\begin{prop}\label{ind} {\rm (See \cite{LS}, \cite{BKr}, \cite{Bor}.)}
\label{inducedprops}
The following are true:
\begin{enumerate}
\item{If $\mathcal{S}$ is a sheet with data $(\li, \Oo_\li)$ then \emph{Ind}$_\li^\h(\Oo_\li)$ is
the unique nilpotent orbit contained in $\mathcal{S}$;}

\smallskip

\item{For each nilpotent orbit $\Oo \subseteq \h$ we have that $\Oo=\text{\emph{Ind}}_{\h}^\h(\Oo);$}

\smallskip

\item{If $\li_1$ and $\li_2$ are Levi subalgebras of $\h$, $\Oo$ is a
nilpotent orbit in $\li_1$ and $\li_1 \subseteq \li_2$,
then$$\text{\emph{Ind}}_{\li_2}^\h(\text{\emph{Ind}}_{\li_1}^{\li_2}(\Oo)
) = \text{\emph{Ind}}_{\li_1}^\h(\Oo).$$      }
\end{enumerate}
\end{prop}

Fix a partition $\lambda \in \mathcal{P}_\epsilon(N)$. We aim to
classify the $K$-conjugacy classes of pairs $(\li, \Oo)$ where
$\li\subseteq \h$ is a Levi subalgebra and $\Oo \subseteq \li$ is a
rigid nilpotent orbit, such that $\Oo_{e(\lambda)} =
\Ind_{\li}^{\h}(\Oo)$. In view of Part~(1) of the above proposition
this shall parameterise the set of sheets containing $e(\lambda)$.
In order to begin this classification we shall require some general
facts about Levi subalgebras of $\h$.

Every Levi subalgebra is conjugate to a standard Levi subalgebra. If
$\t\subset \h$ is a maximal torus and $\Pi$ a fixed basis of simple
roots associated with $(\h, \t)$ then a standard Levi subalgebra is
constructed from a subset $\Pi_0 \subseteq \Pi$. To each such $\Pi_0$ we
attach the Levi subalgebra $\li$ generated by $\t$ and the roots
spaces $\h_{\pm \gamma}$ with $\gamma \in S$. Now order the simple
roots in $\Pi$ in the usual manner and let $\ii = (i_1,...,i_l)$ be
a sequence with $\sum_j i_j \leq \rank\,\h$. Such sequences are in a
bijection with the subsets of $\Pi$ by letting $\Pi_\ii = \Pi
\backslash \{\alpha_{i_1+\cdots+i_k}\colon\, 1\le k\le l\}$. It is
easy to check that in types $\sf B$ and $\sf C$ the standard Levi
subalgebra constructed from $\Pi_\ii$ is isomorphic to $\gl_{i_1}
\times\cdots\times \gl_{i_l} \times \mm$ where $\mm$ is a classical
algebra. If $\sum_{j} i_j = \rank\,\h
 - 1$ in type $\sf D$ then the Levi subalgebra constructed from $\Pi_\ii$ is actually isomorphic
 to $\gl_{i_1} \times \cdots \times\gl_{i_{l-1}} \times \gl_{i_l + 1}$. If we define another
 sequence $\ii' = (i_1,...,i_{l-1}, i_l+1)$ then the Levi subalgebras constructed from $\ii$
 and $\ii'$ are isomorphic. When all terms of $\ii'$ are even these standard Levi subalgebras are
 not conjugate and we shall label their respective conjugacy classes $I$ and $I\!I$.

 When we refer to a Levi subalgebra by its isomorphism type
 we shall implicitly be referring to a standard Levi subalgebra constructed from a subset of $\Pi$.
Let us record these conclusions formally.
\begin{lem}{\rm (See \cite{K}, \cite{CM},  \cite{Mor}.)} \label{levnilorbs}
The following are true:
\begin{enumerate}
\item{Every Levi subalgebra of $\h$ is
$K$-conjugate to a Lie algebra of the form $$\gl_\ii \times
\mathfrak{m} := \gl_{i_1} \oplus \cdots \oplus \gl_{i_l} \oplus
\mathfrak{m}\,\cong\, \gl_{i_1} \times \cdots \times \gl_{i_l}
\times \mathfrak{m}$$ where $\ii = (i_1,...,i_l)$ is a sequence of
integers with $\sum_j i_j \leq {\rm rank}\,\h$ and where
$\mathfrak{m}$ has the same type as $\h$ and a standard
representation of dimension $R_\ii := N - 2\sum_j i_j$ (under the
restriction that $R_\ii \neq 2$ if $\epsilon = 1$). If $\h$ has type
$\sf D$, $R_\ii = 0$ and all parts of $\ii$ are even then there are
two $K$-conjugacy classes of Levi subalgeras isomorphic to $\gl_\ii
\times \mathfrak{m}$. They are assigned labels $I$ and $I\!I$.
Otherwise there is a unique $K$-conjugacy class of Levi subalgebras
isomorphic to $\gl_\ii \times \mathfrak{m}$.}

\smallskip

\item{If $\li$ is a Levi subalgebra as in Part~1 then the rigid nilpotent
orbits in $\li$ take the form $$\Oo = \underbrace{\Oo_0 \times \cdots \times \Oo_0}_{l \text{ times}}
\times \Oo_{e(\mu)}$$ with $\mu \in \mathcal{P}_\epsilon^\ast(N - 2\sum_j i_j)$ a rigid partition.}
\end{enumerate}
\end{lem}
Let $\Psi_\lambda$ denote the set of all
$K$-conjugacy classes of pairs $(\li, \Oo)$ where $\li = \gl_{i_1} \oplus \cdots
\oplus\gl_{i_l}\oplus \mathfrak{m}\,\cong\,\gl_{i_1} \times \cdots \times \gl_{i_l}
\times \mathfrak{m}$ is a Levi subalgebra of $\h$ and $\Oo = \Oo_0 \times \cdots \times
\Oo_0 \times \Oo_{e(\mu)}$ a rigid nilpotent orbit in $\li$, such that $\Oo_{e(\lambda)} = \Ind_\li^\h(\Oo).$
\begin{lem}\label{contsheets}
$e(\lambda)$ lies in $|\Psi_\lambda|$ distinct sheets.
\end{lem}
\begin{proof}
Let $\mathcal{S}$ be a sheet of $\h$ with data $(\li, \Oo)$. By Theorem~\ref{classsheets} and
Part~1 of Proposition~\ref{inducedprops} we see that $e(\lambda) \in \mathcal{S}$ if and only
if $e(\lambda) = \Ind_\li^\h(\Oo)$. By
Lemma~\ref{levnilorbs} we have
$\li \cong \gl_{i_1} \times \cdots \times \gl_{i_l} \times \mathfrak{m}$
and $\Oo = \Oo_0 \times\cdots \Oo_0 \times \Oo_{e(\mu)} \subseteq \li$.
\end{proof}

We now briefly discuss the partitions associated to induced orbits. The result
stated below may be deduced from \cite[Corollary~7.3.3]{CM}. We warn the reader
that when interpreting the proposition for the Lie algebras of type $\sf B$ the
unique nilpotent orbit in the trivial algebra $\mathfrak{so}_1$ is labelled by
the partition $\lambda = (1)$ contrary to the common convention. Furthermore,
our description of labels attached to induced orbits does not quite agree with
the description in \cite{CM}; see Remark~\ref{amendment} for more detail.

Recall that the natural representation of $\h$ is of dimension $N$.
\begin{prop}\label{indpart}
Choose a positive integer $ i$ with $2i \leq N$ and let $\li
=\gl_{i} \oplus \mathfrak{m}\cong \gl_{i} \times \mathfrak{m}$ be a
maximal Levi subalgebra of $\h$. Let $\Oo = \Oo_0\times \Oo_{\mu}$
be a nilpotent orbit in $\li$ where $\Oo_\mu$ has partition $\mu \in
\mathcal{P}_\epsilon(N-2i)$. Then $\Oo_{e(\lambda)} =
\emph{\Ind}_\li^\g(\Oo)$ has associated partition $\lambda$ where
$\lambda$ is obtained from $\mu$ by the following procedure: add 2
to the first $i$ columns of $\mu$ (extending by zero if necessary);
if the resulting partition lies in $\mathcal{P}_\epsilon(N)$ then we
have found $\lambda$, otherwise we obtain $\lambda$ by subtracting 1
from the $i^\text{th}$ column and adding 1 to the $(i+1)^\text{th}$.

Now suppose we are in type $\sf D$ and $\lambda$ is very even. Then
either $\mu$ is very even or $N=2i$ and ${\rm rank}\,\h$ is even. If
$N>2i$ then $\Oo_{e(\lambda)}$ inherits its label from $\mu$, whilst
if $N=2i$ then the induced orbit inherits its label from $\li$.
\end{prop}
\begin{rem}\label{amendment}
\rm{The above proposition is based on \cite[Theorem~7.3.3]{CM}.
However the reader will notice that the way in which the labels are
chosen does not coincide with that theorem. The reason for this is
that \cite{CM} contains two misprints which we must now
amend\footnote[1]{We are thankful to Monty McGovern for this
clarification.}.

The first problem stems from comparing Lemma~5.3.5 and Theorem~7.3.3(ii) in
\cite{CM}. We see, given the conventions of Lemma~5.3.5 in  {\it
op.\,cit}, that \cite[Theorem~7.3.3(ii)]{CM} should actually state
that the label of $\Ind^\h_{\gl_i\oplus\mathfrak{m}}(\Oo)$ is {\it
different} to the label of $\Oo$ when $(\rank\,\h + \rank\,\mm)/2$
is odd. We could, of course, change \cite[Theorem~7.3.3(ii)]{CM} but a
better amendment is to change \cite[Lemma~5.3.5]{CM} so that the
labelling convention for very even orbits is independent of $n$: in
the notation of {\it op.\,cit.} we take $a = 2$ and $b = 0$
regardless of $n$. With this convention the statement of
\cite[Theorem~7.3.3(ii)]{CM} is correct. However
\cite[Theorem~7.3.3(iii)]{CM} should now state that the label of the
induced orbit {\it coincides} with the label of a Levi subalgebra from which it
is induced. This is the convention we have taken in the above
proposition.

The second misprint regards the number of conjugacy classes of
maximal Levi subalgebras in \cite[Lemma~7.3.2(ii)]{CM}. The reader
will notice that when $\h = \mathfrak{so}_{2\ell}$ and $\ell$ is odd,
the longest element of the Weyl group $w_0$ is the negative of the
outer Dynkin automorphism of the root system. Therefore if $gT = w_0
\in W = N_K(T)/T$ then $\Ad\,g$ exchanges the Levi subalgebras which
are labelled $I$ and $I\!I$ in this case. This confirms that there
is just one class of Levi subalgebras of type $\gl_\ell$ when $\ell$
is odd. When $\ell$ is even there are two such classes and our
convention for labelling conjugacy classes of Levi subalgebras in
Lemma~\ref{levnilorbs} is a natural extension of
\cite[Lemma~7.3.2]{CM}.}
\end{rem}

In light of the above proposition we may explain the definition of
the algorithm. We fix $\lambda$ and want to decide when is it
possible to find a pair consisting of a maximal Levi $\li =
\gl_{i_1}\oplus \mathfrak{m}\cong\gl_{i_1}\times \mathfrak{m}$ and a
nilpotent orbit $\Oo = \Oo_0 \times \Oo_{e(\mu)}$ (with partition
$\mu$) such that $\Ind_\li^\g(\Oo) = \Oo_{e(\lambda)}$. It is now
clear that this occurs precisely when we have an admissible index
$i$ and a Levi subalgebra isomorphic to $\gl_i \times \mm$. In this
case $\mu = \lambda^{(i)}$ and if $\Oo_{e(\mu)}$ has a label then it
is completely determined by that of $\Oo_{e(\lambda)}$. The precise
statement is as follows:
\begin{cor}\label{case1or2}
Let $\lambda \in \mathcal{P}_\epsilon(N)$. Suppose there exists a
maximal Levi $\li \cong\gl_{i}\oplus \mathfrak{m}\cong \gl_{i} \times
\mathfrak{m}$. Then the following are equivalent:
\begin{enumerate}
\item{$i$ is an admissible index for $\lambda$. If $\h$ has type $\sf D$
and there are two conjugacy classes of Levi subalgebras isomorphic
to $\gl_i \times \mm$ then $\li$ belongs to the conjugacy class with
the same label as $\Oo_{e(\lambda)}$;}
\item{There exists an orbit $\Oo = \Oo_0 \times \Oo_{e(\mu)}$ with $\Oo_{e(\lambda)} = \Ind_\li^\h(\Oo)$.}
\end{enumerate}
If these two equivalent conditions hold then $\Oo_{e(\mu)}$ has
partition $\mu = \lambda^{(i)}$. Furthermore, for every other orbit
$\tilde{\Oo} = \Oo_0\times \Oo_{e(\tilde{\mu})}$ in $\li$ such that
$\Oo_{e(\lambda)} = \Ind_\li^\h(\tilde{\Oo})$, we have $(\li, \Oo)/K
= (\li, \tilde{\Oo})/K$.
\end{cor}
\begin{proof}
Assume throughout that $\li \cong \gl_i \times \mm$ exists and let
$\Oo = \Oo_0 \times \Oo_{e(\mu)} \subseteq \li$. The previous
proposition implies that if $\lambda$ is the partition of
$\Ind_\li^\h(\Oo_0 \times \Oo_{e(\mu)})$ then $\lambda^{(i)} = \mu$.
Suppose (1) holds. Then the partition of $\Ind_\li^\h(\Oo_0 \times
\Oo_{e(\lambda^{(i)})})$ is $\lambda$. If $\lambda$ is not very even
then (2) follows. If we are in type $\sf D$ and $\lambda$ is very
even then according to the previous proposition either
$\lambda^{(i)}$ is very even or $\li \cong \gl_i$ where $i = N/2 =
\rank\,\h$ is even. In the first case the orbit $\Oo_0 \times
\Oo_{e(\lambda^{(i)})}$ with the same label as $\Oo_{e(\lambda)}$
induces to $\Oo_{e(\lambda)}$ whilst in the second case there is a
unique orbit of the correct form (the zero orbit) and since the
labels of $\li$ and $\Oo_{e(\lambda)}$ coincide, it induces up to
$\Oo_{e(\lambda)}$.

Now suppose that (2) holds. Since $\mu = \lambda^{(i)}$ the index
$i$ is certainly admissible for $\lambda$. If there are two
conjugacy classes of Levi subalgebras then again $\li \cong \gl_i$
and so $\Oo_{e(\lambda)} = \Ind_\li^\h(\Oo)$ implies that the labels
of $\li$ and $\Oo_{e(\lambda)}$ coincide by the last part of the
previous proposition.

The statement that $\mu = \lambda^{(i)}$ is immediate from the above
discussion. Fix $\Oo = \Oo_0\times\Oo_{e(\mu)}$ fulfilling
$\Oo_{e(\lambda)} = \Ind_\li^\h(\Oo)$. We must show that for every
other orbit of the form $\tilde{\Oo} = \Oo_0
\times\Oo_{e(\tilde{\mu})}$ fulfilling $\Oo_{e(\lambda)} =
\Ind_\li^\h(\tilde{\Oo})$ that the pair $(\li, \tilde{\Oo})$ is
$K$-conjugate to $(\li, \Oo)$. Since we know that $\mu =
\lambda^{(i)} = \tilde{\mu}$ this is now obvious unless $\mu$ is
very even and $\lambda$ is not very even.

So suppose that this is the case. We claim that in this situation
any admissible sequence $\ii$ for $\lambda$ has an odd term. Indeed,
in order to see this it suffices to assume $\ii = (i)$ has length 1.
Since $\lambda^\ii$ is very even or empty we conclude that $(i,
i+1)$ must be the only 2-step for $\lambda$. If $\lambda^\ii$ is
empty then $i = 1$. Assume not. Since the parts of $\lambda$ which
precede $\lambda_i$ are all even they must come in pairs and so $i$
must be odd. The claim follows.

Since we are assuming that $\mu$ is very even and $\lambda$ is not,
the above shows that $i$ is odd. We know that $\rank\,\mm = (N -
2i)/2$ is even. From this we can be sure that $N/2 = \rank\,\h$ is
odd. Now from the tables in \cite{B} we see that the longest element
$w_0$ of the Weyl group $W = N_K(T)/T$ is the negative of the outer
diagram automorphism of the root system of $\h$. Therefore if $gT =
w_0$ then $\Ad\,g$ will preserve $\li$ and exchange the orbits with
partition $\lambda^{(i)}$ labelled $I$ and $I\!I$. This complete the
proof.
\end{proof}

The following proposition uses a similar kind of induction as
\cite[Proposition~3.7]{Mor} and is central to our proof of Theorem
\ref{nosheets}.
\begin{prop}\label{mainpropref}
Let $\ii = (i_1,...,i_l)$ be a sequence of integers with $\sum_j i_j
\leq {\rm rank}\,\h$. Suppose there exists a Levi subalgebra $\li
\cong \gl_\ii \times \mathfrak{m}$. Then following are equivalent:
\begin{enumerate}

\item{$\ii$ is an admissible sequence for $\lambda$. If $\h$ has type $\sf D$ and there are two conjugacy classes of Levi
subalgebras isomorphic to $\gl_\ii \times \mm$ then $\li$ belongs to the conjugacy class with the same label
as $\Oo_{e(\lambda)}$;}

\smallskip

\item{There exists an orbit $\Oo = \Oo_0 \times \cdots\times\Oo_0 \times \Oo_{e(\mu)}$ with
$\Oo_{e(\lambda)} = {\rm Ind}_\li^\h(\Oo)$.}
\end{enumerate}
If these two equivalent conditions hold then $\Oo_{e(\mu)}$ has partition
$\mu = \lambda^{\ii}$. Furthermore, for every other orbit $\tilde{\Oo} \subseteq \li$
with $\tilde{\Oo} = \Oo_0 \times\cdots\times\Oo_0\times \Oo_{e(\tilde{\mu})}$
 such that $\Oo_{e(\lambda)} = \Ind_\li^\h(\tilde{\Oo})$, we have $(\li, \Oo)/K = (\li, \tilde{\Oo})/K$.
\end{prop}

\begin{proof}
The proof proceeds by induction on $l$. When $l = 0$ we have $\li =
\h$ and the proposition holds by Part~2 of
Proposition~\ref{inducedprops} (note that $\lambda^\emptyset =
\lambda$). If $\li$ is a proper Levi subalgebra then $l > 0$. The
case $l = 1$ is simply the previous corollary. The inductive step is
quite similar although to begin with we must exclude the possibility
that $R_\ii = 0$ and $N - 2\sum_{j=1}^{l-1} i_j= 2$ in type $\sf D$.
We will treat this possibility at the end.

Suppose that the proposition has been proven for all $l'< l$. Since
we have excluded this anomalous case in type $\sf D$ we may set
$\ii' = (i_1,...,i_{l-1})$ and deduce that there exists a Levi $\li'
\cong \gl_{\ii'} \oplus \mm'$ where $\mm'$ has a natural
representation of dimension $R_{\ii'}$ and the same type as $\h$.
Let $M'$ be the closed subgroup of $K$ with $\mm' = \Lie(M')$. We
may ensure that $\li \subseteq \li'$ by embedding $\gl_{i_l} \oplus
\mm$ in $\mm'$.

Suppose that $\ii$ is admissible and, if possible, that the label of
$\li$ coincides with that of $\Oo_{e(\lambda)}$. We deduce that
$\ii'$ is also admissible, and since $R_{\ii'} > 0$ there is a
unique class of Levi subalgebras isomorphic to $\gl_{\ii'} \oplus
\mm'$. By the inductive hypothesis we deduce that there exists an
orbit $\Oo' = \Oo_0 \times \cdots\times\Oo_0 \times \Oo_{e(\tau)}
\subseteq \li'$ with $\Oo_{e(\lambda)} = \Ind_{\li'}^\h(\Oo')$. We
also see that this orbit is unique, that it has partition
$\lambda^{\ii'}$ and that if it has a label then it is the same as
$\Oo_{e(\lambda)}$. Clearly $i_l$ is admissible for $\lambda^{\ii'}$
and examining our labelling conventions for Levi subalgebras described
preceding Lemma~\ref{levnilorbs} we see that the label of the $K$-conjugacy
class of $\li$ equals the label of the $M'$-conjugacy class of
$\gl_{i_l} \oplus \mm \subseteq \mm'$. Therefore we can apply
Corollary~\ref{case1or2} to conclude that there exists an orbit $\Oo
= \Oo_0 \times \Oo_{e(\mu)}$ with $\Oo_{e(\tau)} =
\Ind^{\mm'}_{\gl_{i_l} \oplus \mm}(\Oo)$. We make use of Proposition
\ref{inducedprops} in the following calculation:
\begin{eqnarray*}
\Oo_{e(\lambda)} = \Ind_{\li'}^\h(\Oo') &= &\Ind^\h_{\li'}(\Oo_0
\times\cdots\times\Oo_0\times \Ind^{\mm'}_{\gl_{i_l} \oplus
\mm}(\Oo_0\times \Oo_{e(\mu)}))\\& = &
\Ind^\h_{\li}(\Oo_0\times\cdots \times \Oo_0 \times \Oo_{e(\mu)})
\end{eqnarray*}
We have shown that $(1) \Rightarrow (2)$. Before proving $(2)
\Rightarrow (1)$ we shall take a quick detour to show that the final
remarks in the statement of the proposition follow from (1). We
certainly have $\mu = \tau^{(i_l)} = (\lambda^{\ii'})^{(i_l)} =
\lambda^\ii$ by the transitivity of the algorithm. Suppose
$\tilde{\Oo} = \Oo_0 \times\Oo_0\times\Oo_{e(\tilde{\mu})}$ is
another orbit in $\li$ inducing to $\Oo_{e(\lambda)}$. By the
inductive hypothesis the partition of $\Ind^{\mm'}_{\gl_{i_l}\oplus
\mm}(\Oo_0 \times \Oo_{e(\tilde{\mu})})$ is $\lambda^{\ii'}$ and so
we get $\tilde{\mu} = \lambda^\ii = \mu$. The uniqueness assertion
is therefore obvious unless $\mu$ is very even and $\lambda$ is not.
In this case the argument used in the proof of
Corollary~\ref{case1or2} tells us that some term of the sequence
$\ii$ is odd. After conjugating by some element of $K$ we can assume
that $i_l$ is odd. The proof of uniqueness then concludes just as
with the previous corollary, with $\gl_{i_l} \oplus \mm$ playing the
role of our Levi subalgebra and $\mm'$ playing the role of $\h$.

Now we must go the other way. Keep $\li$, $\li'$, $\mm'$, etc as
above. Suppose that there exists an orbit $\Oo = \Oo_0 \times
\cdots\times\Oo_0 \times \Oo_{e(\mu)}\subseteq \li$ with
$\Oo_{e(\lambda)} = \Ind_\li^\h(\Oo)$. Then we set $\Oo_{e(\tau)} :=
\Ind_{\gl_{i_l} \oplus \mm}^{\mm'}(\Oo_0 \times \Oo_{e(\mu)})$,
$\Oo' := \Oo_0 \times\cdots \times\Oo_0\times \Oo_{e(\tau)}
\subseteq \li'$ and conclude that $\Oo_{e(\lambda)} =
\Ind^\h_{\li'}(\Oo')$ using a calculation very similar to the above
one. Applying the inductive hypotheses we get that $\ii'$ is
admissible. There is no label associated to $\li'$ since $R_{\ii'} >
0$. Now Corollary~\ref{case1or2} tells us that $i_l$ is an
admissible index for $\tau = \lambda^{\ii'}$ and so $\ii$ is
admissible for $\lambda$. The same corollary tells us that if the
$M'$-conjugacy class of the Levi subalgebra $\gl_{i_l}\oplus \mm\subseteq \mm'$
has a label then it coincides with that of $\Oo_{e(\tau)}$. The
inductive hypothesis tells us that this label coincides with that of
$\Oo_{e(\lambda)}$.

Finally we must turn our attention to those sequences $\ii$ in type
$\sf D$ for which $R_{\ii'} = 2$ (as before $\ii'$ stands for $\ii$
with the last term removed). In this case there does not exist a
Levi subalgebra of the form $\gl_{\ii'}\oplus \mm$ and the induction
falls down. In order to resolve this we define $\ii'' =
(i_1,...,i_{l-2})$ and let $\li'' = \gl_{\ii''} \oplus \mm''$. Since
$\li$ has the form $\gl_\ii$ we may embed $\gl_{i_{l-1}}\oplus
\gl_{i_l}\subseteq \mm''$ to get $\li \subseteq \li''$. Since $i_l =
1$ there is a unique conjugacy class of Levi subalgebras isomorphic
to $\gl_\ii$. Furthermore, since the $\mm$ part is zero, there is
only one orbit of the prescribed form in $\li$. We let $\Oo$ equal
the zero orbit in $\li$. The proposition in this case is therefore
reduced to the statement that $\ii$ is admissible if and only if
$\Oo_\lambda = \Ind_\li^\h(\Oo)$.

Suppose $\ii$ is admissible for $\lambda$. Then so is $\ii''$ and by
the inductive hypothesis there exists an orbit $\Oo'' =
\Oo_0\times\cdots\times\Oo_0\times \Oo_{e(\tau)}$ in $\li''$ with
$\Oo_{e(\lambda)} = \Ind^\h_{\li''}(\Oo'')$. Since $i_{l-1}$ is an
admissible index for $\tau$ and $\tau^{(i_{l-1})}$ is $(1,1)$ we
conclude that $\tau = (3,1)$ if $i_l = 1$, that $\tau = (3,3)$ if
$i_l = 2$, that $\tau = (3,3,2^{i_{l-1}})$ if $i_{l-1}> 2$ is even
or finally that $\tau = (3,3,2^{i_{l-1}-1},1,1)$ if $i_{l-1}>2$ is
odd. None of these partitions are very even and so there is a unique
orbit with partition $\tau$. According to \cite[Theorem~7.2.3]{CM}
the induced orbit $\Ind_{\gl_{i_{l-1}}\oplus
\gl_{i_l}}^{\gl_{i_{l-1}+1}}(\Oo_0\times\Oo_0)$ is the minimal
nilpotent orbit in $\gl_{i_{l-1}+1}$ with partition $(2,1,...,1)$.
If we induce into $\mm''$ then \cite[Lemma~7.3.3(i)]{CM} tells us
that $\Ind_{\gl_{i_{l-1}+1}}^{\mm''}(\Oo_\text{min}) =
\Oo_{e(\tau)}$. Placing these ingredients together we get
\begin{eqnarray*}
\Ind^\h_{\li}(\Oo) = \Ind^\h_{\li''}(\Ind^{\li''}_{\li}(\Oo)) &=&
\Ind^\h_{\li''}\big(\Oo_0\times
\cdots \times\Oo_0 \times \Ind^{\mm''}_{\gl_{i_{l-1}}\oplus\gl_{i_l}}(\Oo_0 \times \Oo_0) \big)\\
&=& \Ind^\h_{\li''}(\Oo'') = \Oo_{e(\lambda)}
\end{eqnarray*}
as required. To go the other way we assume that such an orbit $\Oo$
exists and go backwards through the above deductions. We will
conclude that $\tau$ has one of the prescribed forms so that
$(i_{l-1}, 1)$ is an admissible sequence for $\ii''$ and conclude
that $\ii$ is admissible.
\end{proof}

\begin{cor}\label{tau}
Let $\lambda\in\mathcal{P}_\epsilon(N)$. If
$\emph\ii=(i_1,\ldots,i_l)$ is an admissible sequence for $\lambda$
(and $R_\ii \neq 2$ in type $\sf D$) then so is $\sigma(\emph\ii)$
for every $\sigma \in \mathfrak{S}_l$.
\end{cor}
\begin{proof}
Since $\gl_{\ii} \times \mathfrak{m}\,\cong\,
\gl_{\sigma(\ii)} \times \mathfrak{m}$, this is immediate from Proposition \ref{mainpropref}.
\end{proof}

We now define a partial function $\varphi$ from the set of all admissible
sequences for $\lambda$ to the set of all $K$-orbits of pairs $(\li, \Oo)$
 where $\li \subseteq \h$ is a Levi subalgebra of $\h$ and $\Oo \subset \li$
 is a nilpotent orbit. The map will remain undefined on sequences $\ii$
 with $R_\ii = 2$ in type $\sf D$. Let $\ii = (i_1,...,i_l)$ be an admissible
 sequence for $\lambda$. Let $\li$ be a Levi subalgebra isomorphic to
 $\gl_\ii \times \mm$. If there are two $K$-conjugacy classes of such Levi subalgebras
 then $\lambda$ is very even and we request that $\li$ has the same label
 as $\Oo_{e(\lambda)}$. Let $\varphi(\ii) = (\li, \Oo)/K$ be the unique pair
 described in Proposition~\ref{mainpropref}. The reader should take note
 that the admissible sequences upon which $\varphi$ is undefined are not
 maximal so the following makes sense.
\begin{cor}\label{bijection}
The restriction of $\varphi$ to the set of maximal admissible
sequences for $\lambda$ descends to a well
 defined bijection from $\Phi_\lambda$ onto $\Psi_\lambda$. In particular
 $|\Phi_\lambda| = |\Psi_\lambda|$.
\end{cor}
\begin{proof}
First of all, we show that $\varphi$ maps the set of maximal
admissible sequences for $\lambda$ to $\Psi_\lambda$. Take $\ii$
maximal admissible and $\varphi(\ii) = (\li, \Oo)/K$ with $\Oo =
\Oo_0\times\cdots\times\Oo_0\times \Oo_{e(\mu)}$. By
Proposition~\ref{mainpropref} we have $\mu= \lambda^\ii$ and so by
Lemma \ref{maxadmiss} $\Oo_{e(\mu)}$ is rigid. By Part~2 of
Lemma~\ref{levnilorbs} the orbit $\Oo$ is also rigid. Furthermore we
have that $\Oo_{e(\lambda)} = \Ind_\li^\h(\Oo)$. Hence $\varphi(\ii)
\in \Psi_\lambda$.

We now claim that the map is well defined on $\Phi_\lambda$, that is
to say that $\varphi(\ii) = \varphi(\jj)$ whenever $\ii \sim \jj$.
Let $\varphi(\ii) = (\li_1, \Oo_1)/K$ and $\varphi(\jj) = (\li_2,
\Oo_2)/K$ where $\li_1 \cong \gl_\ii \times \mm_1$ and $\li_2 \cong
\gl_\jj \times \mm_2$. Since $\ii = \sigma(\jj)$ for some $\sigma
\in \mathfrak{S}_{|\ii|}$ and the labels of $\li_1$ and $\li_2$ are
the same (if they exist), we conclude that they are $K$-conjugate by
Part~1 of Lemma~\ref{levnilorbs}. Thus we may assume that $\li_1 =
\li_2$. Now the uniqueness statement at the end of
Proposition~\ref{mainpropref} asserts that $(\li_1, \Oo_1)/K =
(\Oo_2, \li_2)/K$. For the rest of the proof $\varphi$ shall denote
the induced map $\Phi_\lambda \ra \Psi_\lambda$.

Let us prove that $\varphi$ is surjective. Suppose $(\li, \Oo) \in
\Psi_\lambda$ with $\li$ and $\Oo$ as in the definition of
$\Psi_\lambda$. Then by Proposition~\ref{mainpropref} the sequence
$\ii = (i_1,...,i_l)$ is admissible for $\lambda$ and by
Lemma~\ref{maxadmiss} it is a maximal admissible. Therefore
$\varphi(\ii) = (\li, \tilde{\Oo})/K$ for some orbit $\tilde{\Oo}$.
Since $\Oo = \Oo_0 \times \cdots \times \Oo_0 \times \Oo_{e(\mu)}$
by construction, the uniqueness statement in
Proposition~\ref{mainpropref} tells us that $(\li, \Oo)/K = (\li,
\tilde{\Oo})/K$. Hence $\varphi$ sends the equivalence class of
$\ii$ in $\Phi_\lambda$ to $(\li, \Oo)/K$.

In order to prove the corollary we must show that $\varphi$ is
injective. Suppose that $\ii$ and $\jj$ are maximal admissible for
$\lambda$ and $\varphi(\ii) = \varphi(\jj)$. Again we make the
notation $\varphi(\ii) = (\li_1, \Oo_{e(\lambda^\ii)})/K$ and
$\varphi(\jj) = (\li_2, \Oo_{e(\lambda^\jj)})/K$. Since $\li_1$ and
$\li_2$ are $K$-conjugate the sequences $(i_1, ..., i_{l(1)})$ and
$(j_1,...,j_{l(2)})$ corresponding to isomorphisms $$\li_{1} \cong
\gl_{i_1} \oplus \cdots \oplus \gl_{i_{l(1)}}\oplus
\mathfrak{m}_{\ii}\,\cong\,\gl_{j_1} \oplus \cdots \oplus
\gl_{j_{l(2)}}\oplus \mathfrak{m}_{\jj}\cong\li_2$$ must be of the
same length $l=l(1)=l(2)$ and $\mathfrak{S}_l$-conjugate. This
completes the proof.
\end{proof}
Part~1 of Theorem~\ref{nosheets} follows quickly from the above and
Lemma~\ref{contsheets}. We now prepare to prove Part~2 of that
theorem. Before we proceed we shall need two lemmas.  Define a
function $\kappa : \PP(N) \longrightarrow (\Z/2\Z)^\N$ by setting
$$\kappa(\lambda)_i := \lambda_i - \lambda_{i+1} \mod 2 \quad\text{ for all }\, i>0.$$
The reader should keep in mind here that $\lambda_i = 0$ for all $i > n$ by convention.
\begin{lem}\label{kappa}
Let $M,N \in \N$. If $ \mu \in \mathcal{P}^\ast_\epsilon(M)$ and
$\lambda \in \mathcal{P}^\ast_\epsilon(N)$ then $\mu = \lambda$ if
and only if $\kappa(\mu) = \kappa(\lambda)$.
\end{lem}
\begin{proof}
Evidently $\mu = \lambda$ if and only if $\mu_i - \mu_{i+1} =
\lambda_i - \lambda_{i+1}$ for all $i>0$. Since $\lambda$ is rigid
$\lambda_i - \lambda_{i+1} \in \{0,1\}$; see Theorem \ref{rigids}.
The statement follows.
\end{proof}

\begin{lem}\label{nonsingadmiss}
Suppose $\lambda$ is non-singular and that $i < j$ are admissible indexes
for $\lambda$. Then $i$ is admissible for $\lambda^{(j)}$.
\end{lem}
\begin{proof}
If Case 1 occurs for $\lambda$ at index $i$ then $\lambda^{(j)}_i - \lambda^{(j)}_{i+1}
= \lambda_i - \lambda_{i+1} \geq 2$ unless $j = i+1$ and Case 2 occurs
for $\lambda$ at index $j$. In this situation it follows from the non-singularity of
$\lambda$ that Case 1 occurs for $\lambda^{(j)}$ at index $i$.

Now suppose Case 2 occurs for $\lambda$ at index $i$. Then $(i, i+1) \in \Delta(\lambda)$
and $\lambda_i = \lambda_{i+1}$ by definition. It follows immediately that
$\lambda^{(j)}_i = \lambda^{(j)}_{i+1}$. We shall show that $(i, i+1) \in \Delta(\lambda^{(j)})$
and conclude that Case 2 occurs for $\lambda^{(j)}$. If Case 1 occurs for
$\lambda$ at index $j$ then $(i, i+1)\in \Delta(\lambda^{(j)})$ by the non-singularity
of $\lambda$. If Case 2 occurs at $j$ for $\lambda$ then the same conclusion follows
from Lemma~\ref{deltain2}. This completes the proof.
\end{proof}

\begin{prop}\label{phising}
$|\Phi_\lambda| = 1$ if and only if $\lambda$ is non-singular.
\end{prop}
\begin{proof}
Suppose that $\lambda$ is non-singular. We shall show that all
maximal admissible sequences for $\lambda$ have the same length and
are conjugate under the symmetric group. Let $\ii$ and $\jj$ be two
such sequences. Note that if the type is $\sf D$ then we can be certain that
$R_\ii \neq 2$ and $R_\jj \neq 2$. Therefore, in any type, we might apply Corollary~\ref{tau}
and assume that they are both in ascending order. It is not hard
to see that they are both still maximal after reordering. We shall
show that they are now equal. Suppose not. Then either there exists
an index $k$ such that $i_k \neq i_j$ or one sequence is shorter
than the other, say $|\ii| < |\jj|$ and $i_k = j_k$ for $k = 1,...,|\ii|$.
In the latter situation $\ii$ clearly fails to be maximal, so assume we
are in the former situation. We may assume without loss of generality that
$i_k < j_k$. Now we may apply Lemma~\ref{nonsingadmiss} and conclude that
$\jj$ is not maximal. This contradiction confirms that $\ii = \jj$
and that all maximal admissible sequences for $\lambda$ are conjugate
under the symmetric group.

In order to prove the converse we assume that $\lambda$ is singular.
Let $(i, i+1)$ be a bad 2-step with $i$ maximal. We shall exhibit
two maximal admissible sequences, $\ii$ and $\jj$, for $\lambda$
such that $\kappa(\lambda^{\ii})_{i+1} \neq
\kappa(\lambda^{\jj})_{i+1}$. In view of Lemma \ref{kappa} the
proposition shall follow. There are two possibilities: either
$\lambda_{i+1} - \lambda_{i+2}$ is even, or $i > 1$ and
$\lambda_{i-1} - \lambda_i$ is even. Assume the first of these
possibilities, so that $\lambda_{i+1} - \lambda_{i+2}$ is even. Let
$$\ii' = (\underbrace{i+1,i+1,...,i+1}_{(\lambda_{i+1} -
\lambda_{i+2})/2 \text{ times}}).$$ We have $\lambda_{i+1}^{\ii'} =
\lambda_{i+2}^{\ii'}$. Let $\ii$ be any maximal admissible sequence
for $\lambda$ extending $\ii'$. Then $\kappa(\lambda^{\ii})_{i+1} =
0$. Now let $\jj' = (i)$ so that $\kappa(\lambda^{\jj'})_{i+1} = 1$.
Let $\jj$ be any maximal admissible sequence extending $\jj'$. By
Lemmas \ref{deltain2} and \ref{deltain}, Case 2 does not occur for
$\lambda^{\jj_k}$ at any index $j_k = i$ with $k > 1$ and since $(i,
i+1)$ is a maximal bad 2-step Case 2 cannot occur at index $j_k =
i+2$. So $\kappa(\lambda^{\jj})_{i+1} = \kappa(\lambda^{\jj'})_{i+1}
= 1$ which enables us to conclude that $\kappa(\lambda^{\ii}) \neq
\kappa(\lambda^{\jj})$, $\lambda^{\ii} \neq \lambda^{\jj}$. Hence
$|\Phi_\lambda| > 1$.

The other case is quite similar. This time we assume that $i > 1$,
that $\lambda_{i-1} - \lambda_i$ is even and $\lambda_{i+1} -
\lambda_{i+2}$ is odd. Our deductions will depend upon whether or
not $(i-2,i-1) \in \Delta(\lambda)$. Let us first assume that $i-2 \notin
\Delta(\lambda)$. We take $$\ii' =
(\underbrace{i-1,i-1,...,i-1}_{(\lambda_{i-1} - \lambda_{i})/2
\text{ times}}).$$ Let $\ii$ be any maximal admissible sequence
extending $\ii'$. Much like before $\kappa(\lambda^\ii)_{i-1} = 0$.
Now let $\jj' = (i)$ and let $\jj$ be a maximal admissible extension
of $\jj'$. Since $(i-2,i-1) \notin \Delta(\lambda)$, Lemma \ref{deltain}
shows that Case 2 does not occur for $\lambda^{\jj_k}$ at index $j_k
= i-2$ for any $k$. Since the same can be said for $j_k = i$ at any
index $k > 1$, we deduce that $\kappa(\lambda^\jj)_{i-1} =
\kappa(\lambda)_{i-1} - 1 = 1$. But then $\lambda^{\ii} \neq
\lambda^{\jj}$ and so $|\Phi_\lambda| > 1$ as desired.

To conclude the proof we must consider the final possibility: $i >
1$, $\lambda_{i-1} - \lambda_i$ even, $\lambda_{i+1} -
\lambda_{i+2}$ odd and $(i-2,i-1) \in \Delta(\lambda)$. We let $\ii'$ and
$\ii$ be defined exactly as it was in the previous paragraph. We
have $\lambda^{\ii'}_{i-1} = \lambda^{\ii'}_{i}$, so that Case 2
cannot occur at index $i_k = i$ for any $k$. Since $(i, i+1)$ is a
maximal bad 2-step for $\lambda$ we know that $(i+2,i+3) \notin
\Delta(\lambda)$. Then Lemma \ref{deltain} implies that Case 2
cannot occur at index $i_k = i+2$ for any $k$, yielding
$\kappa(\lambda^{\ii})_{i+1} = \kappa(\lambda)_{i+1} = 1$. Let
$$\jj' = (i, \underbrace{i+1,i+1,...,i+1}_{(\lambda_{i+1} -
\lambda_{i+2})/2 \text{ times}})$$ and $\jj$ be any maximal
admissible sequence extending $\jj'$. Since $\lambda_{i+1} -
\lambda_{i+2}$ is odd, $\lambda^{\jj_2}_{i+1} -
\lambda^{\jj_2}_{i+2}$ is even, and $\lambda^\jj_{i+1} =
\lambda^\jj_{i+2}$. Hence $\kappa(\lambda^\jj) = 0$ and
$|\Phi_\lambda| > 1$ as before.
\end{proof}
We can finally complete the proof of Theorem \ref{nosheets}.
\begin{proof}
Part 1 follows directly from Corollary \ref{bijection} and Lemma
\ref{contsheets}. For Part 2 use Part 1 along with Proposition
\ref{phising}.
\end{proof}

Let $\mathcal{S}$ be a sheet with data $(\li, \Oo)$. Recall that the rank of $\mathcal{S}$ is defined to
be $\dim\mathfrak{z}(\li)$. The importance of the rank is illustrated by the formula
$$\dim(\mathcal{S}) = \rank(\mathcal{S}) + \dim \Ind^\h_\li(\Oo).$$
It should be mentioned here that \cite[Remark~3]{Mor} claims that
all sheets of $\h$ containing a given nilpotent element have the
same rank (hence the same dimension). However, the example given in
Remark~\ref{remmor} shows that in general this is incorrect. A
corrigendum has been published in \cite{Mor1}. The
error may be traced to Proposition~3.11 of {\it loc.\,cit.} Using
the Kempken-Spaltenstein algorithm we can amend that proposition as
follows. First of all note that if $(\li, \Oo)$ is the data
associated to a sheet $\mathcal{S}$ then $(\li, \Oo) \in
\Psi_\lambda$, so Corollary~\ref{bijection} tells us that
$\varphi^{-1}(\li, \Oo)$ is a well-defined equivalence class in
$\Phi_\lambda$. Clearly all admissible sequences in that equivalence
class have the same length, which we may denote by
$|\varphi^{-1}(\li,\Oo)|$.
\begin{prop}\label{morrep} Let $\mathcal{S}$ be a sheet of $\h$ with data $(\li,\Oo)$.
Then $${\rm rank}(\mathcal{S}) = |\varphi^{-1}(\li,\Oo)|.$$
\end{prop}
\begin{proof}
If $\li = \gl_{i_1} \oplus \cdots \oplus \gl_{i_k} \oplus
\mathfrak{m}$ as in Lemma~\ref{levnilorbs} then clearly
$\rank(\mathcal{S}) = \dim \mathfrak{z}(\li) = k$. On the other hand
$\varphi^{-1}(\li, \Oo)$ is just the equivalence class of the
maximal admissible sequence $(i_1,...,i_k)$ which itself has length
$k$.
\end{proof}
\begin{cor}\label{newz}
If $\lambda \in \mathcal{P}_\epsilon(N)$ and $e = e(\lambda)$ then
$$r(e)=\max_{\ e \in \mathcal{S}}\, {\rm rank} (\mathcal{S}) =
z(\lambda)$$ where the maximum is taken over all sheets of $\h$
containing $e$.
\end{cor}
\begin{proof}
Use the above proposition and Theorem~\ref{zismax}.
\end{proof}

\begin{rem}\label{remmor}
{\rm Some sheets of different ranks in Lie algebras of type ${\sf
B}, {\sf C}$ or ${\sf D}$ may share the same nilpotent orbit.
Indeed, the partition $\lambda=(4,2,2)\in \mathcal{P}_\epsilon(-1)$
affords three maximal admissible sequences $(1,3)$, $(3,1)$ and
$(2)$ of lengths $2$, $2$ and $1$, respectively, which lead to two
rigid partitions $(0)$ and $(2,1,1)$.  So it follows from
Proposition~\ref{morrep} that the nilpotent element
$e(\lambda)\in\mathfrak{sp}_8$ lies in two sheets $\mathcal{S}_1$
and $\mathcal{S}_2$ of $\mathfrak{sp}_8$ with ${\rm
rk}(\mathcal{S}_1)=2$ and ${\rm rk}(\mathcal{S}_2)=1$. Note that
$s(\lambda)=2$, $\Sigma(\lambda)=\emptyset$ and
$|\Delta(\lambda)|=|\Delta_{\rm bad}(\lambda)|=1$. Hence
$e(\lambda)$ is singular and
$$z(\lambda)=s(\lambda)+|\Delta(\lambda)|-\big(|\Delta_{\rm bad}(\lambda)|-|\Sigma(\lambda)|\big)=2+1-(1-0)=2.$$
This example shows that different sheets of classical Lie algebras
containing a given nilpotent element may have different dimensions.}
\end{rem}
\begin{rem}
{\rm Suppose that char$(\K)=0$ and let $x$ be an arbitrary element of $\h$. Then it is immediate from
Theorem \ref{nosheets} and
Corollary~\ref{izosim}
that the following are equivalent:
\begin{itemize}
\item[(i)]
$x$ belongs to a unique sheet of $\h$;

\smallskip

\item[(ii)]
$x$ is a smooth point of the quasi-affine variety
$\h^{(m)}$ where  $\dim \h_x=m$;
\smallskip

\item[(iii)]  the maximal rank of the sheets of $\h$ containing $x$ equals
$\dim (\h_x/[\h_x,\h_x])$.
\end{itemize}
Indeed, if $x=x_s+x_n$ is the Jordan--Chevalley decomposition of $x$
then it is well known (and easy to see) that $\li:=\h_{x_s}$ is a
Levi subalgebra of $\h$ and $\dim (\h_x/[\h_x,\h_x])=\dim
(\li_{x_{n}}/[\li_{x_n},\li_{x_n}])$. On the other hand, it follows from our
discussion in Subsection~\ref{2.1}
and the description of the Zariski closure of a sheet given in \cite[Theorem~5.4]{BKr} that there is a
rank-preserving bijection between the sheets of $\h$ containing $x$
and the sheets of $\li$ containing $x_n$. So the problem reduces
quickly to the case where $x=x_n$, and since all simple ideals of
$\li$ are Lie algebras of classical types, Theorem~\ref{nosheets}
and Corollary~\ref{izosim} apply to $x_n$ and give the desired
result. This confirms the first part of Izosimov's conjecture; see
\cite[Conjecture 1]{Iz}. The second part of his conjecture (which
is also interesting and plausible) remains open at the moment. We
mention for completeness that the three statements above hold true
for all $x$ in $\mathfrak{sl}_n$ and  $ \mathfrak{gl}_n$; see \cite[Proposition~3.3]{Iz}.}
\end{rem}

\section{Commutative quotients of finite
 $W$-algebras and sheets}
 \subsection{Finite $W$-algebras and related commutative algebras}\label{3.1} From now on
we assume that $\K$ is an algebraically closed field of
characteristic $0$ and $G$ is a connected reductive $\K$-group. Let
$\g=\Lie(G)$ and let $e$ be a non-zero nilpotent element in $\g$. We
include $e$ into an $\mathfrak{sl}(2)$-triple $\{e,h,f\}\subset\g$
and consider the {\it Slodowy slice} $S_e:=e+\g_f$, an affine
subspace of $\g$ transversal to the adjoint $G$-orbit of $e$. The
finite $W$-algebra $U(\g,e)$ is a non-commutative filtered
deformation of the algebra $\K[S_e]$ of regular functions on $S_e$
endowed with the Slodowy grading; see \cite[5.1]{Sasha1} for more
detail. By using the Killing form of $\g$ we may identify $\K[S_e]$
with the symmetric algebra $S(\g_e)$.

The action of ${\rm ad}\,h$ gives $\g$ a $\Z$-graded Lie algebra
structure $\g=\bigoplus_{i\in\Z}\,\g(i)$ and we have that
$e\in\g(2)$ and $\g_e=\bigoplus_{i\ge 0}\,\g_e(i)$ where
$\g_e(i):=\g_e\cap \g(i)$. Let $v_1,\ldots, v_r$ be a basis of
$\g_e$ such that $v_i\in \g_i(n_i)$ for some $n_i\ge 0$. According
to \cite[Theorem~4.6]{Sasha1} the finite $W$-algebra $U(\g,e)$ has a
$\K$-basis consisting of all monomials
$v^{\mathbf{a}}=v_1^{a_1}\cdots v_r^{a_r}$ with $a_i\in\Z_{\ge 0}$
and assigning to $v^{\bf a}$ filtration degree $|{\bf
a}_e|:=\sum_{i=1}^r\,a_i(n_i+2)$ gives $U(\g,e)$ an algebra
filtration called the {\it Kazhdan filtration} of $U(\g,e)$.
Furthermore, the corresponding graded algebra, ${\rm gr}_{\sf
K}\,U(\g,e)$, is isomorphic to the symmetric algebra $S(\g_e)$ with
$v_i$ having Kazhdan degree $n_i+2$ and the following relations hold
in $U(\g,e)$ for all $1\le i<j\le r$:
\begin{eqnarray}\label{eqn1}
\qquad\quad v_i\cdot v_j-v_j\cdot v_i=[v_i,v_j]+q_{ij}(v_1,\ldots,v_r)+\mbox{terms of lower Kazhdan deree},
\end{eqnarray}
where $q_{ij}$ is a polynomial of Kazhdan degree $n_i+n_j+2$ whose
constant and linear parts are both zero (here $[v_i,v_j]$ is the Lie
bracket of $v_i$ and $v_j$ in $\g_e$). We write ${\sf K}_l\,U(\g,e)$
for of the $l^{\rm th}$ component of the Kazhdan filtration of
$U(\g,e)$.

It is well known that the group $C(e):=G_e\cap G_f$ is reductive and
its finite quotient $\Gamma:=C(e)/C(e)^\circ$ identifies canonically
with the component group of $\Ad\,G_e$. Besides, $\Lie(C(e))=\g_e(0)$.
As explained in \cite[2.2, 2.3]{Sasha2}, the group $C(e)$ acts on
$U(\g,e)$ by algebra automorphisms and preserves all components of
the Kazhdan filtration of $U(\g,e)$. Moreover, there exists an
injective $C(e)$-module homomorphism $\Theta\colon\,\g_e\rightarrow
U(\g,e)$ with the property that $\Theta(\g_e)$ generates $U(\g,e)$
as an algebra and ${\rm gr}_{\sf K}\,\Theta(\g_e)\cong \g_e[2]$ as
$C(e)$-modules, where $\g_e[2]$ stands for the $\Ad\,C(e)$-module
$\g_e$ with all degrees shifted by $2$. To be more precise, the
group $C(e)\subset G_h$ preserves both the Slodowy grading of
$S(\g_e)$ and the grading of $S(\g_e)$ given by total degree. In
view of \cite[Lemma 4.5]{Sasha1} this implies that the graded linear
map ${\rm gr}_{\sf K}\,\Theta(\g_e)\rightarrow \g_e[2]$ sending
${\rm gr}_{\sf K}\,\Theta(v)\in S(\g_e)$ to its linear part is an
isomorphism of $C(e)$-modules. In what follows we shall denote by
${\rm gr}^0_{\sf K}\,\Theta(v)$ the linear part of ${\rm gr}_{\sf
K}\,\Theta(v)$.

To ease notation we shall sometimes suppress the notion of $\Theta$
and assume from now on that the above-mentioned identification of
${\rm gr}_{\sf K}\,U(\g,e)$ and $S(\g_e)$ is $C(e)$-equivariant.
Thanks to \cite[Lemma 2.4]{Sasha2} we then have that $v_i\cdot
v_j-v_j\cdot v_i=[v_i,v_j]$ for all $v_i,v_j\in\g_e(0)$, where the
products of $v_i$ and $v_j$ are taken in $U(\g,e)$ and the Lie
bracket $[v_i,v_j]$ is taken in $\g_e$.

To shorten notation we write $\mathfrak{c}_e$ for $\g_e^{\rm
ab}=\g_e/[\g_e,\g_e]$. Since $[\g_e(0),\g_e]\subset [\g_e,\g_e]$ and
$\g_e(0)=\Lie(C(e))$, it follows from Weyl's theorem that
$C(e)^\circ$ acts trivially on $\mathfrak{c}_e$. This gives rise to
a natural linear action of the component group
$\Gamma=G_e/G_e^\circ$ on the vector space $\mathfrak{c}_e$. We
denote by $\mathfrak{c}_e^\Gamma$ the fixed point space of this
action and set $$c(e):=\dim \mathfrak{c}_e\ \, \mbox { and }\
\,c_\Gamma(e):=\dim \mathfrak{c}_e^\Gamma.$$

Let $\mathcal{S}_1,\ldots,\mathcal{S}_t$ be all pairwise distinct
sheets of $\g$ containing $e$. As we explained in
Subsection~\ref{2.1}, every sheet $\mathcal{S}_i$ contains a unique
Zariski open  decomposition class $\mathcal{D}(\li_{i},e_i)=(
\Ad\,G)(e_i+\z(\li_{i})_{\rm reg})$ characterised by the property
that $e_i$ is rigid in $\li_{i}$. We write $r_i=\dim\z(\li_i)$ for
the rank of $\mathcal{S}_i$. It is well known that the set
$X_i:=\mathcal{S}_i\cap (e+\g_f)$ is a connected affine variety
acted upon by the reductive group $C(e)$. Working over complex
numbers, Katsylo proved in \cite{Kat} that the subgroup $C(e)^\circ$
operates trivially on $X_i$, the induced action of
$\Gamma=C(e)/C(e)^\circ$ on the irreducible components of $X_i$ is
transitive,  and  the  morphism $$G\times X_i \rightarrow
\mathcal{S}_i, \qquad\,(g,x)\mapsto (\Ad\,g)\cdot x,$$ is smooth,
surjective of relative dimension $\dim \g_e$. Moreover, Katsylo
showed that it gives rise to an open morphism
$\psi_i\colon\,\mathcal{S}_i\rightarrow X_i/\Gamma$ with the
following properties:

\begin{itemize}
\item [(i)]the fibres of $\psi_i$ are
$G$-orbits;
\smallskip

\item[(ii)]
 for any open subset
$X$ of $X_i/\Gamma$ the induced algebra map $\K[U]\rightarrow \K[\psi_{i}^{-1}(U)]^G$ is an isomorphism.
\end{itemize}
In brief, each morphism $\psi_i$ is a geometric quotient and
$\dim(\mathcal{S}_i/G)=\dim(X_i/\Gamma)=r_i.$
A purely algebraic proof of Katsylo's results can be found in \cite{Im}.

Denote by $U(\g,e)^{\rm ab}$ the largest commutative quotient of
$U(\g,e)$ (it has the form $U(\g,e)/I_c$ where $I_c$ is the
two-sided ideal of $U(\g,e)$ generated by all commutators $u\cdot
v-v\cdot u$ with $u,v\in U(\g,e)$). This finitely generated
$\K$-algebra is important because its  maximal spectrum, $\mathcal
E$,  parametrises the $1$-dimensional representations of $U(\g,e)$.
According to \cite[Theorem~1.2]{Sasha3}, for any {\it induced}
nilpotent element $e$ of $\g$ the Krull dimension of $U(\g,e)^{\rm
ab}$ equals $r(e):=\max\{r_1,\ldots, r_t\}$ and the number of
irreducible components of $\mathcal E$ is greater than or equal to
the total number of all irreducible components of the $X_i$'s. If
$\g=\mathfrak{sl}_n$ then every nilpotent element $e\in\g$ lies in a
unique sheet $\mathcal{S}=\mathcal{S}(e)$. Since every nilpotent
element of $\mathfrak{sl}_n$ is Richardson, the sheet
$\mathcal{S}(e)$ contains a dense decomposition class of the form
$(\mathrm{Ad}\,G)(\z(\li)_{\rm reg})$. Using Remark \ref{R1} it is
easy to see that $\dim\,\mathfrak{c}_e=\dim\,\z(\li)$. On the other
hand, it was proved in \cite{Sasha3} that for $\g=\mathfrak{sl}_n$
the algebra $U(\g,e)^{\rm ab}$ is isomorphic to a polynomial algebra
in $r(e)=\dim\,\z(\li)$ variables. The proof in ${\it loc.\,cit.}$
relied heavily on the explicit presentation of finite $W$-algebras
of type $\sf A$ obtained by Brundan--Kleshchev in \cite{BK}.

In this section we make an attempt to classify those induced
nilpotent elements $e\in\g$ for which $U(\g,e)^{\rm ab}$ is
isomorphic to a polynomial algebra. In view of the above discussion,
this can happen only if $e$ lies in a unique sheet of $\g$ which
makes one wonder to what extent the converse is true. For $\g$
classical, we are going to apply our results on non-singular
nilpotent elements to show that this is always the case (even for
$e$ rigid!), whilst for $\g$ exceptional we shall rely on de Graaf's
computations in \cite{deG2} to show the same is true for {\it almost
all} induced induced orbits.

Since the group $C(e)$ operates on $U(\g,e)$ by algebra
automorphisms, it acts on the variety $\mathcal E$ which identifies
naturally with the set of all ideals of codimension $1$ in
$U(\g,e)$. Since the group $C(e)^\circ$ preserves any two-sided
ideal of $U(\g,e)$ by \cite[p.~501]{Sasha2}, it must act trivially
on $\mathcal E$. We thus obtain a natural action of
$\Gamma=C(e)/C(e)^\circ$ on the affine variety $\mathcal E$. We
denote by $\mathcal{E}^\Gamma$  the corresponding fixed point set
and let $I_\Gamma$ be the ideal of $U(\g,e)^{\rm ab}$ generated by
all $\phi-\phi^\gamma$ with $\phi\in U(\g,e)^{\rm ab}$ and
$\gamma\in\Gamma$. It is clear that $\mathcal{E}^\Gamma$ is
contained in the zero locus of $I_\Gamma$. Conversely, if
$\eta\in\mathcal{E}$ is such that $\phi(\eta)=0$ for all $\phi\in
I_\Gamma$, then $\gamma(\eta)=\eta$  for all $\gamma\in\Gamma$.
Indeed, otherwise $\eta$ and $\gamma_0^{-1}(\eta)$ would be distinct
maximal ideals of $U(\g,e)^{\rm ab}$ for some $\gamma_0\in \Gamma$
and we would be able to find an element $\phi\in U(\g,e)^{\rm ab}$
with $\phi(\eta)=0$ and $\phi(\gamma_0^{-1}(\eta))\ne 0$. But this
would imply that $(\phi-\phi^{\gamma_0})(\eta)\ne 0$, a
contradiction.  As a result, $\mathcal{E}^\Gamma$ coincides with the
zero locus of $I_\Gamma$ in $\mathcal E$. We denote by
$U(\g,e)_\Gamma^{\rm ab}$ the finitely generated $\K$-algebra
$U(\g,e)^{\rm ab}/I_\Gamma$. The above discussion shows that
$$\mathcal{E}^\Gamma\,=\,{\rm Specm}\,U(\g,e)_\Gamma^{\rm ab}.$$

In this section we aim to show that in most cases $U(\g,e)^{\rm
ab}_\Gamma$ is isomorphic to a polynomial algebra in $c_\Gamma(e)$
variables. As will be explained later, the polynomiality of
$U(\g,e)^{\rm ab}_\Gamma$ can be used to classify those primitive
ideals $I$ of $U(\g)$ whose associated variety ${\rm VA}(I)$ appears
with multiplicity one in the associated cycle ${\rm AC}(I)$.
Detailed information on such ideals is very important because the
primitive quotients $U(\g)/I$ extend to the Dixmier algebras
quantising the nilpotent orbits of $\g$ in the sense of \cite[5.1,
5.3]{Lo1}.

\subsection{A sufficient condition for polynomiality}
The goal of this subsection is to give a sufficient condition of
polynomiality of $U(\g,e)^{\rm ab}$ and $U(\g,e)_\Gamma^{\rm ab}$
for an arbitrary simple Lie algebra $\g$ and use it to classify
those nilpotent elements in the Lie algebras of classical groups for
which $U(\g,e)^{\rm ab}$ is a polynomial algebra. In view of our
discussion in Subsection~\ref{3.1} we may (and will) identify $\g_e$
with a $C(e)$-submodule of $U(\g,e)$ containing a PBW basis of
$U(\g,e)$. For $k,l\in\Z_{\ge 0}$, we set $S_{\le
l}(\g_e):=\bigoplus_{i\le l}\,S^i(\g_e)$ and denote by $S^{\langle
k\rangle}(\g_e)$ the linear span of all monomials $v^{\bf
a}:=v_1^{a_1}\cdots v_r^{a_r}$ in $S(\g_e)$ with $|{\bf a}|_e=k$.
\begin{lem}\label{generate}
Let $I$ be a proper two-sided ideal of $U(\g,e)$ and let $V_I$ and
$V_I'$ be two $C(e)$-submodules of $\g_e$, identified with
$\Theta(\g_e)$, such that $\g_e[2]={\rm gr}_{\sf K}^0(V_I)\oplus{\rm
gr}_{\sf K}^0(V_I')$ as graded ${\rm Ad}\,C(e)$-modules. Suppose
further that
$${\rm gr}_{\sf K}(v)\in \Big({\rm gr}_{\sf K}(I)+S_{\ge 2}(\g_e)\Big)\cap S^{\langle i\rangle}(\g_e)$$
for all $v\in (V_I\cap {\sf K}_i\, U(\g,e))\setminus {\sf K}_{i-1}\, U(\g,e))$, where $i\in\Z_{\ge 0}$. Then the unital
algebra $U(\g,e)/I$ is generated by the subspace $V'_I$.
\end{lem}
\begin{proof} Since $\g_e[2]={\rm gr}_{\sf K}^0(V_I)\oplus{\rm gr}_{\sf K}^0(V_I')$, it is easy to see
that $\g_e=V_I\oplus V_I'$. Let $\pi\colon\,\g_e\twoheadrightarrow
V_I$ and $\pi'\colon\,\g_e\twoheadrightarrow V_I'$ be the
$C(e)$-equivariant projections induced by the direct sum
decomposition $\g_e=V_I\oplus V_I'$ and denote by $\mathcal{A}$ the
$\K$-span in $U(\g,e)$ of all $\pi'(v)^{\bf
i}:=\pi'(v_1)^{i_1}\cdots\pi'(v_r)^{i_r}$ with ${\bf i}\in\Z_{\ge
0}^r$. We shall prove by induction on $k$ that every monomial
$v^{\bf a}\in U(\g,e)$ with $|{\bf a}|_e=k$ lies in $\mathcal{A}+I$.
Then the lemma will follow.

The statement is obviously true for $k=0$. Suppose that it holds for
all $k<m$. If $|{\bf a}|>1$ then the statement follows by induction
on $k$. Hence we may assume further that $|{\bf a}|=1$, so that
$v^{\bf a}=v_s$ for some $s\in\{1,\ldots, r\}$ and $k=n_s+2$. Thanks
to our assumption on $V_I$ we have that
\begin{eqnarray}\label{eqn2}
\pi(v_s)=u_s+\sum_{|{\bf i}|_e=n_s+2,\ |{\bf i}|\ge 2}
\lambda_{s,\,{\bf i}}\,v^{\bf i}+\mbox{ terms of lower Kazhdan
degree}
\end{eqnarray}
for some $u_s\in I$ and $\lambda_{s,\,{\bf i}}\in\K$. Therefore,
$$\pi(v_s)\,\equiv \sum_{|{\bf i}|_e=n_s+2,\ |{\bf i}|\ge 2}
\lambda_{s,\,{\bf i}}\,v^{\bf i}\mod (\mathcal{A}+I)$$ by the induction assumption.
Our aim is to show that
$v^{\bf i}\in \mathcal{A}+I$ for all $\bf i$ such that
$|{\bf i}|_e=n_s+2$ and we are going to use downward induction on the total degree of ${\bf i}$
(this is possible since there are only finitely many ${\bf i}\in\Z_{\ge 0}^r$ for which $|{\bf i}|=n_s+2$).
If $\jj$ is such that $|\jj|_e=n_s+2$ and
$|\jj|\ge |\ii|$ whenever
$|\ii|_e=n_s+2$, then
$$v^\jj=\prod_{i=1}^r(\pi(v_i)+\pi'(v_i))^{j_i}
\equiv\,\pi'(v)^\jj\mod\big(\mathcal{A}+I+{\sf K}_{n_s+1}\,U(\g,e)\big)$$
thanks to (\ref{eqn2}) and (\ref{eqn1}).
This takes care of the induction base.
Now suppose
$v^{\bf i}\in \mathcal{A}+I$ for all $\bf i$ with
$|{\bf i}|_e=n_s+2$ and $|{\bf i}|>d$ and take any ${\bf j}\in\Z_{\ge 0}^r$ with $|{\bf j}|_e=n_s+2$ and $|{\bf j}|=d$.
Since $v^{\bf j}=\prod_{i=1}^r(\pi(v_i)+\pi'(v_i))^{j_i}$, combining (\ref{eqn2}) and (\ref{eqn1}) yields that
$$v^{\bf j}\,\equiv\, \pi'(v)^{\bf j}+\sum_{|{\bf i}|_e=n_s+2,\ |{\bf i}|>d}
\mu_{{\bf j},\,{\bf i}}\,v^{\bf i}\mod \big(\mathcal{A}+I+{\sf K}_{n_s+1}\,U(\g,e)\big)$$
for some $\mu_{{\bf j},\,{\bf i}}\in\K$.
As ${\sf K}_{n_s+1}\,U(\g,e)$ is spanned by all $v^{\bf b}$ with $|{\bf b}|_e<n_s+2$, we know that
 ${\sf K}_{n_s+1}\,U(\g,e)\subseteq \mathcal{A}+I$.
 Then our present induction assumption gives
 $v^{\bf j}\in\mathcal{A}+I$, as claimed. But then $\pi(v_s)\in \mathcal{A}+I$ in view of (\ref{eqn2}). Since
 $v_s=\pi(v_s)+\pi'(v_s)$ and $\pi'(v_s)\in\mathcal{A}$ by the definition of $\mathcal{A}$, we deduce that
 $v_s\in\mathcal{A}+I$ finishing the proof.
\end{proof}
It should be stressed at this point that in Lemma \ref{generate} we do not require $V_I$  to be contained in $I$.
\begin{prop}\label{abgen}
Let $e$ be any nilpotent element of $\g$. Then the following are true:

\smallskip

\begin{itemize}
\item[(i)]
If $\mathcal{E}\ne \emptyset$, then the unital algebra $U(\g,e)^{\rm ab}$ is generated by
$c(e)$ elements.
\smallskip
\item[(ii)]
If $\mathcal{E}^\Gamma\ne \emptyset$, then the unital algebra $U(\g,e)^{\rm ab}_\Gamma$ is generated by
$c_\Gamma(e)$ elements.
\end{itemize}
\end{prop}
\begin{proof}
Due to a  possibility of confusion, in this proof we shall
distinguish between $\g_e=S^{1}(\g_e)\subset S(\g_e)$ and its
isomorphic copy $\Theta(\g_e)\subset U(\g,e)$.

\smallskip

\noindent (i) The defining ideal $I_c$ of $U(\g,e)^{\rm ab}$
contains all commutators $[\Theta(u),\Theta(v)]$ with $u,v\in \g_e$.
Since $C(e)$ is a reductive group, $\g_e$ contains a graded
$\Ad\,C(e)$-submodule of dimension $c(e)$ complementary to the
derived subalgebra $[\g_e,\g_e]$. Let $M$ be such a submodule and
recall the $C(e)$-equivariant isomorphism
$\Theta(\g_e)\stackrel{\sim}{\longrightarrow}{\rm gr}_{\sf
K}^0\,\Theta(\g_e)=\g_e[2]$ described in Subsection~\ref{3.1}. We
choose for $V_{I_c}$ and $V_{I_c}'$ the preimages  under this
isomorphism  of $[\g_e,\g_e]$ and $M$, respectively. It is immediate
from (\ref{eqn1}) and our earlier remarks in this proof that the
$C(e)$-submodules $V_{I_c}$ and $V_{I_c}'$ of $\Theta(\g_e)$ satisfy
all conditions of Lemma \ref{generate}. Since $\dim V_{I_c}'=\dim
M=c(e)$, the first statement follows.

\smallskip

\noindent (ii) Let $\widetilde{I}_c$ be the preimage of the ideal
$I_\Gamma$ of $U(\g,e)^{\rm ab}$ under the canonical homomorphism
$U(\g,e)\twoheadrightarrow U(\g,e)^{\rm ab}$. Then $\widetilde{I}_c$
is a two-sided ideal of $U(\g,e)$ and $U(\g,e)/\widetilde{I}_c\cong
U(\g,e)^{\rm ab}_\Gamma$ as algebras. Since $[\g_e(0),M]\subseteq
[\g_e,\g_e]$ and $M\cap [\g_e,\g_e]=0$, it follows from Weyl's
theorem that the connected reductive group $C(e)^\circ$ acts
trivially on $M$. Therefore, $M$ has a natural structure of a
$\Gamma$-module. There exists a $\Gamma$-submodule $M'$ of $M$
complementary to $M^\Gamma:=\{x\in M\colon\, \gamma(x)=x\}$. We
choose for $V_{\widetilde{I}_c}$ and $V'_{\widetilde{I}_c}$ the
preimages in $U(\g,e)$ of $M'\oplus [\g_e,\g_e]$ and $M^\Gamma$
under the above-mentioned isomorphism
$\Theta(\g_e)\stackrel{\sim}{\longrightarrow}{\rm gr}_{\sf
K}^0\,\Theta(\g_e)=\g_e[2]$ of $C(e)$-modules. Note that
$V_{\widetilde{I}_c}= V_{I_c}\oplus N'$ where $N'$ is the preimage
of $M'$ in $\Theta(\g_e)$. Due to our choice of $M'$ the group
$C(e)^\circ$ acts trivially on $N'\cong M'$ and $N'$ is spanned by
the elements of the form $u-\gamma(u)$ with $u\in N'$ and
$\gamma\in\Gamma$. The definition of $I_\Gamma$ implies that
$N'\subset \widetilde{I}_c$. As $I\subseteq \widetilde{I}$, our
discussion in Part (i) now shows that the modules
$V_{\widetilde{I}_c}$ and $V_{\widetilde{I}_c}'$ satisfy all
conditions of Lemma \ref{generate}. Since $\dim
V_{\widetilde{I}_c}'=\dim M^\Gamma=c_\Gamma(e)$ we obtain (ii).
 \end{proof}
\begin{cor}\label{poly} Let $e$ be an induced nilpotent element of $\g$. Then the following hold:

\smallskip

\begin{itemize}
\item[(i)] If $c(e)=r(e)$, then
$U(\g,e)^{\rm ab}\cong S(\mathfrak{c}_e)$ as $\K$-algebras and
$\Gamma$-modules and $U(\g,e)^{\rm ab}_\Gamma\cong
S(\mathfrak{c}_e^\Gamma)$ as $\K$-algebras.

\smallskip

\item[(ii)] If $\mathcal{E}^\Gamma\ne\emptyset$ and $\dim \mathcal{E}^\Gamma\ge c_\Gamma(e)$, then
$U(\g,e)^{\rm ab}_\Gamma\,\cong \, S(\mathfrak{c}_e^\Gamma)$ is a polynomial algebra in $c_\Gamma(e)$ variables.
\end{itemize}
\end{cor}
\begin{proof}
(i) Combining \cite[Theorem 1.2]{Sasha3} with the main results of
\cite{GRU} we see that $\dim U(\g,e)^{\rm ab}=r(e)$ (in particular,
$\mathcal{E}\neq \emptyset$). On the other hand, since $V'_{I_c}$ is
a $C(e)$-submodule of $U(\g,e)$ isomorphic to $\mathfrak{c}_e$,
Propostion \ref{abgen}(i) implies that there exists a natural
surjective $C(e)$-equivariant algebra homomorphism
$\psi\colon\,S(\mathfrak{c}_e)\twoheadrightarrow U(\g,e)^{\rm ab}$.
If $c(e)=\dim S(\mathfrak{c}_e)$ equals $r(e)= U(\g,e)^{\rm ab}$,
the map $\psi$ must be injective. Since $C(e)^\circ$ acts trivially
on $\mathfrak{c}_e$ we deduce that $U(\g,e)^{\rm ab}\cong
S(\mathfrak{c}_e)$ as $\K$-algebras and $\Gamma$-modules. But then
$\mathcal{E}\cong \mathfrak{c}_e^*$ as $\Gamma$-varieties implying
that $\mathcal{E}^\Gamma\cong(\mathfrak{c}_e^*)^\Gamma$. Since the
defining ideal in $S(\mathfrak{c}_e)\cong\K[\mathfrak{c}_e^*]$ of
the linear subspace $(\mathfrak{c}_e^*)^\Gamma$ is generated by all
$f-f^\gamma$ with $f\in S(\mathfrak{c}_e)$ and $\gamma\in\Gamma$,
its image under $\psi$ coincides with $I_\Gamma$. This implies that
$S(\mathfrak{c}_e^\Gamma)\cong U(\g,e)^{\rm ab}_\Gamma$ as
$\K$-algebras.

\smallskip

\noindent (ii) As $\mathcal{E}^\Gamma\ne\emptyset$, it follows from
Proposition~\ref{abgen}(ii) that there is a surjective algebra
homomorphism $S(V'_{\widetilde{I}_c})\twoheadrightarrow U(\g,e)^{\rm
ab}_\Gamma$. As a consequence, $c_\Gamma(e)=\dim
V_{\widetilde{I}_c}'\ge\dim U(\g,e)^{\rm ab}_\Gamma.$ If
 $\dim U(\g,e)_\Gamma^{\rm ab}=\dim\mathcal{E}^\Gamma\ge c_\Gamma(e)$, then it must be that
 $U(\g,e)^{\rm ab}_\Gamma\cong S(V'_{\widetilde{I}_c})\cong S(\mathfrak{c}_e^\Gamma)$ as $\K$-algebras.
\end{proof}
\subsection{Further results on polynomiality of commutative quotients}
In this subsection we are going to apply our results on non-singular
nilpotent elements to give a complete description of those nilpotent
elements $e$ in classical Lie algebras $\g$ for which $U(\g,e)^{\rm
ab}$ is a polynomial algebra. For {\it induced} nilpotent orbits in
exceptional Lie algebras, we are going to apply de Graaf's
computations in \cite{deG2} to obtain  strong partial results in
this direction which will leave undecided only six of such orbits,
one in type ${\sf F_4}$, one in type ${\sf E_6}$, one in type ${\sf
E_7}$ and three in type $\sf E_8$. The very challenging case of
rigid nilpotent orbits in exceptional Lie algebras requires
completely different methods and is dealt with in \cite{Pr}.
\begin{thm}\label{class}
Let $e$ be a nilpotent element in a classical Lie algebra $\g$. Then the following are equivalent:
\begin{itemize}
\item[(1)]
$e$ belongs to a unique sheet of $\g$;
\smallskip

\item[(2)] $U(\g,e)^{\rm ab}$ is isomorphic to a polynomial algebra in $c(e)$ variables.
\smallskip
\end{itemize}
\end{thm}
\begin{proof}
If $e$ belongs to a unique sheet of $\g$, then
$c(e)=r(e)$ by Corollary \ref{izosim} and Theorem~\ref{zismax}.
Since $\g$ is classical, it follows from \cite{KP}, \cite{RBr} and \cite{Lo1} that $\mathcal{E}\ne \emptyset$.
But then Corollary \ref{poly}(i) shows that
$U(\g,e)^{\rm ab}\cong S(\mathfrak{c}_e)$ as unital $\K$-algebras.

If $U(\g,e)^{\rm ab}$ is isomorphic to a polynomial algebra, then
the variety $\mathcal{E}$ is irreducible. If $e$ is induced, then applying \cite[Theorem 1.2]{Sasha3} yields
that $e$ belongs to a unique sheet. If $e$ is rigid, this holds automatically as $({\rm Ad}\,G)\, e$ is a sheet of $\g$.
This completes the proof.
\end{proof}
\begin{rem}\label{R5}
{\rm (a) Suppose $\g$ is a classical Lie algebra and $e$ is a rigid
nilpotent element of $\g$. Then it follows from Theorem \ref{class} and Corollary~\ref{rrr}
that the algebra $U(\g,e)$ admits a {\it unique} $1$-dimensional
representation.

\smallskip

\noindent (b) Suppose $\g$ is a classical Lie algebra and
$e=e(\lambda)$ is a nilpotent element of $\g$ associated with a
non-singular partition $\lambda$. Then combining \cite[Theorem
1.2]{Sasha3} with Corollary~\ref{maxseq}(2) and (the proof of)
Theorem \ref{class} we deduce that $e$ belongs to a unique sheet
$\mathcal{S}(e)$ of $\g$ and the variety $(e+\g_f)\cap
\mathcal{S}(e)$ is irreducible. Of course, the uniqueness of
$\mathcal{S}(e)$ also follows from Proposition~\ref{phising} which
we proved by purely combinatorial arguments. The irreducibility of
the variety $(e+\g_f)\cap \mathcal{S}(e)$ is actually a consequence
of the following more general result which follows from Im Hof's
theorem on smoothness of sheets in classical Lie algebras:

\smallskip

\noindent {\it for any sheet $\mathcal S$ containing a nilpotent
element $e\in\g$ the affine variety $X=\mathcal{S}\cap (e+\g_f)$ is smooth
and irreducible.}
\smallskip

\noindent Indeed, it is immediate from Katsylo's results mentioned
in Subsection~\ref{3.1} that $\dim \mathcal{S}=\dim X+\dim\,
(\Ad\,G)\, e$. Since $\mathcal{S}$ contains both $X$ and $(\Ad\,G)e$
the tangent space $T_e(\mathcal{S})$ contains $T_e(X)+
T_e\big((\Ad\,G)\, e\big)$. Since $T_e(X)\subset T_e(e+\g_f)=\g_f$
and $T_e\big((\Ad\,G)\, e\big)=[e,\g]$, it follows that
$T_e(\mathcal{S})$ contains $T_e(X)\oplus [e,\g]$. As the variety
$\mathcal S$ is smooth and $\dim\big(T_e(X)\oplus [e,\g]\big)\ge
\dim\,\mathcal{S}$ it must be that $T_e(\mathcal{S})=T_e(X)\oplus
[e,\g]$ and $\dim T_e(X)=\dim\,X$. As a consequence, $e$ is a smooth
point of $X$. But then $e$ belongs to a unique irreducible component
of $X$; see \cite[Chapter II, \S2, Theorem 6]{Sh}. On the other
hand, there is a regular $\K^\times$-action on $X$ attracting every
point $x\in X$ to $e$. Therefore, all irreducible components of $X$
contain $e$ and hence the variety $X$ is irreducible. Since the
singular locus ${\rm Sing}(X)$ of $X$ is Zariski closed and
invariant under the above $\K^\times$-action, this argument also
shows that $X$ is a smooth variety.}
\end{rem}

Our next result relies heavily on Losev's work \cite{Lo3}. Together
with Corollary~\ref{poly}(ii) it will enable us to describe the
variety $\mathcal{E}^\Gamma$ for many induced nilpotent elements
$e\in \g$.
\begin{prop}\label{mult1}
Let $P=LU$ be a proper parabolic subgroup of $G$, where $L\subset G$
is a Levi subgroup and $U=R_u(P)$, and suppose that a nilpotent
element $e=e_0+e_1\in{\rm Lie}(P)$ with $e_0\in \li={\rm Lie}(L)$
and  $e_1\in{\rm Lie}(U)$ is induced from $e_0$ in such a way that
$G_e\subset P$. Let $\mathcal{E}_0={\rm
Specm}\,U([\li,\li],e_0)^{\rm ab}$ and suppose further that
$\mathcal{E}_0^{\Gamma_0}\neq \emptyset$ where
$\Gamma_0=L_{e_0}/(L_{e_0})^\circ$. Then
$\mathcal{E}^\Gamma\ne\emptyset$ and
$\dim\mathcal{E}^\Gamma\ge\dim\z(\li)$ where $\z(\li)$ denotes the
centre of the Levi subalgebra $\li={\rm Lie}(L)$.
\end{prop}
\begin{proof}
(a) Let $(e_0,h_0,f_0)$ be an $\mathfrak{sl}_2$-triple of $\li$
containing $e_0$ (if $e_0=0$ then $(e_0,h_0,f_0)$ is the zero
triple). By the $\mathfrak{sl}_2$-theory, the reductive group
$C(L,e_0):=L_{e_0}\cap L_{h_0}$ is a Levi subgroup of the
centraliser $L_{e_0}$. We include $e$ into an
$\mathfrak{sl}_2$-triple $(e,h,f)$ of $\g$ and denote by $\lambda_e$
the cocharacter in $X_*(G)$ with $h\in \Lie(\lambda_e(\K^\times))$.
Note that $C(e)=G_e\cap G_h=G_e\cap Z_G(\lambda_e)$. Since
$C(e)\subset P$ by our assumption on $e$, it follows from
\cite[Proposition~6.1.2(4)]{Lo3} that the reductive group
$\lambda_e(\K^\times)C(e)$ is contained in $P$. Since any reductive
subgroup of $P$ is conjugate under $P$ to a subgroup of $L$ by
Mostow's theorem, we may assume without loss of generality that $
\lambda_e(\K^\times)C(e)\subseteq L$. Since $C(e)\subseteq L\cap
G_e$ preserves both $\li$ and $\Lie(U)$, it must be that
$C(e)\subseteq L_{e_0}$. Since the group $C(e)$ is reductive, it
follows from Mostow's theorem that it is conjugate under $L$ to a
subgroup of $C(L,e_0)$. Thus no generality will be lost by assuming
further that $C(e)\subseteq C(L,e_0)$.

\medskip

\noindent (b) In \cite[6.3]{Lo3}, Losev used the techniques of
quantum Hamiltonian reduction to define a completion $U(\li,e_0)'$
of the finite $W$-algebra $U(\li,e_0)$  and an injective algebra
homomorphism $\Xi\colon\,U(\g,e)\rightarrow U(\li,e_0)'$. By
construction, the reductive group $C(L,e_0)$ acts on $U(\li,e_0)'$
by algebra automorphisms. Since $C(e)\subseteq C(L,e_0)$, one can
see by inspection that all maps involved in Losev's construction are
$C(e)$-equivariant (a related discussion can also be found in
\cite[2.5]{Lo4}). This implies, in particular, that in our situation
Losev's homomorphism $\Xi$ is $C(e)$-equivariant. Here $C(e)$
operates on $U(\g,e)$ as in Subsection~\ref{3.1} and the action of
$C(e)$ on $U(\li,e_0)'$ is given by inclusion $C(e)\subseteq
C(L,e_0)$.

\medskip

 \noindent
(c) Given an associative algebra $\mathcal A$ over $\K$ and a
positive integer $d$ we denote by $\mathcal{A}^{(d)}$ the quotient
of $\mathcal A$ by its two-sided ideal generated by all
$s_{2d}(a_1,a_2,\ldots, a_{2d})$ with $a_i\in \mathcal{A}$, where
$$s_{2d}(X_1,X_2,\ldots, X_{2d})=\sum_{\sigma\in\mathfrak{S}_{2d}}{\rm sgn}(\sigma)X_{\sigma(1)}X_{\sigma(2)}\cdots X_{\sigma(2d)}.
$$
According to \cite[Proposition~6.5.1]{Lo3}, the inclusion
$U(\li,e_0)\hookrightarrow U(\li,e_0)'$ induces an algebra
isomorphism algebras ${U(\li,e_0)'}^{(d)}\cong U(\li,e_0)^{(d)}$.
Therefore, for every $d\in\mathbb{N}$ the map $\Xi$ gives rise to a
$C(e)$-equivariant algebra homomorphism $U(\g,e)^{(d)}\rightarrow
U(\li,e_0)^{(d)}$. Since $U(\g,e)^{(1)}\cong U(\g,e)^{\rm ab}$ and
$U(\li,e_0)^{(1)}\cong U(\li,e_0)^{\rm ab}$ as algebras, we thus
obtain a $C(e)$-equivariant algebra homomorphism
$\xi\colon\,U(\g,e)^{\rm ab}\rightarrow U(\li,e_0)^{\rm ab}$.

Let $\widetilde{\mathcal{E}}_0={\rm Specm}\,U(\li,e_0)^{\rm ab}$.
According to \cite[Theorem~6.5.2]{Lo3} the morphism of affine
varieties $\xi^*\colon\,\widetilde{\mathcal{E}}_0
\rightarrow\mathcal{E}$ associated with $\xi$ is finite. In
particular, it has finite fibres. Since $\xi$ is $C(e)$-equivariant
and $C(e)^\circ$ acts trivially on both $U(\g,e)^{\rm ab}$ and
$U(\li,e)^{\rm ab}$, the morphism $\xi^*$ maps
$\widetilde{\mathcal{E}}_0^\Gamma$ into $\mathcal{E}^\Gamma$. It
follows that
$$\dim\,\widetilde{\mathcal{E}}_0^\Gamma=\dim\,\xi^*(\widetilde{\mathcal{E}}_0^\Gamma)\le \dim\,\mathcal{E}_\Gamma.$$

\smallskip

\noindent (d) Write $\z$ for the centre $\z(\li)$ of the Levi
subalgebra $\li$. Clearly, $\z$ is a toral subalgebra of $\g$ and
$\li=\z\oplus [\li,\li]$. It follows that $U(\li,e_0)\cong
S(\z)\otimes U([\li,\li],e_0)$. This, in turn, implies that
$U(\li,e)^{\rm ab}\cong S(\z)\otimes U([\li,\li],e_0)^{\rm ab}$ as
algebras.  Since the subalgebra $U([\li,\li],e_0)$ of $U(\li, e_0)$
is stable under the action of $C(e)$ on $U(\li,e_0)$, we have a
natural action of $\Gamma$ of the affine variety
$\mathcal{E}_0:={\rm Specm}\,U([\li,\li],e_0)^{\rm ab}$. Since
$C(e)\subseteq C(L,e_0)$, and $C(L,e_0)^\circ$ acts trivially on
$\mathcal{E}_0$, the variety  $\mathcal{E}_0^\Gamma=\{\eta\in
\mathcal{E}_0\,:\,\,\gamma(\eta)=\eta \mbox{ for all }
 \gamma\in\Gamma\}$ contains $\mathcal{E}_0^{\Gamma_0}$ and hence  is non-empty by
our assumption on $e$.

Note that $L$ acts trivially on the centre $\z$ of $\Lie(L)$ and
hence so does $C(e)\subset L$. It follows that
$\widetilde{\mathcal{E}}_0^\Gamma\cong\z^*\times
\mathcal{E}_0^\Gamma$ as affine varieties. In particular,
$$\dim\,\widetilde{\mathcal{E}}_0^\Gamma
=\dim\,\mathcal{E}_0^\Gamma+\dim\,\z\ge \dim\,\z.$$ But then
$\mathcal{E}^\Gamma\supseteq\xi^*(\widetilde{\mathcal{E}}^\Gamma)\neq
\emptyset$ and
$\dim\,\mathcal{E}^\Gamma\ge\dim\,\widetilde{\mathcal{E}}_0^\Gamma\ge
\dim\,\z(\li)$ as claimed.
\end{proof}
\begin{rem}
{\rm It is established in \cite{Pr} that the variety $\mathcal{E}^\Gamma$ is non-empty for any nilpotent element in a finite dimensional simple Lie algebra over $\mathbb{C}$. This implies, among other things, that the condition that $\mathcal{E}_0^{\Gamma_0}\ne \emptyset$ imposed in the statement of Proposition~\ref{mult1} can be dropped.}
\end{rem}
\subsection{Describing the varieties ${\mathcal E}^\Gamma$ for classical Lie algebras}
In this subsection we assume that $\g$ is either $\mathfrak{so}_N$
or $\mathfrak{sp}_N$ and $e$ is an arbitrary nilpotent element of
$\g$. It is quite surprising that in this setting we have a very
uniform description of the algebra $U(\g,e)_\Gamma^{\rm ab}$. We
call a partition
$\lambda=(\lambda_1,...,\lambda_n)\in
\mathcal{P}_1(N)$ {\it exceptional} if there exists a $k\le n$ such that the parts
$\lambda_{k},\lambda_{k+1}$ are odd and the parts $\lambda_i$
with $i\not\in \{k, k+1\}$ are all even. Note that
$\Delta(\lambda)=\{(k, k+1)\}$ and $\Delta_{\rm
bad}(\lambda)=\emptyset$, which shows that all exceptional
partitions in $\mathcal{P}_1(N)$ are non-singular. Using the KS
algorithm it is straightforward to see that any nilpotent element of
$\g$ associated with an exceptional partition $\lambda$ is
Richardson (i.e. is induced from $0$).

We call a nilpotent element $e\in \g$ {\it almost rigid} if the
partition
$\lambda=(\lambda_1,\ldots,\lambda_n)\in\mathcal{P}_\epsilon(N)$ of
$e$ has the property that $\lambda_i-\lambda_{i+1}\in\{0,1\}$ for
all $i\le n$ (recall that $\lambda_j=0$ for $j>n$ by convention).
Since any such partition has no bad 2-steps, all almost rigid
nilpotent elements of $\g$ are non-singular.
 \begin{thm}\label{class1}
Let $\lambda=(\lambda_1,\ldots,\lambda_n)\in\mathcal{P}_\epsilon(N)$
and let $e=e(\lambda)$ be a nilpotent element of $\g$ associated
with $\lambda$. Then $U(\g,e)_\Gamma^{\rm ab}$ is isomorphic to a
polynomial algebra in $s(\lambda)$ variables unless $\g$ is of type
${\sf D}$ and $\lambda\in \mathcal{P}_\epsilon(N)$ is
exceptional, in which case $U(\g,e)_\Gamma^{\rm ab}$ is a polynomial
algebra in $s(\lambda)+1=(\lambda_1-\lambda_n+1)/2$ variables.
\end{thm}
\begin{proof} We denote by $\Oo$ the
$G$-orbit of $e=e(\lambda)$, adopt the notation of
Subsection~\ref{1.1} pertaining to $\g_e$, and choose the subspaces
$V[i]$, $1\le i\le n$, as in the proof of Lemma~\ref{nilpotents}. In
proving the theorem we may and will assume that $G=SL(V)\cap
G(\Psi)$ where $G(\Psi)$ is the stabiliser in $GL(V)$ of the
bilinear form $\Psi=(\,\cdot\,,\,\cdot)$.

\smallskip

\noindent (a) Let $I=\{1\le i\le n\colon\,i'=i,\,
\lambda_i>\lambda_{i+1}\}$ and set $\nu(\lambda):=|I|$. Note that
$\nu(\lambda)$ is the number of {\it distinct} $\lambda_i$ with
$i=i'$. For $i\in I$ we let $g_i$ denote the involution in $G(\Psi)$
such that $g_i(e^sv_j)=(-1)^{\delta_{i,j}}e^s(w_j)$ for all $1\le
j\le n$ and $0\le s\le \lambda_j$, where $\delta_{i,j}$ is the
Kronecker delta, and define $\widetilde{\Gamma}:= \langle
g_i\colon\,i\in I\rangle$, a subgroup of $G(\Psi)$. As the
involutions $g_i$ pairwise commute, $\widetilde{\Gamma}$ is an
elementary abelian $2$-group of order $2^{\nu(\lambda)}$. Using
\cite[3.8, 3.13]{Ja04} it it is straightforward to see that the
centraliser $G_e$ is generated by $\widetilde{\Gamma}\cap SL(V)$ and
$G_e^\circ$ (see also \cite[Theorem~2.7$'$]{McG}).

Let $\overline{\HH}_0$ denote the image of $\HH_0$ in
$\mathfrak{c}_e=\g_e/[\g_e,\g_e]$. Since $\HH_0$ is spanned by
elements that preserve every subspace $V[i]$ with $1\le i\le n$,
direct verification shows that the group $\widetilde{\Gamma}$ acts
trivially on $\overline{\HH}_0$. Since $G_e^\circ$ acts trivially on
$\mathfrak{c}_e$ and $G_e= \big(\widetilde{\Gamma}\cap
SL(V)\big)\cdot G_e^\circ$ by our preceding remark, we now deduce
that $\overline{\HH}_0\subseteq \mathfrak{c}_e^\Gamma$. In view of
Corollary~\ref{expression} this yields
$$\dim\,\mathfrak{c}_e^\Gamma\ge \dim (\HH_0/\HH_0^+)=s(\lambda).$$

The proof of Corollary~\ref{expression} also shows that the images
of $\zeta_i^{i+1,\lambda_{i+1}-1}$ with $(i,i+1)\in\Delta(\lambda)$ in the
quotient space $\overline{\mathfrak{c}}_e:=
\mathfrak{c}_e/\overline{\HH}_0$ form a $\K$-basis of
$\overline{\mathfrak{c}}_e$. Note that
$g_{i+1}\in\widetilde{\Gamma}$ for every 2-step $(i,i+1)$ of
$\lambda$ and, moreover, $g_i,g_{i+1}\in\widetilde{\Gamma}$ if
$(i,i+1)$ is a 2-step of $\lambda$ such that $\lambda_i\ne
\lambda_{i+1}$. If $(j, j+1) \in \Delta(\lambda)$ then direct computation shows that
\begin{equation}\label{zeta}
g_i\big(\zeta_j^{j+1,\lambda_{j+1}-1}\big)=\begin{cases}\ \,
\zeta_j^{j+1,\lambda_{j+1}-1}\ \quad \mbox{ if }
j\notin \{i-1,i\},
\\-\zeta_j^{j+1,\lambda_{j+1}-1} \ \,\ \,\mbox{ if } j \in \{i-1,i\}.\end{cases}
\end{equation}

Let $\bar{\zeta}_i^{i+1,\lambda_{i+1}-1}$ denote the image of
$\zeta_i^{i+1,\lambda_{i+1}-1}$ in $\overline{\mathfrak c}_e$ and
suppose that $$g_ig_j\Big(\sum_{(k,k+1)\in\Delta(\lambda)}\,
a_k\bar{\zeta}_k^{k+1,\lambda_{k+1}-1}\Big)=\sum_{(k,k+1)\in
\Delta(\lambda)}\,
a_k\bar{\zeta}_k^{k+1,\lambda_{k+1}-1}\qquad\quad\big(\forall\
i,j\in I\big)$$  for some $a_k\in\K$. Set
$\alpha:=\sum_{(k,k+1)\in\Delta(\lambda)}\,
a_k\bar{\zeta}_k^{k+1,\lambda_{k+1}-1}$. If $a_{k_1}\ne 0$ and
$a_{k_2}\ne 0$ for some $(k_1,k_1+1),(k_2,k_2+1)\in\Delta(\lambda)$ with
$k_1+1<k_2$, then it follows from (\ref{zeta}) that
$g_{k_1+1}g_{k_2+1}(\alpha) \ne\alpha$, a contradiction. If $a_k\ne
0$ and $a_{k+1}\ne 0$ for some $(k,k+1),(k+1,k+2)\in\Delta(\lambda)$, then
$\lambda_{k-1}>\lambda_k>\lambda_{k+1}\ge \lambda_{k+2}>
\lambda_{k+3}$ if $k>1$ and $\lambda_1>\lambda_{2}\ge \lambda_{3}>
\lambda_{4}$ if $k=1$, which implies that
$g_k,g_{k+2}\in\widetilde{\Gamma}$. Since $g_kg_{k+2}(\alpha)
\ne\alpha$ by (\ref{zeta}), we now deduce that
$\alpha=a_k\bar{\zeta}_k^{k+1,\lambda_{k+1}-1}$ for some
$(k,k+1)\in\Delta(\lambda)$. If $a_k\ne 0$ and $I$ is not contained in
$\{k,k+1\}$ then it is straightforward to see that
$g_ig_{k+1}(\alpha)\ne \alpha$ for any $i\in I\setminus\{k, k+1\}$.
Therefore, $\alpha\ne 0$ implies that $I\subseteq\{k,k+1\}$ where
$(k,k+1)$ is a 2-step of $\lambda$.

If $\g$ is not of type $\sf C$ and $\alpha\ne 0$, then the above
implies that $\Delta(\lambda)=\{(k,k+1)\}$ for some $k<n$ and all
$\lambda_i$ with $i\not\in\{k,k+1\}$ are even. So $\lambda$ is
exceptional and $h$ has type $\sf D$ in this case. Finally, if $\g$ is of type $\rm C$ then $\det(g_i)=1$
for all $i\in I$ and hence
$\widetilde{\Gamma}=\widetilde{\Gamma}\cap SL(V)$. Furthermore, if
$(k,k+1)$ is a 2-step of $\lambda$ then $g_{k+1}(\alpha)=-\alpha$ by
(\ref{zeta}). So in type $\sf C$ it must be that $\alpha=0$.

As a result of these deliberations we obtain that $\mathfrak{c}_e^\Gamma=\overline{\HH}_0$
and $\dim\,\mathfrak{c}_e^\Gamma=s(\lambda)$
unless $\lambda$ is an exceptional partition in $\mathcal{P}_1(N)$, in which case
$\dim\,\mathfrak{c}_e^\Gamma=s(\lambda)+1$.

\smallskip

\noindent(b) Suppose that $\lambda_k-\lambda_{k+1}\ge 2$ for some
$k\in\{1,\ldots, n\}$. For each $i\in\{1,\ldots, k\}$ we denote by
$V'[i]$ the linear span of all $e^s(w_i)$ with $1\le
s\le\lambda_i-2$ and set
$$V_k:=\Big(\textstyle{\bigoplus}_{i=1}^k\,V'[i]\Big)
\bigoplus\Big(\textstyle{\bigoplus}_{i>k}\,V[i]\Big),$$ a
non-degenerate subspace of $V$ with respect to $\Psi$. The
stabiliser $L_k$ of $V_k$ in $G$ is a Levi subgroup of $G$ and
$\li_k:=\Lie(L_k)$ is isomorphic to $\mathfrak{gl}_{k}\times
\mathfrak{m}_k$ where $\mathfrak{m}_k$ is a Lie algebra of the same
type as $\g$.

Let $t_k$ be the semisimple element of $\g$ with $\Ker\,t_k=V_k$
such that $t_k(w_i)=-w_i$ and $t_k(e^{\lambda_i-1}w_i)=
e^{\lambda_i-1}w_i$ for all $1\le i\le k$. It is straightforward to
see that $t_k$ spans the $1$-dimensional centre of the Lie algebra
$\li_k$. Let $e_k$ be the nilpotent element of $\li_k$ with the
property that $e_k(w_i)=e_k(e^{\lambda_i-2}w_i)=e_k(e^{\lambda_i-1}w_i)=0$ for $1\le i\le k$
and $e_k(e^sw_i)=e^{s+1}w_i$ for all $(i,s)$ with $i>k$ and $0\le
s\le \lambda_i-1$ and all $(i,s)$ with $1\le i\le k$ and $1\le s\le
\lambda_i-3$. By construction, $e_k\in\mathfrak{m}_k$.

In view of Corollary~\ref{case1or2}, passing from $(\g,e)$ to
$(\li_k,e_k)$ corresponds to applying Case~1 of the KS algorithm at
index $k$. Hence the orbit $\Oo$ lies in the Zariski closure of the
decomposition class $(\Ad\,G)\cdot(e_k+\K^\times t_k)$ and $e\in{\rm
Ind}_{\li_k}^\g(\Oo_k)$ where $\Oo_k$ is the $L_k$-orbit of $e_k$.

Let $W_k$ be the span of all $e^{\lambda_i-1}w_i$ with $1\le i\le k$
and set $\widetilde{V}_k:=V_k\oplus W_k$. Let $P_k$ be the parabolic
subgroup of $G$ which stabilises the partial flag $V\supset
\widetilde{V}_k\supset W_k$ in $V$. Using the description of $\g_e$
in Subsection~\ref{1.1} it is immediate to that
$\g_e\subset\Lie(P_k)$ which, in turn, implies that
$G_e^\circ\subset P_k$. Since $L_k$ is contained in $P_k$ as well
and $\widetilde{\Gamma}\cap SL(V) \subset L_k$ by the definition of
$\widetilde{\Gamma}$, we now obtain $G_e=\big(\widetilde{\Gamma}\cap
SL(V)\big)\cdot G_e^\circ\subset P_k$.

\smallskip

\noindent (c) The maximality of $L_k$ in the class of Levi subgroups
of $G$ yields that $L_k$ is a Levi subgroup of $P_k$ and
$P_k=L_kU_k$ where $U_k=R_u(P_k)$. Furthermore, our discussion in
Part~(b) implies that $e\in \Lie(P_k)$ is induced from $e_k\in\li_k$
and $G_e\subset (L_k)_{e_k}\cdot U_k$.

Continuing the process described in Part~(b) as many times as
possible and using the transitivity of induction stated in
Proposition~\ref{ind}(3) we shall eventually arrive at a parabolic
subgroup $P=LU$ of $G$ and a nilpotent element $e_0\in\li=\Lie(L)$
such that $G_e\subset P$ and $e\in{\rm Ind}_\li^\g(\Oo_0)$ where
$\Oo_0=(\Ad\,L)e_0$. From the description in Part~(b) it follows
that $\li\cong \bar{\li}\oplus\mathfrak{m}$ where $\mathfrak{m}$ has
the same type as $\g$ and $\bar{\li}$ is a Lie-algebra direct sum of
$s(\lambda)$ copies of various $\mathfrak{gl}_{k_i}$ with
$k_i\in\N$. Since the process terminates at the $s(\lambda)^{\rm
th}$ step, the nilpotent element $e_0\in\mathfrak{m}$ must be almost
rigid and hence non-singular.

Let $M$ be the special orthogonal or symplectic group with
$\Lie(M)=\mathfrak{m}$ and denote by $\Gamma(0)$ the component group
$M_{e_0}/M_{e_0}^\circ$. Let $\lambda_0$ be the partition of
$e_0\in\mathfrak{m}$. If $\lambda_0$ is not exceptional, then
combining Corollary~\ref{izosim} with Corollary~\ref{poly}(i) we
deduce that $U(\mathfrak{m},e_0)^{\rm ab}_{\Gamma(0)}\cong
S\big(\mathfrak{c}_{e_0}^{\Gamma(0)}\big)$. Since in the present
case $\mathfrak{c}_{e_0}^{\Gamma(0)}=\{0\}$ by our discussion in
Part~(a) (applied to $e_0\in\mathfrak{m}$) we conclude that the
maximal spectrum of $U(\mathfrak{m},e_0)^{\rm ab}_{\Gamma(0)}$ is a
single point! Note that $U([\li,,\li],e_0)\cong
U([\bar{\li},\bar{\li}])\otimes U(\mathfrak{m},e_0)$ as
$\K$-algebras and both tensor factors are stable under the natural
action of the reductive part of $L_{e_0}$ on $U([\li,\li],e_0)$.
  Proposition~\ref{mult1} now yields that
$\mathcal{E}^\Gamma\neq\emptyset$ and $\dim\,\mathcal{E}^\Gamma\ge \dim\,\z(\li)$.
On the other hand,
$\dim\,\z(\li)=s(\lambda)=\dim\,\mathfrak{c}_\Gamma(e)$ by our discussion in Part~(a) applied to $e\in\g$.
In this situation Corollary~\ref{poly}(ii) yields that $U(\g,e)^{\rm ab}_\Gamma\cong
S(\mathfrak{c}_e^\Gamma)\cong \K[X_1,\ldots, X_{s(\lambda)}]$.

\smallskip

\noindent (d) Finally, suppose that $\lambda_0$ is exceptional.
Since we only applied Step~1 of the KS algorithm to reach $e_0$, so
must be $\lambda$. In particular, $\g$ is of type $\sf D$.
Since $e_0$ is almost rigid and $\lambda_0$ is
exceptional, we have that $\lambda_0=(2,\ldots,2,1,1)$. Then
$\Gamma(0)=\{1\}$ which enables us to apply Step~2 of the KS
algorithm. After doing so we arrive at a parabolic subgroup
$P'=L'U'\subset P$ such that the centre of $\Lie(L')$ has dimension
$s(\lambda)+1$ and $e\in \Lie(U')$ is a Richardson element of
$\Lie(P')$. Since $\Gamma(0)=\{1\}$ and $e_0$ is a Richardson
element of $\li\cap\Lie(P')$, we also have that $$G_e\subset
L_{e_0}U=L_{e_0}^\circ U\subset (L\cap P')\cdot U\subseteq P'.$$ Let
$\li'=\Lie(L')$ and adopt the notation of Proposition~\ref{mult1}.
Since the augmentation ideal of the finite $W$-algebra
$U([\li',\li'],0)=U([\li',\li'])$ is  $(\Ad\,L')$-stable, we have
that $\mathcal{E}_0^{\Gamma_0}\ne \emptyset$. Applying
Proposition~\ref{mult1} now yields that $\mathcal{E}^\Gamma\ne
\emptyset$ and $\dim\,\mathcal{E}^\Gamma\ge
\dim\z(\li')=s(\lambda)+1$. As $\dim\,\mathcal{E}^\Gamma=
c_\Gamma(e)$ by our discussion in Part~(a), Corollary~\ref{poly}(ii)
applies to $e$ showing that $U(\g,e)^{\rm ab}_\Gamma$ is isomorphic
to a polynomial algebra in $s(\lambda)+1$ variables.

The proof of the theorem is now complete.
\end{proof}

\subsection{Describing the varieties ${\mathcal E}^\Gamma$ for exceptional Lie algebras}
In this subsection, $G$ is an exceptional algebraic group of adjoint
type and $\g=\Lie(G)$, a Lie algebra of type ${\sf G_2}$, ${\sf
F_4}$, ${\sf E_6}$, ${\sf E_7}$ or ${\sf E_8}$. We assume that $e$
is an induced nilpotent element of $\g$ and we embed it into an
$\mathfrak{sl}_2$-triple $\{e,h,f\}\subset\g$. By the
$\mathfrak{sl}_2$-theory, all eigenvalues of ${\rm ad}\,h$ are
integers and $\g_e=\bigoplus_{i\ge 0}\,\g_e(i)$ where $\g(k)$
denotes the $k$-eigenspace of ${\rm ad}\,h$ and
$\g_e(k)=\g_e\cap\g(k)$. Since the derived subalgebra of $\g_e$ is
$({\rm ad}\,h)$-stable, the vector space
$\mathfrak{c}_e=\g_e/[\g_e,\g_e]$ carries a natural ${\mathbb
Z}_{\ge 0}$-grading:
$$\mathfrak{c}_e=\,\textstyle{\bigoplus}_{i\ge 0}\,\mathfrak{c}_e(i),\qquad\quad
\mathfrak{c}_e(i)\cong\,\g_e(i)/[\g_e,\g_e]\cap\g(i).$$
Let $P(e)$ be the parabolic subgroup of $G$ with
$\mathfrak{p}(e):=\bigoplus_{i\ge 0}\,\g(i)$. It is well known that
$P(e)$ is the optimal parabolic subgroup for the $G$-unstable vector
$e\in\g$ in the sense of the Kempf--Rousseau theory; see
\cite{Sasha0}, for example. In particular, $G_e\subset P(e)$.

Recall that $e$ is called {\it even} if all eigenvalues of ${\rm
ad}\,h$ are in $2{\mathbb Z}$. It follows from the
$\mathfrak{sl}_2$-theory that any even $e\in\g$ is a Richardson
element of $\mathfrak{p}(e)$. For $e$ even, we denote by $d(e)$ the
number of $2$'s on the weighted Dynkin diagram of $e$ (see the
second column of the tables in \cite[pp.~401--407]{C}). It is well
known (and easy to see) that $d(e)$ coincides with the dimension of
the centre of the Levi subalgebra $\g(0)$ of $\g$.

In what follows we are going to rely on the detailed information on
the centralisers of nilpotent elements obtained by Lawther and
Testerman in \cite{LT2}. In fact, we shall require the extended
version of \cite{LT2} which, due to its size, is only available as a
preprint; see \cite{LT1}.  We shall also rely on de Graaf's
computation of $c(e)=\dim\,\mathfrak{c}_e$ in \cite{deG2} and the
explicit description of sheets in exceptional Lie algebras obtained
by Borho \cite{Bor} (in type ${\sf F_4}$) and Elashvili \cite{ela}
(in type ${\sf E}$) and recently double-checked by computational
methods in \cite{deGE}.

The number of sheets containing an induced nilpotent element is
given in the third column of Tables~1--6 whilst their ranks can be
found in the fourth column. The numbers $c(e)$ are listed in the
fifth column. This information is taken from the tables in
\cite{deGE} and \cite{deG2} and included for the reader's
convenience. We should stress at this point that the last column of
Tables~1--6 contains new information which will only become
available {\it after} we establish the main results of this
subsection.

From now on we shall use freely the notation from the tables of \cite{LT1}.
\begin{prop}\label{even}
If $\g$ is an exceptional Lie algebra and $e$ is an even nilpotent
element of $\g$, then $d(e)=c_\Gamma(e)$ and $U(\g,e)^{\rm
ab}_\Gamma\cong S(\mathfrak{c}_e^\Gamma)$ is isomorphic to a
polynomial algebra in $d(e)$ variables.
\end{prop}
\begin{proof}
(a) Since $e$ is a Richardson element of $\mathfrak{p}(e)$ and
$G_e\subset P(e)$, Proposition~\ref{mult1} implies that in the
present case $\mathcal{E}^\Gamma\ne \emptyset$ and
$\dim\,\mathcal{E}^\Gamma\ge d(e)$.

\smallskip

\noindent (b) If $e$ lies in a single sheet of $\g$ then inspecting
Tables~1-6 reveals that $c(e)=r(e)$ in all cases. Since $e$ is even we must have $r(e)=d(e)$.
Applying Corollary~\ref{poly}(i) then yields that there is a
$\Gamma$-equivariant algebra isomorphism $U(\g,e)^{\rm ab}\,\cong
S(\mathfrak{c})$ and $U(\g,e)^{\rm ab}_\Gamma\,\cong
S(\mathfrak{c}_e^\Gamma)$ as algebras. On the other hand,
$\mathcal{E}^\Gamma\neq \emptyset$ and $\dim\,{\mathcal E}^\Gamma\ge
d(e)$ by Part~(a). From this it follows that
$\mathfrak{c}_e^\Gamma=\mathfrak{c}_e$, that is $\Gamma$ acts
trivially on $\mathfrak{c}_e$, forcing $$U(\g,e)^{\rm ab}\,\cong
U(\g,e)_\Gamma^{\rm ab}\,\cong \K[X_1,\ldots,X_{d(e)}].$$

From now on we assume that $e$ lies in more than one sheet of $\g$.
According to Tables~1--6, in this case we always have that
$\Gamma\ne \{1\}$. By Part~(a) and our discussion in Subsection~\ref{2.1}, at least one of the sheets
containing $e$ must have rank $d(e)$ but it may happen that
$r(e)>d(e)$.

\smallskip

\noindent (c) In this part we assume that $\Gamma\cong S_2$.
Inspecting Tables 1--6 one finds out that in this case
$c(e)-d(e)\in\{1,2\}$ (the values of $d(e)$ can be found in
\cite[pp.~401--407]{C}, for example). If $c(e)=d(e)+1$, then
combining Tables~1--6 with \cite[pp.~405--407]{C} one observes that
the Dynkin label of $e$ is one of ${\sf E_8(b_4)}$, ${\sf
D_7(a_1)}$, ${\sf E_6(a_1)}$, ${\sf D_5+A_2}$, ${\sf E_6(a_3)}$ if
$\g$ is of type ${\sf E_8}$, one of ${\sf E_7(a_4)}$, ${\sf
E_6(a_3)}$, ${\sf A_4}$ if $\g$ is of type ${\sf E_7}$, and one of
${\sf F_4(a_1)}$, ${\sf F_4(a_2)}$  if $\g$ is of type ${\sf F_4}$.

Suppose $\g$ is of type ${\sf E_8}$. If $e$ has type ${\sf
E_8(b_4)}$ then $\mathfrak{c}_e(4)$ is $1$-dimensional by
\cite{deG2}. Then \cite[p.~290]{LT1} yields that $\Gamma$ is
generated by the image of $h_4(-1)$ and $\cc_e(4)$ is spanned by the
image of $v_3$. As $(\Ad\,c)(v_3)=-v_3$ we deduce that $\Gamma$ acts
nontrivially on $\cc_e$ so that $c_\Gamma(e)\le d(e)$. Combining Corollary~\ref{poly}(ii)
and Part~(a) we now deduce that
$U(\g,e)^{\rm ab}_\Gamma$ is a polynomial algebra in $d(e)$
variables.

If $e$ has type ${\sf D_7(a_1)}$ or ${\sf D_5+A_2}$ then
$\mathfrak{c}_e(0)$ is $1$-dimensional by \cite{deG2}. On the other
hand, it follows from \cite[pp.~263, 277]{LT1} that $\g_e(0)$ is a
$1$-dimensional toral subalgebra of $\g$ upon which $\Gamma$ acts
nontrivially. Then again $c_\Gamma(e)\le d(e)$ and we can argue as
before to conclude that $U(\g,e)^{\rm ab}_\Gamma$ is a polynomial
algebra in $d(e)$ variables.

Now suppose $e$ has type ${\sf E_6(a_1)}$. This case is more subtle
due to the complicated nature of the generator of $\Gamma$. From
\cite{deG2} we know that $\cc_e(4)$ is $1$-dimensional, whilst
\cite[p.~261]{LT1} yields that $\g_e(0)$ is simple and the largest
trivial $({\rm ad}\,\g_e(0))$-submodule of $\g_e(4)$ is spanned by
$v_8$. From this it follows that $\cc_e(4)$ is generated by the
image of $v_8$. By \cite[p.~261]{LT1}, the group $\Gamma$ is
generated by the image of
$$c=n_{\textstyle{{0\atop }\!{1\atop }\!{2\atop 1}\!
{2\atop }\!{2\atop}\!{1\atop }\!{0\atop }}}
n_{\textstyle{{1\atop }\!{1\atop }\!{2\atop 1}\!
{2\atop }\!{1\atop}\!{1\atop }\!{0\atop }}}
n_{\textstyle{{1\atop }\!{2\atop }\!{2\atop 1}\!
{1\atop }\!{1\atop}\!{1\atop }\!{0\atop }}}h_1(-1)
h_2(-1)h_3(-1)h_5(-1)h_6(-1)h_8(-1).
$$ Since the roots
${\textstyle{{0\atop }\!{1\atop }\!{2\atop 1}\! {2\atop
}\!{2\atop}\!{1\atop }\!{0\atop }}}, {\textstyle{{1\atop }\!{1\atop
}\!{2\atop 1}\! {2\atop }\!{1\atop}\!{1\atop }\!{0\atop }}},
{\textstyle{{1\atop }\!{2\atop }\!{2\atop 1}\! {1\atop
}\!{1\atop}\!{1\atop }\!{0\atop }}}$ and ${\textstyle{{0\atop
}\!{1\atop }\!{1\atop 0}\! {1\atop }\!{0\atop}\!{0\atop }\!{0\atop
}}}$ are pairwise orthogonal, we have that that
$(\Ad\,c)\Big(e_{\textstyle{{0\atop }\!{1\atop }\!{1\atop 0}\!
{1\atop }\!{0\atop}\!{0\atop }\!{0\atop
}}}\Big)=-e_{\textstyle{{0\atop }\!{1\atop }\!{1\atop 0}\! {1\atop
}\!{0\atop}\!{0\atop }\!{0\atop }}}$. Since $e_{\textstyle{{0\atop
}\!{1\atop }\!{1\atop 0}\! {1\atop }\!{0\atop}\!{0\atop }\!{0\atop
}}}$ occurs with a nonzero coefficient in the expression of $v_8$
via Chevalley generators of $\g$, this implies that
$(\Ad\,c)(v_8)=-v_8$. But then $c_\Gamma(e)\le d(e)$ and we can
argue as in the previous cases to establish the polynomiality of
$U(\g,e)^{\rm ab}_\Gamma$.

If $e$ has type ${\sf E_6(a_3)}$ then \cite{deG2} says that
$\cc_e=\mathfrak{c}_e(2)$ is $3$-dimensional. In view of
\cite[p.~234]{LT1} this means that $\cc_e\cong \g_e(2)$ as
$\Gamma$-modules (in the present case the group $C(e)^\circ$ acts
trivially on $\g_e(2)$). As $(\Ad\,c)(v_1)=-v_1$ we see that
$\Gamma$ acts non-trivially on $\cc_e$ implying $c_\Gamma(e)\le
d(e)$. But then again $U(\g,e)^{\rm ab}_\Gamma$ is a polynomial
algebra in $d(e)$ variables.

Suppose $\g$ is of type ${\sf E_7}$. If $e$ is of type ${\sf
E_7(a_4)}$ then $\mathfrak{c}_e=\mathfrak{c}_e(2)$ is
$4$-dimensional by \cite{deG2}. As $\dim\,\g_e(e)=4$ by
\cite[p.~155]{LT1}, the image of $v_1$ in $\mathfrak{c}_e$ is
nonzero. Since $\Gamma$ is generated by the image of $c=h_4(-1)$ and
$(\Ad\,c)(v_1)=-v_1$ by {\it loc.\,cit.}, we argue as before to
deduce that $U(\g,e)^{\rm ab}_\Gamma$ is a polynomial algebra in
$d(e)$ variables. The case where $\g$ is of type ${\sf E_7}$ and $e$
is of type ${\sf E_6(a_3)}$ is very similar. Here
$\mathfrak{c}_e=\mathfrak{c}_2(2)$ is $2$-dimensional by
\cite{deG2}, the group $\Gamma$ is again generated by the image of
$c=h_4(-1)$, the image of $v_1$ in $\mathfrak{c}_e$ is nonzero and
$(\Ad\,c)(v_1)=-v_1$ by \cite[p.~149]{LT1}.

If $e$ is of type ${\sf A_4}$ then $d(e)=2$ and
$\cc_e=\cc_e(0)\oplus \cc_e(2)\oplus \cc_e(6)$ and all nonzero
$\cc_e(i)$ are $1$-dimensional; see \cite{deG2}. By
\cite[pp.~133]{LT1}, the group $\Gamma$ is generated by the image of
$$c=n_{\textstyle{{0\atop }\!{1\atop }\!{1\atop 1}\! {1\atop
}\!{1\atop}\!{0\atop }}} n_{\textstyle{{1\atop }\!{1\atop }\!{1\atop
0}\! {1\atop }\!{1\atop}\!{0\atop }}}n_{\textstyle{{1\atop
}\!{2\atop }\!{2\atop 1}\! {2\atop }\!{1\atop}\!{1\atop
}}}n_{\textstyle{{1\atop }\!{2\atop }\!{4\atop 2}\! {3\atop
}\!{2\atop}\!{1\atop }}}h_2(-1)h_3(-1)h_4(-1)h_6(-1),$$ and
$\g_e(0)=[\g_e(0),\g_e(0)]\oplus\Lie(T_1)$ with
$[\g_e(0),\g_e(0)]\cong\mathfrak{sl}_3$ and and $\Lie(T_1)$ spanned
by the element $t\in\Lie(T)$ such that $\alpha_i(t)=\delta_{6,i}$
for all $1\le i\le 7$. Direct computation shows that $\Ad\,c$
negates $t$. But then $d(e)=2\ge c_\Gamma(e)$ and we can proceed as
before.

Suppose $\g$ is of type ${\sf F_4}$. If $e$ is of type ${\sf
F_4(a_1)}$ then $\dim\,\mathfrak{c}_e(4)=1$ by \cite{deG2}. In view
of \cite[p.~81]{LT1} this shows that the image of $v_2$ in
$\mathfrak{c}_e$ is nonzero. Since $\Gamma$ is generated by the
image of $c=h_4(-1)$ and $(\Ad\,c)(v_2)=-v_2$ by {\it loc.\,cit.},
the result follows. If $e$ is of type ${\sf F_4(a_2)}$ then
$\mathfrak{c}_2=\mathfrak{c}_e(2)$ is $3$-dimensional by
\cite{deG2}. As $\dim\,\g_2(2)=3$, the image of $v_1$ in
$\mathfrak{c}_e$ is nonzero. It remains to note that $\Gamma$ is
generated by the image of $c=h_2(-1)$ and $(\Ad\,c)(v_1)=-v_1$; see
\cite[p.~80]{LT1}.

If $c(e)=d(e)+2$, then combining Tables~1--6 with
\cite[pp.~405--407]{C}) one observes that the Dynkin label of $e$ is
one of ${\sf E_8(a_3)}$, ${\sf E_8(a_4)}$, ${\sf E_8(a_5)}$ if $\g$
is of type ${\sf E_8}$, one of ${\sf E_7(a_3)}$, ${\sf E_6(a_1)}$ if
$\g$ is of type ${\sf E_7}$, and ${\sf E_6(a_3)}$ if $\g$ is of type
${\sf E_6}$.

Suppose $\g$ is of type ${\sf E_8}$. If $e$ has type ${\sf
E_8(a_3)}$ then \cite{deG2} says that both $\mathfrak{c}_e(8)$ and
$\mathfrak{c}_e(16)$ are $1$-dimensional. In view of
\cite[p.~295]{LT1} this implies that the images of $v_3$ and $v_7$
in $\mathfrak{c}_e$ are linearly independent. Since {\it loc.\,cit.}
also shows that $\Gamma$ is generated by the image of $c=h_4(-1)$
and $ (\Ad\,c)(v_i)=-v_i$ for $i=3,7$, we deduce that
$c_\Gamma(e)\le d(e)$. Arguing as before we now conclude that in the
present case $U(\g,e)^{\rm ab}_\Gamma$ is a polynomial algebra in
$d(e)$ variables.

If $e$ has type ${\sf E_8(a_4)}$ then \cite{deG2} shows that both
$\mathfrak{c}_e(4)$ and $\mathfrak{c}_e(8)$ are $1$-dimensional.
Thanks to \cite[p.~293]{LT1} this yields that the images of $v_2$
and $v_4$ in $\mathfrak{c}_e$ are linearly independent. Since {\it
loc.\,cit.} also shows that $\Gamma$ is generated by the image of
$c=h_4(-1)h_8(-1)$ and $ (\Ad\,c)(v_i)=-v_i$ for $i=2,4$, we deduce
that $c_\Gamma(e)\le d(e)$. The result then follows as in the
previous case.

Now suppose $e$ is of type ${\sf E_8(a_5)}$. In this case we have to
work harder. First note that
$\mathfrak{c}_e=\mathfrak{c}_e(2)\oplus \mathfrak{c}_2(10)$ and
$\dim\,\mathfrak{c}_e(10)=2$ by \cite{deG2}. Since $e$ is
distinguished, $\g_e(2)$ maps isomorphically onto
$\mathfrak{c}_e(2)$. By \cite[pp.~288, 289]{LT1}, the group $\Gamma$
is generated by the image of $c=h_4(-1)h_7(-1)$ and $\g_e(2)$ has
basis $\{v_1,v_2, e\}$ such that $(\Ad\,c)(v_1)=-v_1$ and
$(\Ad\,c)(v_2)=v_2$. Since $\g_e(i)\subset [\g_e,\g_e]$ for
$i=4,6,8$ by \cite{deG2}, it follows from \cite[p.~288]{LT1} that
the subspaces $\g_e(4)=\K v_4$, $\g_e(6)=\K v_5$ and $\g_e(8)=\K
v_6$ are spanned by $[v_1,v_2]$, $[v_2,[v_2,v_1]]$ and
$[v_2,[v_2,[v_2,v_1]]]$, respectively (one should keep in mind here
that $(\Ad\,c)(v_4)=-v_4$, $(\Ad\,c)(v_5)=v_5$ and
$(\Ad\,c)(v_6)=-v_6$, which is immediate from \cite[p.~289]{LT1}).
Also, $[v_1,[v_1,v_2]]=0$. Since
$$\g(10)\cap [\g_e,\g_e]\,=\,[\g_e(2),\g_e(8)]+[\g_e(4),\g_e(6)],$$ the LHS is spanned by
$u_1:=[v_1,[v_2,[v_2,[v_2,v_1]]]]$,
$u_2:=[v_2,[v_2,[v_2,[v_2,v_1]]]]$ and
$u_3:=\big[[v_1,v_2],[v_2,[v_2,v_1]\big]$. As $[v_1,[v_2,v_1]]=0$,
the Jacobi identity yields
$u_3=u_1-[v_2,[v_1,[v_2,[v_2,v_1]]]]=u_1$. In view of \cite{deG2}
this implies that $u_1$ and $u_2$ form a basis of
$\g(10)\cap[\g_e,\g_e]$. Note that $(\Ad\,c)(u_1)=u_1$ and
$(\Ad\,c)(u_2)=-u_2$. Since it follows from \cite[p.~289]{LT1} (with
the misprint in the expression for $v_7$ corrected in
\cite[p.~179]{LT2}) that the kernel of $(\Ad\,c+{\rm Id})_{\vert
\g_e(10)}$ is $2$-dimensional, we are now able to conclude that
$c_\Gamma(e)\le d(e)$, which yields the desired result in the
present case.

Suppose $\g$ is of type ${\sf E_7}$.
If $e$ has type ${\sf E_7(a_3)}$ then
$\dim\,\cc_e(4)=\dim\,\cc_e(8)=1$ by \cite{deG2}, whilst \cite[p.~160]{LT1}
says that $\g_e(4)=\K v_3$, $\g_e(8)=\K v_6$
and $\Gamma$ is generated by the image of $c=h_4(-1)$. Since
$(\Ad\,c)(v_3)=-v_3$ and $(\Ad\,c)(v_6)=-v_6$,
this implies that $c_\Gamma(e)\le d(e)$ as wanted.

If $\g$ is of type ${\sf E_7}$ and $e$ has type ${\sf E_6(a_1)}$
then $d(e)=3$ by \cite[p.~404]{C} and
$\dim\,\cc_e(0)=\dim\,\cc_e(4)=1$ by \cite{deG2}. By
\cite[p.~158]{LT1}, we have that the reductive part
$\g_e(0)=\Lie(C(e))$ is $1$-dimensional and $\Gamma$ is generated by
the image of
$$c=n_{\textstyle{{0\atop }\!{1\atop }\!{2\atop 1}\!
{2\atop }\!{2\atop}\!{1\atop }}}
n_{\textstyle{{1\atop }\!{1\atop }\!{2\atop 1}
\!{2\atop }\!{1\atop}\!{1\atop }}}
n_{\textstyle{{1\atop }\!{2\atop }{2\atop 1}
\!{1\atop }\!{1\atop}\!{1\atop }}}h_1(-1)
h_2(-1)h_3(-1)h_5(-1)h_6(-1).
$$
Direct computation shows that $\Ad\,c$ acts as $-{\rm Id}$ on the
$1$-dimensional toral subalgebra $\g_e(0)$ and the basis vectors
$v_2,v_3\in\g_e(4)$ have nonzero weights with respect to the adjoint
action of the torus $C(e)^\circ$. Since $\dim\,\g_e(4)=3$, it
follows that the image of $v_4$ in $\cc_e$ is nonzero. As the roots
${\textstyle{{0\atop }\!{1\atop }\!{2\atop 1}\! {2\atop
}\!{2\atop}\!{1\atop }}}, {\textstyle{{1\atop }\!{1\atop }\!{2\atop
1}\! {2\atop }\!{1\atop}\!{1\atop }}}, {\textstyle{{1\atop
}\!{2\atop }\!{2\atop 1}\! {1\atop }\!{1\atop}\!{1\atop }}}$ and
${\textstyle{{0\atop }\!{1\atop }\!{1\atop 0}\! {1\atop
}\!{0\atop}\!{0\atop }\!{0\atop }}}$ are pairwise orthogonal, it
must be that that $(\Ad\,c)\Big(e_{\textstyle{{0\atop }\!{1\atop
}\!{1\atop 0}\! {1\atop }\!{0\atop}{0\atop
}}}\Big)=-e_{\textstyle{{0\atop }\!{1\atop }\!{1\atop 1}\! {1\atop
}\!{0\atop}\!{0\atop }}}$. As $e_{\textstyle{{0\atop }\!{1\atop
}\!{1\atop 0}\! {1\atop }\!{0\atop}\!{0\atop }}}$ occurs with a
nonzero coefficient in the expression of $v_4$ via Chevalley
generators of $\g$ we deduce that $(\Ad\,c)(v_4)=-v_4$. But then
$c_\Gamma(e)\le d(e)$ and we can argue as in the previous cases to
establish the polynomiality of $U(\g,e)^{\rm ab}_\Gamma$.

If $\g$ is of type ${\sf E_6}$ and $e$ has type ${\sf E_6(a_3)}$
then $\dim\,\cc_e(2)=3$ and $\dim\,\cc_e(4)=2$ by \cite{deG2},
whilst \cite[p.~100]{LT1} shows  $\Gamma$ is generated by the image
of $c=h_4(-1)$ and $\g_e(2)$ has basis $\{v_1,v_2,e\}$ such that
$(\Ad\,c)(v_1)=-v_1$ and $(\Ad\,c)(v_2)=v_2$. It is also immediate
from {\it loc.\,cit.} that $[\g_2(e),\g_2(e)]$ has dimension $1$.
Since  $\Ad\,c$ negates both $v_5$ and $v_6$ and these vectors are
linearly independent in $\g_e(4)$, we get $c_\Gamma(e)\le d(e)$
which yields the desired result in the present case.

\smallskip

\noindent (d) Next we assume that $\Gamma\cong S_3$. In this case
$e$ is one of ${\sf E_8(b_5)}$, ${\sf E_8(b_6)}$ or ${\sf D_4(a_1)}$
if $\g$ is of type ${\sf E_8}$, one of ${\sf E_7(a_5)}$ or ${\sf
D_4(a_1)}$ if $\g$ is of type ${\sf E_7}$, has type ${\sf D_4(a_1)}$
if $\g$ is of type ${\sf E_6}$ and has type ${\sf G_2(a_1)}$ if $\g$
is of type ${\sf G_2}$.

If $e$ is of type  ${\sf E_8(b_5)}$ then $d(e)=3$ and
$\cc_e=\cc_e(2)\oplus\cc_e(4)\oplus\cc_e(10)$ where
$\dim\,\cc_e(2)=4$, $\dim\,\cc_e(6)=2$ and $\dim\,\cc_e(10)=1$; see
\cite{deG2}. On the other hand, \cite[pp.~285, 286]{LT1} shows that
$\dim\,\g_e(2)=4$ and $\dim\,\g_e(6)=2$ implying that the canonical
homomorphism $\g_e\rightarrow \g_e/[\g_e,\g_e]$ is bijective on
$\g_e(2)\oplus\g_2(6)$. It also follows from {\it loc.\,cit.} that
$\Gamma$ contains the image of
$c_1=h_1(\omega)h_2(\omega)h_5(\omega^2)$ where $\omega$ is a third
primitive root of $1$. Direct computation shows that
$(\Ad\,c_1)(v_1)=\omega v_1$, $(\Ad\,c_1)(v_2)=\omega^2 v_2$,
$(\Ad\,c_1)(v_6)= \omega^2 v_6$ and $(\Ad\,c_1)(v_7)=\omega v_7$.
From this it is immediate that $c_\Gamma(e)\le 7-4=d(e)$. So we can
argue as before to deduce the polynomiality of $U(\g,e)_\Gamma^{\rm
ab}$.

If $e$ is of type  ${\sf E_8(b_6)}$ then $d(e)=2$ and
$\cc_e=\cc_e(2)\oplus\cc_e(4)$ where $\dim\,\cc_e(2)=4$ and
$\dim\,\cc_e(4)=1$; see \cite{deG2}. By \cite[p.~275]{LT1},
$\dim\,\g_e(2)=\dim\,\g_e(4)=4$ and $\Gamma$ is generated by
$c_1=h_1(\omega)h_2(\omega)h_5(\omega^2)$, where $\omega$ is a third
primitive root of $1$, and by
$$c_2=n_{\textstyle{{1\atop }\!{0\atop }\!{0\atop 0}\!
{0\atop }\!{0\atop }\!{0\atop }\!{0\atop }}} n_{\textstyle{{0\atop
}\!{0\atop }\!{0\atop 1} \!{0\atop }\!{0\atop }\!{0\atop }\!{0\atop
}}} n_{\textstyle{{0\atop }\!{0\atop }\!{0\atop 0} \!{1\atop
}\!{0\atop }\!{0\atop }\!{0\atop }}}h_2(-1)
h_3(-1)h_4(-1)h_5(-1)h_8(-1).$$ Direct verification shows that
$(\Ad\,c_1)(v_1)=\omega v_1$, $(\Ad\,c_1)(v_2)=\omega^2v_2$ and
$(\Ad\,c_1)(v_3)=v_3$, which in view of {\it loc.\,cit.} implies
that $\dim\,\cc_e(2)^\Gamma\le 2$. Similarly,
$(\Ad\,c_1)(v_6)=\omega v_6$, $(\Ad\,c_1)(v_7)=\omega^2 v_7$,
$((\Ad\, c_1)(v_5)=v_5$ and $(\Ad\,c_1)(v_8)=v_8$. Since
$\{v_1,v_2,v_3,e\}$ and $\{v_5,v_6,v_7,v_8\}$ are bases of $\g_e(2)$
and $\g_e(4)$, respectively, and $\dim\,\cc_e(4)=1$ by our earlier
remark, the vectors $[v_1,v_2]$, $[v_1,v_3]$ and $[v_2,v_3]$ must
form a basis of $\g_2(4)\cap [\g_e,\g_e]$. Comparing the respective
eigenvalues for $\Ad\,c_1$ yields $v_6,v_7\in [\g_e,\g_e]$. Using
the explicit formulae for $v_1$ and $v_2$ in \cite[p.~276]{LT1} one
observes that $e_{\textstyle{{1\atop }\!{2\atop }\!{2\atop 1}\!
{2\atop }\!{2\atop }\!{1\atop }\!{0\atop }}} $ occurs with
coefficient $\pm 3$ in the expression of $[v_1,v_2]$ via Chevalley
generators of $\g$. As a consequence, $v_5+\lambda v_8\in
[\g_e,\g_e]$ for some $\lambda\in\K$ implying that $\cc_e(4)$ is
generated by the image of $v_8=e_{\textstyle{{1\atop }\!{2\atop
}\!{2\atop 1}\! {1\atop }\!{0\atop}\!{0\atop }\!{0\atop }}}$. Since
the roots
$$\textstyle{{1\atop }\!{0\atop }\!{0\atop 0}\!
{0\atop }\!{0\atop}\!{0\atop }\!{0\atop }}, \textstyle{{0\atop
}\!{0\atop }\!{0\atop 1}\! {0\atop }\!{0\atop}\!{0\atop }\!{0\atop
}}, \,\textstyle{{0\atop }\!{0\atop }\!{0\atop 0}\! {1\atop
}\!{0\atop}\!{0\atop }\!{0\atop }}, \,\textstyle{{1\atop }\!{2\atop
}\!{2\atop 1}\! {1\atop }\!{0\atop}\!{0\atop }\!{0\atop }}$$ are
pairwise orthogonal, we have that that $(\Ad\,c_2)(v_8)=-v_8$. But
then $c_\Gamma(e)=\dim\,\cc_e(2)^\Gamma\le 2=d(e)$ and we can argue
as in the previous cases.

Now suppose $e$ has type ${\sf D_4(a_1)}$. The $d(e)=1$ by
\cite[pp.~402, 403, 405]{C}. If $\g$ is of type ${\sf E_8}$ then
$\cc_e=\cc_e(2)$ has dimension $3$ by \cite{deG2} and
$[\g_e(0),\g_e(2)]$ has codimension $3$ in $\g_e(2)$ by
\cite[pp.~190, 191]{LT1}. This implies that $\cc_e$ is generated by
the images of $v_{25}$, $v_{26}=e_2+e_5$ and $v_{27}=e$. By {\it
loc.\,cit.}, the group $\Gamma$ is generated by the images of
$c_1=n_{\textstyle{{1\atop }\!{1\atop }\!{1\atop 1}\! {0\atop
}\!{0\atop }\!{0\atop }\!{0\atop }}} n_{\textstyle{{1\atop
}\!{1\atop }\!{1\atop 0} \!{1\atop }\!{0\atop }\!{0\atop }\!{0\atop
}}} h_2(-1)$ and $c_2=(n_{\textstyle{{1\atop }\!{2\atop }\!{2\atop
1}\! {1\atop }\!{1\atop }\!{0\atop }\!{0\atop }}}
n_{\textstyle{{1\atop }\!{1\atop }\!{2\atop 1} \!{2\atop }\!{1\atop
}\!{0\atop }\!{0\atop }}} h_1(-1)h_2(-1)h_6(-1))^g$ where
$$g=x_{\textstyle{{0\atop }\!{0\atop }\!{1\atop 0}\!
{0\atop }\!{0\atop }\!{0\atop }\!{0\atop }}}
\textstyle{(\frac{1}{3}})n_{\textstyle{{0\atop }\!{0\atop }\!{1\atop 0}\!
{0\atop }\!{0\atop }\!{0\atop }\!{0\atop }}}
h_1(4)h_2(-4)h_3(16)h_4(-48)h_5(16)h_6(-8)
 x_{\textstyle{{0\atop }\!{0\atop }\!{1\atop 0}\!
{0\atop }\!{0\atop }\!{0\atop }\!{0\atop }}}
(\textstyle{-\frac{1}{3}}).
$$
 Since $\Ad\,c_1$ fixes $e$ and permutes the lines $\K e_2$ and $\K e_5$, it must permute
 $e_2$ and $e_5$. But then $(\Ad\, c_1)(v_{26})=v_{26}$. Similarly, $\Ad\,c_1$ must permute
 $e_{\textstyle{{0\atop }\!{0\atop }\!{1\atop 1}\!
{0\atop }\!{0\atop}\!{0\atop }\!{0\atop }}}$ and
$e_{\textstyle{{0\atop }\!{0\atop }\!{1\atop 0}\! {1\atop
}\!{0\atop}\!{0\atop }\!{0\atop }}}$. Since the roots
$\textstyle{{1\atop }\!{1\atop }\!{1\atop 1}\! {0\atop
}\!{0\atop}\!{0\atop }\!{0\atop }}, \textstyle{{1\atop }\!{1\atop
}\!{1\atop 0}\! {1\atop }\!{0\atop}\!{0\atop }\!{0\atop }},
\,\textstyle{{0\atop }\!{1\atop }\!{1\atop 0}\! {0\atop
}\!{0\atop}\!{0\atop }\!{0\atop }}$ are pairwise orthogonal and
$e_{\textstyle{{0\atop }\!{1\atop }\!{1\atop 0}\! {0\atop
}\!{0\atop}\!{0\atop }\!{0\atop }}}$ occurs with a nonzero
coefficient in the expression of $v_{25}$ via Chevalley generators
of $\g$, it must be that $(\Ad\,c_1)(v_{25})=-v_{25}\pm 2 v_{26}$.
Note that $(\Ad\,g)(e_2)$ is a linear combination of $e_2$ and
$e_{\textstyle{{0\atop }\!{0\atop }\!{1\atop 1}\! {0\atop
}\!{0\atop}\!{0\atop }\!{0\atop }}}$ and  $(\Ad\,g)(e_5)$ is a
linear combination of $e_5$ and $e_{\textstyle{{0\atop }\!{0\atop
}\!{1\atop 0}\! {1\atop }\!{0\atop}\!{0\atop }\!{0\atop }}}$. From
this it is immediate that $(\Ad\,c_2)(v_{26})$ is a linear
combination of $e_2$, $e_{\textstyle{{0\atop }\!{0\atop }\!{1\atop
1}\! {0\atop }\!{0\atop}\!{0\atop }\!{0\atop }}}$, $e_3$ and
$e_{\textstyle{{0\atop }\!{1\atop }\!{1\atop 0}\! {0\atop
}\!{0\atop}\!{0\atop }\!{0\atop }}}$. In particular,
$(\Ad\,c_2)(v_{26})\ne v_{26}$. In conjunction with our earlier
remarks this implies that $\cc_e^\Gamma$ is spanned by the image of
$e$. Therefore, $c_\Gamma(e)=d(e)$ and we can argue as before to
deduce the polynomiality of $U(\g,e)_\Gamma^{\rm ab}$.

The cases where $e$ is of type ${\sf D_4(a_1)}$ and $\g$ is of type
${\sf E_6}$ or ${\sf E_7}$ are very similar because here $e$, $c_1$
and $c_2$ have the same expressions as in the previous case; see
\cite[pp.~93, 124]{LT1}. If $\g$ is of type ${\sf E_7}$, then
$\cc_e=\cc_e(2)$ is $3$-dimensional, whilst in type ${\sf E_6}$ we
have that $\cc_e=\cc_e(0)\oplus \cc_e(2)$ where $\dim\,\cc_e(2)=3$
and $\cc_e(0)\cong \g_e(0)$ as vector spaces; see \cite{deG2}.
Arguing as in the ${\sf E_8}$-case we  obtain that $\cc_e(2)^\Gamma$
is generated by the image of $e$. This takes care of the ${\sf
E_7}$-case and reduces the ${\sf E}_6$-case to verifying that the
group generated by $c_1$ and $c_2$ acts fixed-point freely on the
$2$-dimensional toral subalgebra $\g_e(0)$. Since $\g_e(0)$ is
described explicitly in \cite[p.~93]{LT1}, the latter is easily seen
by a direct computation (we leave the details to the interested
reader).

Finally, suppose $\g$ is of type ${\sf G_2}$ and $e$ has type ${\sf
G_2(a_1)}$. Then $\cc_e=\cc_e(2)$ is $3$-dimensional by \cite{deG2}
and $\Gamma$ contains the image of $c_1=h_1(\omega)$ where $\omega$
is a primitive third root of $1$; see \cite[p.~66]{LT1}. Since
$\cc_e(2)\cong \g_e(2)$ has basis $\{e_{11}, e_{21}, e\}$ and
$e_{11}$, $e_{21}$ are short root vectors, it is straightforward to
see that $\cc_e^\Gamma$ is spanned by the image of $e$. Then
$c_\Gamma(e)=d(e)$ and we can argue as in the previous cases.

\smallskip

\noindent (d) If $\Gamma\cong S_4$ then $\g$ is of type ${\sf F_4}$
and $e$ has type ${\sf F_4(a_3)}$. In this case
$\cc_e=\cc_e(2)\cong\g_e(2)$ is $6$-dimensional, whilst $\Gamma\cong
C(e)$ is generated by $c_1=h_1(\omega)h_3(\omega)$,
$c_2=n_{1000\,}n_{0010\,}h_2(-1)h_3(-1)$ and
$c_3=(n_{0011\,}\,h_3(-\frac{2}{3})h_4(\frac{2}{3}))^u$ where
$\omega$ is a third primitive root of $1$ and
$u=x_{0011\,}(-\frac{1}{2})x_{0001}(1)x_{0010}(-1)$; see
\cite[p.~77]{LT1}. Straightforward verification shows that
$(\Ad\,c_1)(v_i)=\omega v_i$ for $i=2,4$,
$(\Ad\,c_1)(v_i)=\omega^{-1} v_i$ for $i=1,3$ and
$(\Ad\,c_1)(v_5)=v_5$. Since $\g_e(2)$ has basis
$\{v_1,v_2,v_3,v_4,v_5, e\}$, it follows that
$\g_e(2)^\Gamma\subseteq {\rm span}\,\{v_5,e\}$. By
\cite[p.~77]{LT1}, $e=e_{0100}+e_{1120}+e_{1111}+e_{0121}$ and
$v_5=e_{0100}+e_{1120}$. Since $c_2\in G_e$ and $\Ad\, c_2$ permutes
the lines $\K e_{0100}$ and $\K e_{1120}$, it must be that
$(\Ad\,c_2)(v_5)=v_5$.

Unfortunately, this means that we have to examine $(\Ad\,c_3)(v_5)$
which is rather more complicated. Suppose for a contradiction that
$(\Ad\,c_3)(v_5)=v_5$. Then
$(\Ad\,c_3)(e_{1111}+e_{0121})=e_{1111}+e_{0121}$. Note that
$(\Ad\,u)^{-1}(e_{1111}+e_{0121})\equiv e_{1111}+e_{0121}\mod V$
where $V={\rm span}\,\{e_{0122},e_{1121},e_{1122}\}$. It follows
that
$$\Ad\big(n_{0011\,}\,h_3(-\textstyle{\frac{2}{3}})h_4(\textstyle{\frac{2}{3}})u^{-1}\big)(e_{1111}+
e_{0121})\equiv \lambda e_{1111}+\mu e_{0110}\mod n_{0011\,}(V)$$
for some $\lambda,\mu\in\K^\times$. Since $n_{0011\,}(V)={\rm
span}\,\{e_{0100},e_{1110},e_{1100}\}$ we have that
$$\Ad\big(u^{-1}c_3\big)(e_{1111}+ e_{0121})=
\lambda e_{1111}+\mu e_{0110}+ae_{0100}+be_{1110}+ce_{1100}$$ for
some $a,b,c\in \K$. If $a\ne 0$ then $e_{0100}$ would occur with a
nonzero coefficient in the expression of
$(\Ad\,c_3)(v_5)=(\Ad\,u)(\lambda e_{1111}+\mu
e_{0110}+ae_{0100}+be_{1110}+ce_{1100})$ via Chevalley generators of
$\g$ contrary to our assumption that $\Ad\,c_3$ fixes $v_5$. Hence
$a=0$. But then $e_{0110}$ occurs with coefficient $\mu \ne 0$ in
the expression of $(\Ad\,c_3)(v_5)$ via Chevalley generators of
$\g$, a contradiction. We thus conclude that $\g_e(2)^\Gamma=\K e$
which yields $c_\Gamma(e)=1=d(e)$.

\smallskip

\noindent(e) Finally, suppose $\Gamma\cong S_5$. Then $\g$ is of
type ${\sf E_8}$ and $e$ has type ${\sf E_8(a_7)}$. By
\cite[p.~251]{LT1}, the group $\Gamma\cong C(e)$ contains
$c_1=h_2(\zeta)h_3(\zeta^4)h_4(\zeta)
h_6(\zeta^4)h_7(\zeta)h_8(\zeta^2),$ where $\zeta$ is a fifth
primitive root of $1$, and $$c_2=n_{\textstyle{{0\atop }\!{1\atop
}\!{0\atop 0}\! {0\atop }\!{0\atop}\!{0\atop }\!{0\atop
}}}n_{\textstyle{{0\atop }\!{0\atop }\!{1\atop 0}\! {0\atop
}\!{0\atop}\!{0\atop }\!{0\atop }}}n_{\textstyle{{0\atop }\!{0\atop
}\!{0\atop 1}\! {0\atop }\!{0\atop}\!{0\atop }\!{0\atop
}}}n_{\textstyle{{0\atop }\!{0\atop }\!{0\atop 0}\! {0\atop
}\!{1\atop}\!{0\atop }\!{0\atop }}}n_{\textstyle{{0\atop }\!{0\atop
}\!{0\atop 0}\! {0\atop }\!{0\atop}\!{1\atop }\!{0\atop
}}}n_{\textstyle{{0\atop }\!{0\atop }\!{0\atop 0}\! {0\atop
}\!{0\atop}\!{0\atop }\!{1\atop }}}h,$$ where
$h=h_1(-1)h_3(-1)h_5(-1)h_6(-1)h_8(-1)$. By \cite{deG2}, we have
that $\cc_e=\cc_e(2)\cong\g_e(2)$. Direct computation shows that the
basis $\{v_1,v_2,\ldots,v_9,e\}$ of $\g_e(2)$ described in
\cite[p.~256]{LT1} consists of eigenvectors for $\Ad\,c_1$. More
precisely, one has $(\Ad\,c_1)(v_i)=\zeta v_i$ for $i=2,5$,
$(\Ad\,c_1)(v_i)=\zeta^2v_i$ for $i=1,8$,
$(\Ad\,c_1)(v_i)=\zeta^3v_i$ for $i=3,6$,
$(\Ad\,c_1)(v_i)=\zeta^4v_i$ for $i=4,7$ and  $(\Ad\,c_1)(v_9)=v_9$.
Therefore, $\g_e(2)^\Gamma\subseteq {\rm span}\,\{v_9,e\}$. Since
$$v_9=e_{\textstyle{{0\atop }\!{0\atop }\!{0\atop 0}\!
{1\atop }\!{0\atop}\!{0\atop }\!{0\atop }\!}}+e_{\textstyle{{1\atop }\!{1\atop }\!{2\atop 1}\!
{1\atop }\!{1\atop}\!{0\atop }\!{0\atop }}}+e_{\textstyle{{1\atop }\!{1\atop }\!{1\atop 0}\!
{1\atop }\!{1\atop}\!{1\atop }\!{1\atop }}}+e_{\textstyle{{1\atop }\!{1\atop }\!{2\atop 1}\!
{1\atop }\!{1\atop}\!{1\atop }\!{0\atop }}},
$$ we see that $(\Ad\,c_2)(v_9)$ is a nonzero linear combination of
$e_{\textstyle{{0\atop }\!{1\atop }\!{1\atop 0}\! {1\atop
}\!{1\atop}\!{0\atop }\!{0\atop }\!}}$, $e_{\textstyle{{1\atop
}\!{1\atop }\!{1\atop 1}\! {1\atop }\!{1\atop}\!{1\atop }\!{0\atop
}}}$, $e_{\textstyle{{1\atop }\!{2\atop }\!{2\atop 1}\! {1\atop
}\!{0\atop}\!{0\atop }\!{0\atop }}}$ and $e_{\textstyle{{1\atop
}\!{1\atop }\!{1\atop 1}\! {1\atop }\!{1\atop}\!{1\atop }\!{1\atop
}}} $. But then $(\Ad\,c_2)(v_9)\ne v_9$ forcing $\g_e(2)^\Gamma=\K
e$ and implying that $c_\Gamma(e)=1=d(e)$. So we can argue as before
to show that $U(\g,e)_\Gamma^{\rm ab}$ is a polynomial algebra in
$d(e)$ variables. The proof of the proposition is now complete.
\end{proof}

Next we are going to investigate the case where $e$ is an  induced
nilpotent element of $\g$ which is not even and lies in a single
sheet of $\g$.
\begin{prop}\label{noteven1}
Suppose $e$ is not even, induced, and lies in a single sheet
$\mathcal{S}(e)$ of $\g$. Assume further that $e$ is not of type
${\sf E_7(a_2)}$ or ${\sf E_6(a_3)}+{\sf A_1}$ if $\g$ is of type
${\sf E_8}$, not of type ${\sf D_6(a_2)}$ if $\g$ is of type ${\sf
E_7}$ or ${\sf E_8}$, not of type ${\sf A_3+A_1}$ if $\g$ is of type
${\sf E_6}$,  and not of type ${\sf C_3(a_1)}$ if $\g$ is of type
${\sf F_4}$. Then $c(e)={\rm rk}\,\,\mathcal{S}(e)$ and
$U(\g,e)^{\rm ab}$ is a polynomial algebra in $c(e)$ variables.
Furthermore, $U(\g,e)^{\rm ab}_\Gamma\cong S(\cc_e^\Gamma)$ as
algebras and the value of $c_\Gamma(e)=\dim\,\cc_e^\Gamma$ is  given in
the 6$^{\,\rm th}$ column of Tables~1--6.
\end{prop}
\begin{proof}
If $e$ satisfies the above conditions then  $c(e)={\rm
rk}\,\,\mathcal{S}(e)$ by \cite[Proposition~2]{deG2}.
Corollary~\ref{poly}(i) then shows that $U(\g,e)^{\rm ab}\cong\,
S(\cc_e)$ and $U(\g,e)_\Gamma^{\rm ab}\cong\, S(\cc_e^\Gamma)$ as
$\K$-algebras. Since the value of $c(e)$ is computed in \cite{deG2}
in all cases, it remains to determine the value of $c_\Gamma(e)$. We
thus may assume from now on that $\Gamma\ne \{1\}$. Inspecting
Tables~1--6 one observes that this happens only if $\g$ is of type
${\sf E_8}$ or ${\sf E_7}$ and $\Gamma\cong S_2$.

\smallskip

\noindent (a) Suppose $\g$ is of type ${\sf E_8}$. Then $e$ is one
of ${\sf E_7(a_3)}$, ${\sf E_6(a_1)+A_1}$, ${\sf D_6(a_1)}$, ${\sf
A_4+2A_1}$, ${\sf D_5(a_1)}$ or ${\sf A_4+A_1}$.

If $e$ is of type ${\sf E_7(a_3)}$ then \cite{deG2} shows that
$\cc_e=\cc_e(2)\oplus\cc_e(4)\oplus\cc_e(6) \oplus\cc_e(8)$ and each
nonzero $\cc_e(i)$ is $1$-dimensional. Since $[\cc_e(1),\cc_e(1)]$
is $1$-dimensional by \cite{deG2}, using the explicit expressions
for the $v_i$'s given in \cite[p.~272]{LT1} it is straightforward to
see that $[v_1,v_2]\in \K^\times v_3$ and $[v_3,v_5]=\pm v_9$. This
implies that $\cc_e$ has basis consisting of the images of $e$,
$v_5$, $v_8$ and $v_{10}$. By \cite[p.~271]{LT1}, the group $\Gamma$
is generated by the image of $c=h_4(-1)$. As $(\Ad\,c)(v_5)=-v_5$,
$(\Ad\,c)(v_8)=v_8$ and $(\Ad\,c)(v_{10})=-v_{10}$ we deduce that
$c_\Gamma(e)=2$ in the present case.

If $e$ is of type ${\sf E_6(a_1)+A_1}$ then \cite{deG2} shows that
$\cc_e=\cc_e(0)\oplus\cc_e(2)\oplus\cc_e(4) $ and each nonzero
$\cc_e(i)$ is $1$-dimensional. It follows from \cite[p.~271,
272]{LT1} that $\g_e(0)$ is a $1$-dimensional toral subalgebra
spanned by the element $t\in \Lie(T)$ such that
$\alpha_i(t)=\delta_{7,i}$ for $1\le i\le 8$. The explicit
expressions for the $v_i$'s in {\it loc.\,cit.} show that
$[v_1,v_2]=\pm v_3$, $[v_1,v_6]=\pm v_7$ and $[v_2,v_5]=\pm v_8$. It
follows that the images of $t$, $e$ and $v_9$ form a basis of
$\cc_e$. It is also shown in {\it loc.\,cit.} that $\Gamma$ is
generated by the image of
$$c=n_{\textstyle{{1\atop }\!{2\atop }\!{4\atop 2}\!
{4\atop }\!{3\atop}\!{2\atop }\!{1\atop }}}
n_{\textstyle{{1\atop }\!{3\atop }\!{4\atop 2}\!
{3\atop }\!{3\atop}\!{2\atop }\!{1\atop }}}
n_{\textstyle{{2\atop }\!{3\atop }\!{4\atop 2}\!
{3\atop }\!{2\atop}\!{2\atop }\!{1\atop }}}h_1(-1)
h_2(-1)h_3(-1)h_5(-1)h_6(-1)h_8(-1).
$$ It is straightforward to see that $(\Ad\,c)(t)=-t$. Since the roots
${\textstyle{{1\atop }\!{2\atop }\!{4\atop 2}\! {4\atop
}\!{3\atop}\!{2\atop }\!{1\atop }}}, {\textstyle{{1\atop }\!{3\atop
}\!{4\atop 2}\! {3\atop }\!{3\atop}\!{2\atop }\!{1\atop }}},
{\textstyle{{2\atop }\!{3\atop }\!{4\atop 2}\! {3\atop
}\!{2\atop}\!{2\atop }\!{1\atop }}}$ and ${\textstyle{{0\atop
}\!{1\atop }\!{1\atop 0}\! {1\atop }\!{0\atop}\!{0\atop }\!{0\atop
}}}$ are pairwise orthogonal, we have that that
$(\Ad\,c)\Big(e_{\textstyle{{0\atop }\!{1\atop }\!{1\atop 0}\!
{1\atop }\!{0\atop}\!{0\atop }\!{0\atop
}}}\Big)=-e_{\textstyle{{0\atop }\!{1\atop }\!{1\atop 0}\! {1\atop
}\!{0\atop}\!{0\atop }\!{0\atop }}}$. Since $e_{\textstyle{{0\atop
}\!{1\atop }\!{1\atop 0}\! {1\atop }\!{0\atop}\!{0\atop }\!{0\atop
}}}$ occurs with a nonzero coefficient in the expression of $v_9$
via Chevalley generators of $\g$, this implies that
$(\Ad\,c)(v_9)=-v_9$. As a result, $\cc_e^\Gamma$ is spanned by the
image of $e$ and hence $c_\Gamma(e)=1$.

If $e$ is of type ${\sf D_6(a_1)}$ then \cite{deG2} shows that
$\cc_e=\cc_e(2)\oplus \cc_e(10)$ where $\dim\,\cc_e(2)=2$ and
$\dim\,\cc_e(10)=1$. By \cite[p.~256, 257]{LT1}, the Lie algebra
$\g_e(0)$ is semisimple and $\g_e(2)^{{\rm ad}\,\g_e(0)}$ is spanned
by $v_5$ and $e$, whilst $\g_e(10)$ is spanned by $v_{26}$ and
$v_{27}$. Since it is easy to see that $[v_1,v_{20}]=\pm v_{26}$,
the images of $e$, $v_5$ and $v_{27}$ form a basis of $\cc_e$. By
{\it loc.\,cit.}, the group $\Gamma$ is generated by the image of
$c=n_{\textstyle{{0\atop }\!{0\atop }\!{1\atop 1}\! {1\atop
}\!{1\atop}\!{1\atop }\!{1\atop }}} n_{\textstyle{{0\atop }\!{1\atop
}\!{1\atop 0}\! {1\atop }\!{1\atop}\!{1\atop }\!{1\atop }}} h_2(i)
h_3(i)$, where $i$ is a fourth primitive root of $1$, and $\Ad\,c$
fixes $v_{27}$. Note that $v_5=e_2+e_3$ and $\Ad\,c$ permutes the
lines $\K e_2$ and $\K e_3$. As $e=e_2+e_3+(\mbox{sum of other root
vectors})$ and $c$ fixes $e$, it must be that $(\Ad\,c)(v_5)=v_5$.
This shows that $c_\Gamma(e)=c(e)=3$.

If $e$ is of type ${\sf A_4+2A_1}$ then $\cc_e=\cc_e(0)$ is
$1$-dimensional by \cite{deG2}. On the other hand, it is immediate
from \cite[p.~221]{LT1} that $\g_e(0)=[\g_e(0),\g_e(0)]\oplus
\Lie(T_1)$ where $[\g_e(0),\g_e(0)]\cong\mathfrak{sl}_2$ and
$\Lie(T_1)$ is a $1$-dimensional toral subalgebra spanned by the
element $t\in\Lie(T)$ such that $\alpha_i(t)=\delta_{5,i}$ for all
$1\le i\le 8$. This implies that $\cc_e$ is spanned by the image of
$t$. Since $\Gamma$ is generated by the image of
$$c=n_{\textstyle{{1\atop }\!{2\atop }\!{2\atop 1}\!
{2\atop }\!{2\atop}\!{1\atop }\!{0\atop }}}
n_{\textstyle{{1\atop }\!{2\atop }\!{2\atop 1}\!
{2\atop }\!{1\atop}\!{1\atop }\!{1\atop }}}
n_{\textstyle{{1\atop }\!{3\atop }\!{5\atop 3}\!
{4\atop }\!{3\atop}\!{2\atop }\!{1\atop }}}
n_{\textstyle{{2\atop }\!{3\atop }\!{5\atop 2}\!
{4\atop }\!{3\atop}\!{2\atop }\!{1\atop }}}
h_1(-1)h_3(-1)h_4(-1)h_6(-1),
$$ it is straightforward to check that
$(\Ad\,c)(t)=-t$. So $c_\Gamma(e)=0$ in the present case and hence
$U(\g,e)_\Gamma^{\rm ab}\cong \K$ (i.e. $\mathcal{E}^\Gamma$ is a
single point!).

If $e$ is of type ${\sf D_5(a_1)}$ then \cite{deG2} shows that
$\cc_e=\cc_e(2)\oplus\cc_e(4)$ and
$\dim\,\cc_e(2)=\dim\,\cc_e(4)=1$. It follows from \cite[pp.~219,
220]{LT1} that $\g_e(0)\cong\mathfrak{sl}_4$, the subspace
$\g_e(2)^{{\rm ad}\,\g_e(0)}$ has basis $\{v_{15},e\}$, and
$\g_e(4)$ is spanned by $v_{17}$. Direct computation shows that
$[v_1,v_8]=\pm v_{15}$ implying that the images of $e$ and $v_{17}$
span $\cc_e$. By \cite[pp.~219]{LT1}, the group $\Gamma$ is
generated by the image of $c=n_{\textstyle{{1\atop }\!{2\atop
}\!{3\atop 1}\! {3\atop }\!{2\atop}\!{1\atop }\!{0\atop }}}
n_{\textstyle{{1\atop }\!{2\atop }\!{3\atop 2}\! {2\atop
}\!{2\atop}\!{1\atop }\!{0\atop }}} h_1(-1)h_2(-1)h_4(-1)$. Since
the roots $\textstyle{{1\atop }\!{2\atop }\!{3\atop 1}\! {3\atop
}\!{2\atop}\!{1\atop }\!{0\atop }}, \textstyle{{1\atop }\!{2\atop
}\!{3\atop 2}\! {2\atop }\!{2\atop}\!{1\atop }\!{0\atop }},
\textstyle{{1\atop }\!{1\atop }\!{1\atop 0}\! {0\atop
}\!{0\atop}\!{0\atop }\!{0\atop }}$ are pairwise orthogonal,
$(\Ad\,c)\Big(e_{\textstyle{{1\atop }\!{1\atop }\!{1\atop 0}\!
{0\atop }\!{0\atop}\!{0\atop }\!{0\atop }}}\Big)=
-e_{\textstyle{{1\atop }\!{1\atop }\!{1\atop 0}\! {0\atop
}\!{0\atop}\!{0\atop }\!{0\atop }}}.$ Since
$v_{17}=e_{\textstyle{{1\atop }\!{1\atop }\!{1\atop 0}\! {0\atop
}\!{0\atop}\!{0\atop }\!{0\atop }}}+ \mbox{(sum of other root
vectors)}$, we have $(\Ad\,c)(v_{17})=-v_{17}$. As a result,
$c_\Gamma(e)=1$.

If $e$ is of type ${\sf A_4+A_1}$ then $\cc_e=\cc_e(0)$ is
$1$-dimensional by \cite{deG2}. On the other hand,
\cite[p.~214]{LT1} shows that $\g_e(0)=
[\g_e(0),\g_e(0)]\oplus\Lie(T_1)$ where $[\g_e(0),\g_e(0)]\cong
\mathfrak{sl}_3$ and $\Lie(T_1)$ is spanned by the element $t\in
\Lie(T)$ such that $\alpha_5(t)=3$, $\alpha_7(t)=-5$ and
$\alpha_i(t)=0$ for $i=1,2,3,4,6,8$. Since $\Gamma$ is generated by
the image of
$$
c=n_{\textstyle{{0\atop }\!{1\atop }\!{1\atop 1}\!
{1\atop }\!{1\atop}\!{1\atop }\!{0\atop }}}
n_{\textstyle{{1\atop }\!{1\atop }\!{1\atop 0}\!
{1\atop }\!{1\atop}\!{1\atop }\!{0\atop }}}
n_{\textstyle{{1\atop }\!{2\atop }\!{2\atop 1}\!
{2\atop }\!{1\atop}\!{0\atop }\!{0\atop }}}
n_{\textstyle{{1\atop }\!{2\atop }\!{4\atop 2}\!
{3\atop }\!{2\atop}\!{1\atop }\!{0\atop }}}
h_1(-1)h_3(-1)h_4(-1)h_6(-1)h_8(-1),
$$ it is straightforward to check that
$\Ad\,c$ negates $t$. We thus deduce that $c_\Gamma(e)=0$ and hence $\mathcal{E}^\Gamma$ is a single point!

\smallskip

\noindent (b) Now suppose $\g$ is of type ${\sf E_7}$. Then $e$ is
one of ${\sf D_5(a_1)}$, ${\sf A_4+A_1}$, ${\sf D_4(a_1)+A_1}$ or
${\sf A_2+A_1}$.

If $e$ is of type ${\sf D_5(a_1)}$ then
$\cc_e=\cc_e(0)\oplus\cc_e(2)\oplus\cc_e(4)$ and all nonzero
$\cc_e(i)$ are $1$-dimensional; see \cite{deG2}. By
\cite[p.~140]{LT1}, we have that
$\g_e(0)=[\g_e(0),\g_e(0)]\oplus\Lie(T_1)$, where
$[\g_e(0),\g_e(0)]\cong\mathfrak{sl}_2$, and $\g_e(2)^{{\rm
ad}\,\g_e(0)}$ has basis $\{e, v_7\}$. It is also shown in {\it
loc.\,cit.} that $\Gamma$ is generated by the image of
$c=n_{\textstyle{{1\atop }\!{2\atop }\!{3\atop 1}\! {3\atop
}\!{2\atop}\!{1\atop }}} n_{\textstyle{{1\atop }\!{2\atop }\!{3\atop
2}\! {2\atop }\!{2\atop}\!{1\atop }}}h_1(-1)h_2(-1)h_4(-1)$ and
$\g_e(1)$ is spanned by $v_{23}$ which is fixed by the adjoint
action of $G_e$. It is straightforward to see that $\Lie(T_1)$ is
spanned by the element $t\in\Lie(T)$ such that
$\alpha_i(t)=\delta_{6,i}$ for $1\le i\le 7$ and $[v_1,v_4]=\pm
v_7$. Since $\Ad\,c$ negates $t$ we deduce that $\cc_e^\Gamma$ is
spanned by the images of $e$ and $v_{23}$ and so $c_\Gamma(e)=2$.

 If $e$ is of type ${\sf A_4+A_1}$ then $\cc_e=\cc_e(0)$ is $2$-dimensional by \cite{deG2}.
 On the other hand it is immediate from \cite[p.~138]{LT1} that $\g_e(0)$ is a $2$-dimensional subalgebra
of $\Lie(T)$ spanned by $t_1$ and $t_2$ such that $\alpha_i(t_1)=\delta_{5,i}$ and $\alpha_i(t_2)=\delta_{7,i}$
for $1\le i\le 7$. Since in the present case $\Gamma$ is generated by the image of
$$c=n_{\textstyle{{0\atop }\!{1\atop }\!{1\atop 1}\!
{1\atop }\!{1\atop}\!{1\atop }}}
n_{\textstyle{{1\atop }\!{1\atop }\!{1\atop 0}\!
{1\atop }\!{1\atop}\!{1\atop }}}n_{\textstyle{{1\atop }\!{2\atop }\!{1\atop 1}\!
{2\atop }\!{1\atop}\!{0\atop }}}n_{\textstyle{{1\atop }\!{2\atop }\!{4\atop 2}\!
{3\atop }\!{2\atop}\!{1\atop }}}h_3(-1)h_4(-1)h_6(-1),$$ one checks by direct computation that
$\Ad\,c$ negates both $t_1$ and $t_2$. Therefore, $c_\Gamma(e)=0$ and hence $\mathcal{E}^\Gamma$ is a single point!

If $e$ is of type ${\sf D_4(a_1)+A_1}$ then $\cc_e=\cc_e(2)$ is
$2$-dimensional by \cite{deG2}. By \cite[pp.~129, 130]{LT1}, we know
that $\g_e(0)$ is semisimple, $\g_e(2)^{{\rm ad}\,\g_e(0)}$ has
basis $\{v_9,v_{10}, v_{11},e\}$, and $\Gamma$ is generated by the
image of $c=n_{\textstyle{{1\atop }\!{1\atop }\!{1\atop 1}\! {0\atop
}\!{0\atop}\!{0\atop }}} n_{\textstyle{{1\atop }\!{1\atop }\!{1\atop
0}\! {1\atop }\!{0\atop}\!{0\atop }}}h_2(-1)$. It also follows from
{\it loc.\,cit.} that $[\g_e(1),\g_e((1)]^{{\rm ad}\,\g_e(0)}$ is
spanned by $[v_1,v_2]$ and $[v_3,v_4]$. Since $\dim\,\cc_e(2)=2$,
the above yields that $[v_1,v_2]$ and $[v_3,v_4]$ are linearly
independent. It is straightforward to see that
$(\Ad\,c)(v_1)\in\K^\times v_3$ and $(\Ad\,c)(v_2)\in \K^\times
v_4$. So $\Ad\,c$ must permute the lines $\K^\times [v_1,v_2]$ and
$\K^\times[v_3,v_4]$. Since $\Ad\,c$ acts on $\g_e(2)^{{\rm
ad}\,\g_e(0)}$ as an involution, this implies that it has
eigenvalues $\pm 1$ on the subspace $[\g_e(1),\g_e((1)]^{{\rm
ad}\,\g_e(0)}=\K[v_1,v_2]\oplus \K[v_3,v_4]$. Since the roots
$\textstyle{{1\atop }\!{1\atop }\!{1\atop 1}\! {0\atop
}\!{0\atop}\!{0\atop }}, \textstyle{{1\atop }\!{1\atop }\!{1\atop
0}\! {1\atop }\!{0\atop}\!{0\atop }}, \textstyle{{0\atop }\!{0\atop
}\!{0\atop 0}\! {0\atop }\!{0\atop}\!{1\atop }}$ are pairwise
orthogonal, $\Ad\,c$ fixes $v_{11}=e_7$. Since $\Ad\,c$ also fixes
$e=e_2+e_5+(\mbox{sum of other root vectors})$ and permutes the
lines $\K e_2$ and $\K e_5$, it must be that
$(\Ad\,c)(e_2+e_5)=e_2+e_5$. So $\Ad\, c$ fixes $v_{10}=e_2+e_5$.
But then the $(-1)$-eigenspace of $\Ad\,c$ on $\g_e(2)^{{\rm
ad}\,\g_e(0)}$ is $1$-dimensional. In conjunction with our earlier
remarks this gives $c(e)=c_\Gamma(e)=2$.

If $e$ is of type ${\sf A_2+A_1}$ then $\cc_e=\cc_e(0)$ is
$1$-dimensional by \cite{deG2}, whilst By \cite[p.~112]{LT1} shows
that $\g_e(0)=[\g_e(0),\g_e(0)]\oplus\Lie(T_1)$ where
$[\g_e(0),\g_e(0)] \cong\mathfrak{sl}_4$ and $\Lie(T_1)$ is spanned
by the element $t\in \Lie(T)$ such that $\alpha_i(t)=\delta_{4,i}$
for $1\le i\le 7$. Furthermore, the group $\Gamma$ is generated by
the image of
$$c=n_{\textstyle{{1\atop }\!{1\atop }\!{2\atop 1}\!
{1\atop }\!{1\atop}\!{1\atop }}}
n_{\textstyle{{1\atop }\!{1\atop }\!{2\atop 1}\!
{2\atop }\!{1\atop}\!{0\atop }}}n_{\textstyle{{1\atop }\!{3\atop }\!{4\atop 2}\!
{3\atop }\!{2\atop}\!{1\atop }}}h_3(-1)h_5(-1)h_7(-1).$$ Direct verification shows that
$\Ad\,c$ negates $t$. Since $\cc_e$ is spanned by the image of $t$ we conclude that
$c_\Gamma(e)=0$ in the present case and hence $\mathcal{E}^\Gamma$ is a single point!

This completes the proof of the proposition.
\end{proof}

Now we deal with those non-even, induced nilpotent elements which
lie in more than one sheet of $\g$. We first recall that if a
nilpotent orbit $\Oo\subset\g$ is induced from a nilpotent orbit
$\Oo_L\subset\Lie(L)$, where $P=LU$ is a proper parabolic subgroup
of $G$ with unipotent radical $U$, then the adjoint action of $G$
induces a surjective morphism
$$\pi\colon\,G\times^P\big(\overline{\Oo}_L+\Lie(U)\big)\rightarrow \overline{\Oo},
\qquad\quad (g,x)\longmapsto\, (\Ad\,g)\,x,$$ sometimes referred to
as a {\it generalized Springer map}; see \cite{fu} for more detail.
It is immediate from \cite[Proposition~6.1.2(4)]{Lo3} that $\pi$ is
birational (that is, generically injective) if and only if
$G_e\subset P$ for some $e\in \Oo\cap (\Oo_L+\Lie(U))$.

\begin{prop}\label{noteven2}
Suppose $e$ is not even, induced, and lies in more than one sheet of $\g$. Assume further that $e$ is not of type
${\sf E_7(a_5)}$ if $\g$ is of type ${\sf E_8}$. Then the following hold:
\begin{itemize}

\item[(i)\,] there exists a parabolic subgroup $P=LU$ of $G$ such that
the pair $(P,e)$ satisfies all conditions of Proposition~\ref{mult1}
and the centre of $L$ has dimension $r(e)$;

\smallskip

\item[(ii)\,]
$U(\g,e)_\Gamma^{\rm ab}$ is isomorphic to a polynomial algebra in
$r(e)=c_\Gamma(e)$ variables.
\end{itemize}
\end{prop}
\begin{proof}
Inspecting Tables~1--6 one observes that if $e$ not even, induced,
and lies in more than one sheet of $\g$ then $\g$ has type ${\sf
E_8}$ or ${\sf E_7}$ and all sheets containing $e$ have the same
rank equal to $r(e)$. Part~(i) then follows from
\cite[Proposition~3.1]{fu} which implies that in our situation there
exists at least one birational morphism
$\pi\colon\,G\times^P\big(\overline{\Oo}_L+\Lie(U)\big)\rightarrow
\overline{\Oo}$ with $e\in\Oo$ (the proof of Proposition~3.1 in {\it
loc.\,cit.} relies on Fu's earlier results obtained in \cite{fu07}).
In view of Proposition~\ref{mult1} and Corollary~\ref{poly}(ii) it
thus suffices to show that  the inequality $r(e)\ge c_\Gamma(e)$
holds for all nilpotent elements $e$ as above and
$\mathcal{E}_0^{\Gamma_0}\ne \emptyset$   (the notation of
Proposition~\ref{mult1}).

Suppose $\g$ is of type ${\sf E_8}$. Then $e$ is one of ${\sf
E_7(a_4)}$, ${\sf D_7(a_2)}$, ${\sf A_3+A_2}$.

If $e$ has type ${\sf E_7(a_4)}$ then $r(e)=2$ by \cite{deGE} and
$\cc_e=\cc_e(2)$ is $3$-dimensional by \cite{deG2}. According to
\cite[pp.~259, 260]{LT1}, the reductive group $C(e)^\circ$ acts
trivially on $\g_e(2)$ which has basis $\{v_3,v_4,v_5,e\}$, and
$\Gamma$ is generated by the image of $c=h_4(-1)$ which negates
$v_1$, $v_2$, $v_3$ and fixes $v_4$, $v_5$ and $e$. It is also
immediate from {\it loc.\,cit.} that $[\g_e(1),\g_e(1)]=\K
[v_1,v_2]$. Since $\dim\,\cc_e(2)=3$ and $\dim\,\g_e(2)=4$ it must
be that $[v_1,v_2]\ne 0$. Since $\Ad\,c$ fixes $[v_1,v_2]$ we now
deduce that the image of $v_3$ in $\cc_e$ is nonzero and hence
$c_\Gamma(e)=r(e)=2$. By \cite[3.4]{fu}, we can take for $P=LU$ a
parabolic subgroup of $G$ with ${\mathcal D}L$ of type ${\sf D_6}$
and for $e_0\in\Lie(L)$ a rigid nilpotent element associated with
the partition $(3,2^2, 1^5)$. But then
$\mathcal{E}_0^{\Gamma_0}=\mathcal{E}_0\ne \emptyset$ because
$\mathcal{E}_0$ is a single point by Remark~\ref{R5}(a).

If $e$ has type ${\sf D_7(a_2)}$ then $r(e)=2$ by \cite{deGE},
whilst \cite{deG2} shows that $\cc_e=\cc(0)\oplus \cc_e(2)\oplus
\cc_e(6)$ and all nonzero $\cc_e(i)$ are $1$-dimensional. On the
other hand, \cite[pp.~267, 268]{LT1} yields that $\Gamma$ is
generated by the image of $c=n_{\textstyle{{2\atop }\!{3\atop
}\!{5\atop 3}\! {4\atop }\!{3\atop}\!{2\atop }\!{1\atop }}}
n_{\textstyle{{2\atop }\!{4\atop }\!{5\atop 2}\! {4\atop
}\!{3\atop}\!{2\atop }\!{1\atop }}}h_4(-1)h_5(-1)$ and $\g_e(0)$ is
spanned by the element $t\in\Lie(T)$ such that
$\alpha_i(t)=\delta_{1,i}$ for all $1\le i\le 8$. It is
straightforward to see that $\Ad\,c$ negates $t$. It follows that
$r(e)=2\ge c_\Gamma(e)$. By \cite[Example~5.12]{fu07}, $e$ is
Richardson and we can take for $P=LU$ a parabolic subgroup of $G$
with ${\mathcal D}L$ of type ${\sf A_3+A_3}$ and for $e_0$ the zero
nilpotent element of $\li=\Lie(L)$. Then $\mathcal{E}_0^{\Gamma_0}$
contains the augmentation ideal of $U([\li,\li],e_0)=U([\li,\li])$.
In particular, $\mathcal{E}_0^{\Gamma_0}\ne \emptyset$.

If $e$ has type ${\sf A_3+A_2}$ then $r(e)=1$ and
$\cc_e=\cc_e(0)\oplus\cc_e(2)$ where
$\dim\,\cc_e(0)=\dim\,\cc_e(2)=1$; see \cite{deGE, deG2}. By
\cite[p.~201]{LT1}, the group $\Gamma$ is generated by the image of
$$c=n_{\textstyle{{0\atop }\!{0\atop }\!{1\atop 1}\!
{1\atop }\!{1\atop}\!{0\atop }\!{0\atop }}} n_{\textstyle{{0\atop
}\!{1\atop }\!{1\atop 0}\! {1\atop }\!{1\atop}\!{0\atop }\!{0\atop
}}}n_{\textstyle{{0\atop }\!{1\atop }\!{2\atop 1}\! {2\atop
}\!{1\atop}\!{1\atop }\!{0\atop }}}h_4(-1)h_5(-1)h_6(-1)h_7(-1)$$
and $\g_e(0)=[\g_e(0),\g_e(0)]\oplus\Lie(T_1)$ where
$[\g_e(0),\g_e(0)]\cong\mathfrak{so}_5$. Furthermore, $\Lie(T_1)$ is
spanned by the element $t\in\Lie(T)$ such that $\alpha_1(t)=-3$,
$\alpha_5(t)=2$, $\alpha_8(t)=-2$ and $\alpha_i(t)=0$ for
$i=2,3,4,6,7$. It is straightforward to check that $\Ad\,c$ negates
$t$. But then $r(e)=1\ge c_\Gamma(e)$. By \cite[3.4]{fu}, we can
take for $P=LU$ a parabolic subgroup of $G$ with ${\mathcal D}L$ of
type ${\sf D_7}$ and for $e_0\in\Lie(L)$ a rigid nilpotent element
associated with the partition $(2^2, 1^{10})$. Then
Remark~\ref{R5}(a) yields $\mathcal{E}_0^{\Gamma_0}\ne \emptyset$.

Now suppose $\g$ is of type ${\sf E_7}$. Then $e$ has type ${\sf
A_3+A_2}$. This case  is almost identical to the previous case. Here
we again have that $r(e)=1$ and $\cc_e=\cc_e(0)\oplus\cc_e(2)$ where
$\dim\,\cc_e(0)=\dim\,\cc_e(2)=1$; see \cite{deGE, deG2}. Also,
\cite[p.~131]{LT1} shows that the group $\Gamma$ is generated by the
image of
$$c=n_{\textstyle{{0\atop }\!{0\atop }\!{1\atop 1}\!
{1\atop }\!{1\atop}\!{0\atop }}} n_{\textstyle{{0\atop }\!{1\atop
}\!{1\atop 0}\! {1\atop }\!{1\atop}\!{0\atop
}}}n_{\textstyle{{0\atop }\!{1\atop }\!{2\atop 1}\! {2\atop
}\!{1\atop}\!{1\atop }}}h_4(-1)h_5(-1)h_6(-1)h_7(-1),$$
$\g_e(0)=[\g_e(0),\g_e(0)]\oplus\Lie(T_1)$ where
$[\g_e(0),\g_e(0)]\cong\mathfrak{sl}_2$, and $\Lie(T_1)$ is spanned
by the element $t\in\Lie(T)$ such that $\alpha_1(t)=-3$,
$\alpha_5(t)=2$ and $\alpha_i(t)=0$ for $i=2,3,4,6,7$. As before, we
check that $\Ad\,c$ negates $t$, yielding $r(e)=1\ge c_\Gamma(e)$.
By \cite[3.3]{fu}, we can take for $P=LU$ a parabolic subgroup of
$G$ with ${\mathcal D}L$ of type ${\sf D_6}$ and for $e_0\in\Lie(L)$
a rigid nilpotent element associated with the partition $(3,2^2,
1^{5})$. Then Remark~\ref{R5}(a) shows that
$\mathcal{E}_0^{\Gamma_0}\ne \emptyset$. This completes the proof.
\end{proof}

For the remaining seven induced orbits in the Lie algebras of
exceptional types our results are weaker. It turns out that the
variety $\mathcal{E}^\Gamma$ is non-empty in all cases and we can
determine its dimension, but our methods seem insufficient for
describing the irreducible components of $\mathcal{E}^\Gamma$. We
state our results as a remark:
\begin{rem}\label{except}
{\rm (a) If $\g$ has type ${\sf F_4}$ and $e$ is of type ${\sf
C_3(a_1)}$ then \cite{deGE} shows that $e$ lies in a single sheet of
rank $1$. So $\dim\,\mathcal{E}=1$ by \cite[Theorem~1.2]{Sasha3}. On
the other hand, $\cc_e=\cc_e(2)$ is $3$-dimensional by \cite{deG2}
whilst \cite[p.~76]{LT1} yields that $\cc_e\cong \g_e(2)$ and
 $\Gamma\cong S_2$ is generated by the image $c=h_4(-1)$. Also, $\g_e(2)$ has basis
 $\{v_1,v_2,e\}$ and $\Ad\,c$ negates $v_1$ and fixes $v_2$. This gives  $c_\Gamma(e)=2$.
 It follows from \cite[3.1]{fu} that $e$ is induced from a Levi subalgebra $\li$ of $\g$
 with $[\li,\li]$ of type ${\sf B_3}$ and a nilpotent element $e_0\in[\li,\li]$
 corresponding to the rigid partition $(2^2,1^3)\in \mathcal{P}_{1}(7)$.
Moreover, the conditions of Proposition~\ref{mult1} are satisfied
and $\mathcal{E}_0^{\Gamma_0}\ne \emptyset$ by Theorem~\ref{class1}.
Therefore, $\mathcal{E}^\Gamma\ne\emptyset$ and $
\dim\,\mathcal{E}^\Gamma\ge \z(\li)=1$. As a consequence,
$\dim\,\mathcal{E}^\Gamma=\dim\,\mathcal{E}=1$. In view of
Proposition~\ref{abgen}(ii), the variety $\mathcal{E}^\Gamma$ is
isomorphic to a non-empty $1$-dimensional closed subset of the
affine plane $\mathbb{A}^2$.
\smallskip

\noindent (b) If $\g$ has type ${\sf E_6}$ and $e$ is of type ${\sf
A_3+A_1}$ then $\Gamma=\{1\}$ and $e$ lies in a single sheet of rank
$1$ by \cite{deGE}. Then $\mathcal{E}^\Gamma=\mathcal{E}\ne
\emptyset$ and $\dim\,\mathcal{E}=1$ thanks to
\cite[Theorem~1.2]{Sasha3}. By \cite{deG2}, we have that
$\cc_e=\cc_e(0)\oplus\cc_e(2)$ is $2$-dimensional. As in the
previous case we now deduce that $\mathcal{E}^\Gamma=\mathcal{E}$ is
isomorphic to a non-empty $1$-dimensional closed subset of the
affine plane $\mathbb{A}^2$.

\smallskip

\noindent (c) If $\g$ has type ${\sf E_7}$ and $e$ is of type ${\sf
D_6(a_2)}$ then $\Gamma=\{1\}$ and $e$ lies in a single sheet of
rank $2$; see \cite{deGE}. On the other hand, $\cc_e=\cc_e(2)$ is
$3$-dimensional by \cite{deG2}. In view of
\cite[Theorem~1.2]{Sasha3} and Proposition~\ref{abgen}(ii) this
means that $\mathcal{E}^\Gamma=\mathcal{E}$ is  isomorphic to a
non-empty closed $2$-dimensional subset of the affine space
$\mathbb{A}^3$.

\smallskip

\noindent (d) If $\g$ has type ${\sf E_8}$ and $e$ is of type ${\sf
E_6(a_3)+A_1}$ then $e$ lies in a single sheet of rank $1$ by
\cite{deGE} and $\cc_e=\cc_e(2)$ is $3$-dimensional by \cite{deG2}.
On the other hand, \cite[pp.~245, 246]{LT1} yields that $\Gamma\cong
S_2$ is generated by the image of $c=h_4(-1)$, the subspace
$\g_e(2)$ has basis $\{v_5,v_6,v_7,e\}$, and $v_7=\pm[v_1,v_4]$.
From this it is immediate that the images of $v_5$, $v_6$ and $e$
under the natural epimorphism $\g_e(2)\twoheadrightarrow\cc_e$ form
a basis of $\cc_e$. Direct computations show that $c$ negates $v_5$
and fixes $v_6$ yielding $c_\Gamma(e)=2$.

By \cite[Theorem~1.2]{Sasha3}, the variety $\mathcal{E}$ is
non-empty and has dimension $r(e)=1$ whereas \cite[3.4]{fu} implies
that $e$ is induced from a Levi subalgebra $\li$ of $\g$ with
$[\li,\li]$ of type ${\sf E_7}$ and a nilpotent element
$e_0\in[\li,\li]$ of type ${\sf 2A_2+A_1}$ in such a way that all
conditions of Proposition~\ref{mult1} are satisfied (one should keep
in mind here that $\Gamma_0=\{1\}$ by \cite[p.~403]{C} and
$\mathcal{E}_0\ne \emptyset$ by \cite{GRU}). Hence
$\dim\,\mathcal{E}^\Gamma\ge 1$. We conclude that
$\mathcal{E}^\Gamma$ is isomorphic to a non-empty $1$-dimensional
closed subset of the affine plane $\mathbb{A}^2$.

\smallskip

\noindent (e)  If $\g$ has type ${\sf E_8}$ and $e$ is of type ${\sf
D_6(a_2)}$ then again $e$ lies in a single sheet of rank $1$ by
\cite{deGE} and $\cc_e=\cc_e(2)$ is $3$-dimensional by \cite{deG2}.
It follows from \cite[3.4]{fu} and Theorem~\ref{class1} that $e$ is
induced from a Levi subalgebra $\li$ of $\g$ with $[\li,\li]$ of
type ${\sf D_7}$ and a nilpotent element $e_0\in[\li,\li]$ attached
to the rigid partition $(3,2^4,1^3)\in\mathcal{P}_{1}(14)$ in such a
way that all conditions of Proposition~\ref{mult1} are satisfied.
This implies that $\mathcal{E}^\Gamma\ne \emptyset$ and
$\dim\,\mathcal{E}^\Gamma\ge \dim\,\z(\li)=1$. On the other hand,
combining \cite{deG2} and \cite[pp.~243, 244]{LT1} we deduce that
$\cc_e=\cc_e(2)\cong\g_e(2)$ and the images of $v_1$, $v_2$ and  $e$
form a basis of $\cc_e$.

The group $\Gamma\cong S_2$ is generated by the image of
$c=n_{\textstyle{{0\atop }\!{0\atop }\!{1\atop 1}\! {1\atop
}\!{1\atop}\!{1\atop }\!{1\atop }}} n_{\textstyle{{0\atop }\!{1\atop
}\!{1\atop 0}\! {1\atop }\!{1\atop}\!{1\atop }\!{1\atop }}}
h_4(-1)h_5(-1) $ and it is straightforward to see that $\Ad\,c$
permutes the lines $\K _2$ and $\K e_3$ and fixes
$e_{\textstyle{{0\atop }\!{0\atop }\!{1\atop 0}\! {1\atop
}\!{0\atop}\!{0\atop }\!{0\atop }}}$. Since $e=e_2+e_3+(\mbox{sum of
other root vectors})$, $\Ad\,c$ must permute $e_2$ and $e_3$. From
this it is immediate that $\Ad\,c$ negates $v_1$ and fixes $v_2$. As
a consequence, $c_\Gamma(e)=2$. Since $\dim\,\mathcal{E}=1$ by
\cite[Theorem~1.2]{Sasha3}, our earlier remarks now show that
 $\mathcal{E}^\Gamma$ is isomorphic to a non-empty closed $1$-dimensional subset of
 the affine plane $\mathbb{A}^2$.

\smallskip

\noindent (f)  If $\g$ has type ${\sf E_8}$ and $e$ is of type ${\sf
E_7(a_5)}$ then $e$ lies in two sheets both of which have rank $1$;
see \cite{deGE}. Also, $\cc_e=\cc_e(2)$ is $6$-dimensional by
\cite{deG2}. It follows from \cite[3.4]{fu} that $e$ is induced from
a Levi subalgebra $\li$ of $\g$ with $[\li,\li]$ of type ${\sf
E_6+A_1}$ and a nilpotent element $e_0\in[\li,\li]$ of type ${\sf
3A_1+0}$ in such a way that all conditions of
Proposition~\ref{mult1} are satisfied (it is important   here that
that $\Gamma_0=\{1\}$ by \cite[p.~402]{C} and $\mathcal{E}_0\ne
\emptyset$ by \cite{GRU}).
 This implies that $\mathcal{E}^\Gamma\ne \emptyset$ and
$\dim\,\mathcal{E}^\Gamma\ge \dim\,\z(\li)=1$. Since
$\dim\,\mathcal{E}=r(e)=1$ by \cite[Theorem~1.2]{Sasha3} this yields
$\dim\,\mathcal{E}^\Gamma=\dim\,\mathcal{E}=1$.

By \cite[pp.~247, 248]{LT1}, the subspace $\g_e(2)\cong\cc_e$ is
spanned by  $v_1, v_2, v_3, v_4, v_5$ and $e$, and $\Gamma\cong S_3$
is generated by the images of
$c_1=h_2(\omega)h_3(\omega)h_5(\omega)$ and
$c_2=n_{\alpha_2}n_{\alpha_3} n_{\alpha_5}h_3(-1)h_4(-1)$ where
$\omega$ is a primitive third root of $1$. Direct computations show
that $(\Ad\,c_1)(v_i)=\omega^{-1}v_i$ for $i=1,4$,
$(\Ad\,c_1)(v_i)=\omega v_i$ for $i=2,3$, and $(\Ad\,c_1)(v_5)=v_5$.
Since $v_5=e_4+e_{\textstyle{{0\atop }\!{1\atop }\!{1\atop 1}\!
{1\atop }\!{0\atop}\!{0\atop }\!{0\atop }}}$ and $\Ad\,c_2$ permutes
the lines $\K e_4$ and $\K e_{\textstyle{{0\atop }\!{1\atop
}\!{1\atop 1}\! {1\atop }\!{0\atop}\!{0\atop }\!{0\atop }}}$ and
fixes $e=e_4+e_{\textstyle{{0\atop }\!{1\atop }\!{1\atop 1}\!
{1\atop }\!{0\atop}\!{0\atop }\!{0\atop }}}+ (\mbox{sum of other
root vectors})$, it must be that $(\Ad\,c_2)(v_5)=v_5$. But then
$c_\Gamma(e)=2$ and we conclude that $\mathcal{E}^\Gamma$ is
isomorphic to a non-empty $1$-dimensional closed subset of the
affine plane $\mathbb{A}^2$.

\smallskip

\noindent (g)  If $\g$ is of type ${\sf E_8}$ and $e$ has type ${\sf
E_7(a_2)}$ then $\Gamma=\{1\}$ and $e$ lies in a single sheet of
rank $3$; see \cite{deGE}. Since $\cc_e=\cc_e^\Gamma$ is
$4$-dimensional by \cite{deG2} and $\mathcal{E}\ne\emptyset$  has
dimension $r(e)=3$ by \cite[Theorem~1.2]{Sasha3} we conclude that
$\mathcal{E}^\Gamma=\mathcal{E}$ is isomorphic to a non-empty
$3$-dimensional closed subset of the affine space $\mathbb{A}^4$. }
\end{rem}
\subsection{Applications to completely prime primitive ideals}
Let $e$ be an induced nilpotent element of $\g$ and let
$\mathcal{X}_{\Oo}$ be the set of all primitive ideals $I$ of the
universal enveloping algebra $U(\g)$ with ${\rm
VA}(I)=\overline{\Oo}$. Here $\Oo$ is the adjoint $G$-orbit of $e$
and ${\rm VA}(I)$ denotes the associated variety of $I$, i.e. the
zero locus in $\g$ of the ideal ${\rm gr}(I)$ of $S(\g)={\rm
gr}(U(\g))$ where, as usual, we identify the maximal spectrum of
$S(\g)$ with $\g$ by means of the Killing form of $\g$.

Let $J$ denote the defining ideal of $\overline{\Oo}$. Since
$A:=S(\g)/{\rm gr}(I)$ is a finitely generated $S(\g)$-module and
$J$ is the only minimal prime ideal of $S(\g)$ containing the
annihilator ${\rm Ann}_{S(\g)} A$ by Joseph's theorem, it follows
from \cite[Theorem~6.4]{Ma}, for instance, that there exist prime
ideals $\mathfrak{p}_1,\ldots, \mathfrak{p}_l$ containing $J$ and a
finite chain $\{0\}=A_0\subset A_1\subset \cdots\subset A_l=A$ of
$S(\g)$-submodules of $A$ such that $A_i/A_{i-1}\cong
S(\g)/\mathfrak{p}_i$ for  $1\le i\le l$. The {\it multiplicity} of
$\Oo$ in $U(\g)/I$, denoted ${\rm mult}_{\Oo}\,(U(\g)/I)$, is
defined as
$${\rm mult}_{\Oo}\,(U(\g)/I):={\rm Card}\,\{i\colon\,1\le i\le l,\ \mathfrak{p}_i=J\}.$$
It is well known that this number is independent of the choices
made; see \cite[9.6]{Ja04} for more detail. The results of the
previous subsection can be applied to characterise those  primitive
ideals $I\in\mathcal{X}_{\Oo}$ for which ${\rm
mult}_{\Oo}(U(\g)/I)=1$; we call such ideals {\it multiplicity
free}. The characterisation we obtain can be regarded as a
generalisation of M{\oe}glin's theorem \cite{Moe} on completely
prime primitive ideals of $U(\mathfrak{sl}_n)$ to simple Lie
algebras of other types (that theorem was recently reproved by
Brundan \cite{Br} by using the theory of finite $W$-algebras).

The rest of this subsection is devoted to proving Theorem~\ref{E}.
First we note that for $\g=\mathfrak{sl}_n$  the statement of Theorem~\ref{E}
is equivalent to M{\oe}glin's theorem thanks to (1) and the main results of \cite{Sasha4}.

\smallskip

\noindent(a) Suppose that $\g$ is one of $\mathfrak{so}_N$ or
$\mathfrak{sp}_N$ and $e$ is associated with a partition
$\lambda\in\mathcal{P}_\epsilon(N)$. In what follows we shall use
the notation introduced in the course of proving
Theorem~\ref{class1}.

Repeating the construction used in Part~(b) of the proof of
Theorem~\ref{class1} as many times as possible we arrive at a pair
$(\mathfrak{p},e_0)$, where $\mathfrak{p}$ is a parabolic subalgebra
of $\g$ with a Levi subalgebra $\li=\bar{\li}\oplus\mathfrak{m}$ and
$e_0$ is an almost rigid nilpotent element of $\mathfrak{m}$, such
that $e$ is induced from $e_0$ (regarded as an element of $\li$).
Let $P$ be the parabolic subgroup of $G$ with
$\Lie(P)=\mathfrak{p}$.

Recall Losev's homomorphism $\Xi\colon\,U(\g,e)\rightarrow
U(\li,e_0)'$ from Part~(b) of the proof of Proposition~\ref{mult1}.
Since $I\in\mathcal{X}_{\Oo}$ is multiplicity free, we have that
$I=Q_e\otimes_{U(\g,e)}\K_\eta$ for some $1$-dimensional
$\Gamma$-invariant representation of $\eta$ of $U(\g,e)$. By
\cite[Theorem~6.5.2]{Lo3}, the morphism $\xi^*\colon\,{\rm
Specm}\,U(\li, e_0)^{\rm ab}\rightarrow\, {\rm Specm}\,U(\g,e)^{\rm
ab}$ induced by $\Xi$ is finite. As explained in the proof of
Proposition~\ref{mult1} the inclusion $G_e\subset P_e$ implies that
$\Gamma$ acts on $\widetilde{\mathcal{E}}_0={\rm
Specm}\,U(\li,e_1)^{\rm ab}$ and $\xi^*$ maps the Zariski closed
subset $\widetilde{\mathcal{E}}_0^{\Gamma_0}$ into
$\mathcal{E}^\Gamma$ (as before, $\Gamma_0$ stands for the component
group of $L_{e_0}$). Since the variety $\mathcal{E}^{\Gamma}$ is
irreducible of dimension $s(\lambda)$ by Theorem~\ref{class1} and
$\dim(\widetilde{\mathcal{E}}_0^{\Gamma_0})\ge\dim\,\z(\li)=s(\lambda)$,
we deduce that $\xi^*\big(\widetilde{\mathcal{E}}_0^{\Gamma_0}\big)=
\mathcal{E}^\Gamma$ (one should keep in mind here that being a
finite morphism $\xi^*$ is closed and has finite fibres). As
$(\Ker\,\eta)/I_c$ lies in $\mathcal{E}^\Gamma$, we obtain that
$\eta=\eta_0\circ\xi$ for some $1$-dimensional $\Gamma_0$-invariant
representation $\eta_0$ of $U(\li, e_0)$ where
$\xi\colon\,U(\g,e)^{\rm ab}\rightarrow U(\li, e_0)^{\rm ab}$ is the
homomorphism of $\K$-algebras induced by $\Xi$.

Let $I_0\subset U(\li)$ be the annihilator of
$\widetilde{E}:=Q_0\otimes_{U(\li,e_0)}\,\K_{\eta_0}$ where $Q_0$ is
a generalised Gelfand--Graev $\li$-module associated with $e_0$. By
construction, $I_0$ is a multiplicity-free primitive ideal of
$U(\li)$. Since $\eta=\eta_0\circ\xi$, it follows from
\cite[Corollary~6.4.2]{Lo3} that $I$ is obtained from $I_0$ by
parabolic induction. More precisely,
$$I={\rm Ann}_{U(\g)}
\big(U(\g) \otimes_{U(\mathfrak{p})}\widetilde{E}\big)$$ where we
regard $\widetilde{E}$ as a $\mathfrak{p}$-module with the trivial
action of the nilradical of $\mathfrak p$.

\smallskip

\noindent (b) Suppose that $\lambda$ is not exceptional. As any
almost rigid nilpotent element of $\mathfrak{m}$ is non-singular,
combining Corollary~\ref{izosim} with Borho's classification of
sheets one observes that there exists a unique (up to conjugacy in
$P$) parabolic subalgebra $\mathfrak{p}_1=
\li_1\oplus\mathfrak{n}_1$ of $\g$ contained in $\mathfrak p$ and a
rigid nilpotent element $e_{1}\in\li_1$ such that $e_0\in\li$ is
induced from $e_{1}$ (here $\mathfrak{n}_1$ is the nilradical of
$\mathfrak{p}_1$ and $\li_1$ is a Levi subalgebra of
$\mathfrak{p}_1$ contained in $\li$). Then $e$ is induced from
$e_{1}$ by Proposition~\ref{ind}(3).

Let $\Xi_0$, $\xi^*_0$ and $\xi_0$ be the analogues of the maps
$\Xi$, $\xi^*$ and $\xi$ associated with the finite $W$-algebras
$(U(\li,e_0)$ and $U(\li_1, e_1)$. It follows from
 Theorem~\ref{class} that
$$U(\li,e_0)^{\rm ab}\,\cong\, U(\overline{\li})^{\rm ab}\otimes
U(\mathfrak{m},e_0)\,\cong\, S\big(\z(\overline{\li})\big)\otimes
U(\mathfrak{m},e_0)^{\rm ab}$$ is a polynomial algebra, hence a
domain. By \cite[Theorem~6.5.2]{Lo3}, the morphism
$\xi^*_0\colon\,{\rm Specm}\,U(\li_1, e_{1})^{\rm ab}\rightarrow\,
{\rm Specm}\,U(\li,e_0)^{\rm ab}$ induced by $\Xi_0$ is finite,
hence closed and has finite fibres. As ${\rm Specm}\,U(\li,e_0)^{\rm
ab}$ is an irreducible variety, the map $\xi^*_0$ must be
surjective. So there exists a $1$-dimensional representation
$\eta_1$ of $U(\li_1,e_1)$ such that $\eta_0=\eta_1\circ\xi_0$.

Let $I_1$ be the annihilator in $U(\li_1)$ of
$E:=Q_1\otimes_{U(\li_1,e_1)}\K_{\eta_1}$, where  $Q_1$ denotes a
generalised Gelfand--Graev $\li_1$-module associated with $e_1$, a
completely prime primitive ideal of $U(\li_1)$.  Since
$\eta_0=\eta_1\circ\xi_0$, applying \cite[Corollary~6.4.2]{Lo3} once
again we deduce that $I_0={\rm Ann}_{U(\li)} \big(U(\li)
\otimes_{U(\mathfrak{p}_1)} E\big)$ where $E$ is regarded as a
$\mathfrak{p}_1$-module with the trivial action of $\mathfrak{n}_1$.
Slightly abusing the notation introduced in Subsection~1.6, we then
have $$I\,=\,\mathfrak{I}_\mathfrak{p}^\g(I_0)
\,=\,\mathfrak{I}_\mathfrak{p}^\g\big
(\mathfrak{I}_{\mathfrak{p}_1}^\mathfrak{p}(I_1)\big)=\,
\mathfrak{I}_{\mathfrak{p}_1}^\g(I_1)$$ where the last equality
follows from \cite[10.4]{BGR} (which, in turn, follows from
transitivity of induction). As a consequence, $I={\rm Ann}_{U(\g)}
\big(U(\g) \otimes_{U(\mathfrak{p}_1)} E\big)$ proving
Theorem~\ref{E} in the present case.

\smallskip

\noindent (c) Now suppose $\lambda$ is exceptional. In this case, we
choose for $\mathfrak{p}_1$ the parabolic subalgebra $\Lie(P')$
where $P'$ is the parabolic subgroup introduced in Part~(d) of
Theorem~\ref{class1}. Then $e_0$ is a Richardson element of
$\mathfrak{p}_1$ and the map $\xi_0^*\colon {\rm
Specm}\,U(\li_1,0)^{\rm ab}\rightarrow {\rm Specm}\,U(\li,e_0)^{\rm
ab}$ is still surjective (here $\li_1$ is a Levi subalgebra of
$\mathfrak{p}_1$). Therefore, $\eta_0=\eta_1\circ \xi_0$ for some
$1$-dimensional representation of $U(\li_1,0)=U(\li_1)$. Applying
\cite[Corollary~6.4.2]{Lo3} and repeating almost verbatim the
argument from Part~(b) we now obtain that
$I=\mathfrak{I}_{\mathfrak{p}_1}^\g(I_1)$ for some ideal $I_1$ of
codimension $1$ in $U(\li_1)$. In other words, $I={\rm Ann}_{U(\g)}
\big(U(\g) \otimes_{U(\mathfrak{p_1})} \K_\lambda\big)$ for some
$1$-dimensional representation $\lambda$ of $\mathfrak{p}_1$.

\smallskip

\noindent (d) Finally, suppose $\g$ is exceptional. If $e$ is even,
then we choose for $\mathfrak{p}$ the Jacobson--Morozov parabolic
subalgebra $\mathfrak{p}(e)$ and for $e_0$ the zero element of the
Levi subalgebra of $\li=\g(0)$ of $\mathfrak{p}(e)$. Since
$U(\li,e_0)=U(\li)$ and it follows from Proposition~\ref{even} that
$\xi^*$ is still surjective, we argue as in Part~(a) to deduce that
$I=\mathfrak{I}_{\mathfrak{p}(e)}^\g(I_0)$ for some ideal $I_0$ of
codimension $1$ in $U(\li)$. As a consequence, $I={\rm Ann}_{U(\g)}
\big(U(\g) \otimes_{U(\mathfrak{p}(e))} \K_\lambda\big)$ for some
$1$-dimensional representation $\lambda$ of $\mathfrak{p}(e)$.

If $e$ satisfies the conditions of Proposition~\ref{noteven1} then
it lies in a unique sheet $\mathcal{S}(e)$ which contains an open
decomposition class $\mathcal{D}(\li,e_0)$ such that $e_0$ is a
rigid nilpotent element of $\li=\Lie(L)$. Furthermore, we may assume
that there exists a parabolic subalgebra
$\mathfrak{p}=\li\oplus\mathfrak{n}$, where $\mathfrak{n}$ is the
nilradical of $\mathfrak{p}$, such that $e\in\mathfrak{p}$ is
induced from $e_0\in\li$. In view of Proposition~\ref{noteven1} we
can now repeat verbatim our arguments from Part~(b) to deduce that
the map $\xi^*\colon {\rm Specm}\,U(\li,e_0)^{\rm
ab}\rightarrow\mathcal{E}$ is surjective. Thanks to
\cite[Corollary~6.4.2]{Lo3} this yields that
$I=\mathfrak{I}_{\mathfrak{p}}^\g(I_0)$ for some completely prime
primitive ideal $I_0$ of $U(\li)$ with ${\rm
VA}(I_0)=\overline{(\Ad\,L)\,e_0}$.

At last, if $e$ satisfies the conditions of
Proposition~\ref{noteven2} then we choose  a Levi subalgebra
$\li={\rm Lie}(L)$ and $e_0\in\li$ according to the recipe described
in Fu's paper \cite{fu} (see the proof of Proposition~\ref{noteven2}
for detail). Then there exists a parabolic subalgebra
$\mathfrak{p}=\li \oplus\mathfrak{n}$ such that $e_0$ is rigid in
$\li$ and $e\in \mathfrak{p}$ is induced from $e_0$. Since the map
$\xi\colon {\rm Specm}\,U(\li,e_0)^{\rm ab}\rightarrow \mathcal{E}$
is $\Gamma$-equivariant by our choice of $\mathfrak{p}$ and $e_0$,
combining \cite[Theorem~6.5.2]{Lo3} with Proposition~\ref{noteven2}
yields  that  $\xi^*\big(\widetilde{\mathcal
E}^{\Gamma_0}\big)=\mathcal{E}^\Gamma$. Then arguing as in Part~(a)
we obtain that $I=\mathfrak{I}_{\mathfrak{p}}^\g(I_0)$ for some
multiplicity-free primitive ideal $I_0$ of $U(\li)$ with ${\rm
VA}(I_0)=\overline{\Oo}_0$, where $\Oo_0=(\Ad\,L)\,e_0$.

This completes the proof of Theorem~\ref{E}. We note that for
$\g=\mathfrak{sl}_n$ one can argue as in Part~(c) (with $\li_1$
replaced by $\g$) to obtain yet another proof of M{\oe}glin's
theorem.

\vfill

\pagebreak

\begin{table}[htb]
\caption{{\sc Data for the induced orbits in type}  $\sf{E}_8$}
\label{data-exc}
\begin{tabular}{|c|c|c|c|c|c|c|c|}
\hline\hline Dynkin label  & Type of $\Gamma$ &Number of sheets&
Ranks of sheets &  $\dim \mathfrak{c}_e$ & $\dim\,
\mathfrak{c}_e^\Gamma$  \\ \hline
${\sf E_8}$ & $1$ & 1 (even)& 8 & 8 &8 \\
${\sf E_8(a_1)}$ & $1$ & 1 (even) & 7 & 7 &7 \\
${\sf E_8(a_2)}$ & $1$ & 1 (even) & 6 & 6 &6 \\
${\sf E_8(a_3)}$ & $S_2$ & 2 (even)& 6,5 & 7 &5\\
${\sf E_7}$ & $1$ & 1 & 4 & 4 &4\\
${\sf E_8(a_4)}$ & $S_2$ & 2 (even)& 5,4 & 6 &4
\\
${\sf E_8(b_4)}$ & $S_2$ & 2 (even) & 4,3 & 5 &4
\\
${\sf E_7(a_1)}$ & $1$ & 1 & 5 & 5 &5\\
${\sf E_8(a_5)}$ & $S_2$ & 2 (even) & 4,3 & 5 &3
\\
${\sf E_8(b_5)}$ & $S_3$ & 3 (even)& 4,4,3 & 7 &3
\\
${\sf D_7}$ & $1$ & 1 & 2 & 2 &2 \\
${\sf E_7(a_2)}$ & $1$ & 1 & 3 & 4 &$4^*$ \\
${\sf E_8(a_6)}$ & $S_3$ & 3 (even)& 3,3,2 & 6 &2
\\
${\sf D_7(a_1)}$ & $S_2$ & 2 (even) & 3,2 & 4 &3
\\
${\sf E_6}+{\sf A_1}$ & $1$ & 1 & 2 & 2 &2 \\
${\sf E_7(a_3)}$ & $S_2$ & 1 & 4 & 4 &2 \\
${\sf E_8(b_6)}$ & $S_3$ & 3 (even)& 2,2,1 & 5 &2
\\
${\sf E_6(a_1)}+{\sf A_1}$ & $S_2$ & 1 & 3 & 3 &1 \\
${\sf A_7}$ & $1$ & 1 & 1 & 1 &1 \\
${\sf E_6}$ & $1$ & 1 (even) & 4 & 4 &4 \\
${\sf D_7(a_2)}$ & $S_2$ & 2 & 2,2 & 3 &2 \\
${\sf D_6}$ & $1$ & 1 & 2 & 2 &2 \\
${\sf E_6(a_1)}$ & $S_2$ & 2 (even)& 3,3 & 4 &3 \\
${\sf D_5}+{\sf A_2}$ & $S_2$ & 2 (even) & 2,1 & 3 &2 \\
${\sf E_7(a_4)}$ & $S_2$ & 2 & 2,2 & 3 &2 \\
${\sf A_6}+{\sf A_1}$ & $1$ & 1 & 1 & 1 &1 \\
\hline
\end{tabular}
\end{table}

\pagebreak

\begin{table}[htb]
\caption{{\sc Data for the induced orbits in type  $\sf{E}_8$} (continued)}
\label{data-exc}
\begin{tabular}{|c|c|c|c|c|c|c|c|}
\hline\hline Dynkin label  & Type of $\Gamma$ &Number of sheets&
Ranks of sheets &  $\dim \mathfrak{c}_e$ & $\dim \,
\mathfrak{c}_e^\Gamma$  \\ \hline
${\sf D_6(a_1)}$ & $S_2$ & 1 & 3 & 3 &3 \\
${\sf A_6}$ & $1$ & 1 (even) & 2 & 2 &2 \\
${\sf E_8(a_7)}$ & $S_5$ & 4 (even) & 2,2,1,1 & 10 &1
\\
${\sf D_5}+{\sf A_1}$ & $1$ & 1 & 2 & 2 &2 \\
${\sf E_7(a_5)}$ & $S_3$ & 2 &1,1 & 6 &$2^*$
\\
${\sf D_6(a_2)}$ & $S_2$ & 1 & 1 & 3 &$2^*$
\\
${\sf E_6(a_3)}+{\sf A_1}$ & $S_2$ & 1 & 1 & 3 &$2^*$ \\
${\sf D_5}$ & $1$ & 1 (even) & 3 & 3 &3 \\
${\sf E_6(a_3)}$ & $S_2$ & 2 (even)& 2,2 & 3 &2
\\
${\sf D_4}+{\sf A_2}$ & $S_2$ & 1 (even)& 2 & 2 &2 \\
${\sf A_5}$ & $1$ & 1 & 1 & 1 &1 \\
${\sf D_5(a_1)}+{\sf A_1}$ & $1$ & 1 & 1 & 1 &1 \\
${\sf A_4}+{\sf A_2}+{\sf A_1}$ & $1$ & 1 & 1 & 1 &1\\

${\sf A_4}+{\sf A_2}$ & $1$ & 1 (even)& 1 & 1 &1 \\
${\sf A_4}+{\sf 2A_1}$ & $S_2$ & 1 & 1 & 1 &0 \\
${\sf D_5(a_1)}$ & $S_2$ & 1 & 2 & 2 &1 \\
${\sf A_4}+{\sf A_1}$ & $S_2$ & 1 & 1 & 1 &0 \\
${\sf D_4}+{\sf A_1}$ & $1$ & 1 & 1 & 1 &1 \\
${\sf D_4(a_1)}+{\sf A_2}$ & $S_2$ & 1 (even)& 1 & 1 &1 \\
${\sf A_4}$ & $S_2$ & 1 (even) & 2 & 2 &2 \\
${\sf A_3}+{\sf A_2}$ & $S_2$ & 2 & 1,1 & 2 &1 \\
${\sf D_4}$ & $1$ & 1 (even)& 2 & 2 &2 \\
${\sf D_4(a_1)}$ & $S_3$ & 2 (even)& 1,1 & 3 &1 \\
${\sf 2A_2}$ & $S_2$ & 1 (even)& 1 & 1 &1 \\
${\sf A_3}$ & $1$ & 1 & 1 & 1 &1 \\
${\sf A_2}$ & $S_2$ & 1 (even) & 1 & 1 &1\\
\hline
\end{tabular}
\end{table}

\vfill

\pagebreak

\begin{table}[htb]
\caption{{\sc Data for the induced orbits in type  $\sf{E}_7$}}
\label{data-exc}
\begin{tabular}{|c|c|c|c|c|c|c|c|}
\hline\hline Dynkin label  & Type of $\Gamma$ &Number of sheets&
Ranks of sheets &  $\dim \mathfrak{c}_e$ & $\dim \,
\mathfrak{c}_e^\Gamma$  \\ \hline
${\sf E_7}$ & $1$ & 1 (even)& 7 & 7 &7 \\
${\sf E_7(a_1)}$ & $1$ & 1 (even)& 6 & 6 &6 \\
${\sf E_7(a_2)}$ & $1$ & 1 (even)& 5 & 5 &5 \\
${\sf E_7(a_3)}$ & $S_2$ & 2 (even)& 5,4 & 6 &4 \\
${\sf E_6}$ & $1$ & 1 (even)& 4 & 4 &4 \\
${\sf D_6}$ & $1$ & 1 & 3 & 3 &3 \\
${\sf E_6(a_1)}$ & $S_2$ & 2 (even) & 4,3 & 5 &3 \\
${\sf E_7(a_4)}$ & $S_2$ & 2 (even) & 3,2 & 4 &3 \\
${\sf D_6(a_1)}$ & $1$ & 1 & 4 & 4 &4 \\
${\sf D_5}+{\sf A_1}$ & $1$ & 1 & 3 & 3 &3 \\
${\sf A_6}$ & $1$ & 1 (even) & 2 & 2 &2 \\
${\sf D_5}$ & $1$ & 1 (even) & 3 & 3 &3 \\
${\sf E_7(a_5)}$ & $S_3$ & 3 (even)& 3,3,2 & 6 &2 \\
${\sf D_6(a_2)}$ & $1$ & 1 & 2 & 3 &$3^*$\\
${\sf E_6(a_3)}$ & $S_2$ & 2 (even)& 2,2 & 3 &2\\
${\sf A_5}+{\sf A_1}$ & $1$ & 1 & 1 & 1 &1 \\
$({\sf A_5})'$ & $1$ & 1 & 1 & 1 &1 \\
${\sf D_5(a_1)}+{\sf A_1}$ & $1$ & 1 (even) & 2 & 2 &2 \\
${\sf D_5(a_1)}$ & $S_2$ & 1 & 3 & 3 &2 \\
${\sf A_4}+{\sf A_2}$ & $1$ & 1 (even)& 1 & 1 &1 \\
${\sf A_4}+{\sf A_1}$ & $S_2$ & 1 & 2 & 2 &0 \\
$({\sf A_5})''$ & $1$ & 1 (even)& 3 & 3 &3\\
${\sf D_4}+{\sf A_1}$ & $1$ & 1 & 1 & 1 &1 \\
${\sf A_4}$ & $S_2$ & 2 (even)  & 2,2 & 3 &2 \\
${\sf A_3}+{\sf A_2}+{\sf A_1}$ & $1$ & 1 (even) & 1 &1 &1 \\
${\sf A_3}+{\sf A_2}$ & $S_2$ & 2 & 1,1 &2 &1 \\
${\sf D_4(a_1)}+{\sf A_1}$ & $S_2$ & 1 & 2 & 2 &2\\
${\sf D_4}$ & $1$ & 1 (even) & 2 & 2 &2\\
${\sf A_3}+{\sf 2A_1}$ & $1$ & 1 & 1 &1 &1 \\
${\sf D_4(a_1)}$ & $S_3$ & 2 (even) & 1,1 & 3 &1\\
$({\sf A_3}+{\sf A_1})''$ & $1$ & 1 (even) & 2 &2 &2 \\
${\sf A_3}$ & $1$ & 1 & 1 &1 &1 \\
${\sf 2 A_2}$ & $1$ & 1 (even) & 1 &1 &1 \\
${\sf A_2}+{\sf 3 A_1}$ & $1$ & 1 (even) & 1 &1 &1 \\
${\sf A_2}+{\sf A_1}$ & $S_2$ & 1 & 1 &1 &0 \\
${\sf A_2}$ & $S_2$ & 1 (even) & 1 &1 &1 \\
$({\sf 3A_1})''$ & $1$ & 1 (even)& 1 &1 &1 \\
\hline
\end{tabular}
\end{table}

\vfill

\begin{table}[htb]
\caption{{\sc Data for the induced orbits in type  $\sf{E}_6$}}
\label{data-exc}
\begin{tabular}{|c|c|c|c|c|c|c|c|}
\hline\hline Dynkin label  & Type of $\Gamma$ &Number of sheets&
Ranks of sheets &  $\dim \mathfrak{c}_e$ & $\dim \,
\mathfrak{c}_e^\Gamma$  \\ \hline
${\sf E_6}$ & 1 & 1 (even)& 6 & 6 &6 \\
${\sf E_6(a_1)}$ & $1$ & 1 (even) & 5 & 5 &5 \\
${\sf D_5}$ &1 & 1 (even) & 4 & 4 &4 \\
${\sf E_6(a_3)}$ & $S_2$ & 2 (even) & 4,3 & 5 &3 \\
${\sf D_4}+{\sf A_1}$ & $1$ & 1 & 1 & 1 &1 \\
${\sf A_5}$ & $1$ & 1 & 2 & 2 &2 \\
${\sf D_5(a_1)}$ & 1 & 1 & 3 & 3 &3 \\
${\sf A_4}+{\sf A_1}$ & 1 & 1& 2 & 2 &2 \\
${\sf A_4}$ & $1$ & 1 (even)& 3 & 3 &3 \\
${\sf D_4(a_1)}$ & $S_3$ & 3 (even)& 2,2,1 & 5 &1 \\
${\sf A_3}+{\sf A_1}$ & $1$ & 1 & 1 & 2 &$2^*$ \\
${\sf A_3}$ & $1$ & 1  & 2 & 2 &2\\
${\sf A_2}+{\sf 2 A_1}$ & 1 & 1 & 1 & 1 &1 \\
${\sf 2 A_2}$ & 1 & 1  & 2 & 2 &2\\
${\sf A_2}+{\sf A_1}$ & 1 & 1 & 1 & 1 &1 \\
${\sf A_2}$ & $S_2$ & 1 (even) & 1 & 1 &1 \\
${\sf 2A_1}$ & 1 & 1 & 1 & 1 &1 \\
\hline
\end{tabular}
\end{table}

\begin{table}[htb]
\caption{{\sc Data for the induced orbits in type  $\sf{F}_4$}}
\label{data-exc}
\begin{tabular}{|c|c|c|c|c|c|c|c|}
\hline\hline Dynkin label  & Type of $\Gamma$ &Number of sheets&
Ranks of sheets &  $\dim \mathfrak{c}_e$ & $\dim \,
\mathfrak{c}_e^\Gamma$  \\ \hline

${\sf F_4}$ & 1 & 1 (even)& 4 & 4 &4 \\
${\sf F_4(a_1)}$ & $S_2$ & 2 (even) & 3,3 & 4 &3 \\
${\sf F_4(a_2)}$ &$S_2$ & 2 (even) & 2,2 & 3 &2 \\
${\sf B_3}$ & 1 & 1 (even) & 2 & 2 &2 \\
${\sf C_3}$ & $1$ & 1  & 2 & 2 &2 \\
${\sf F_4(a_3)}$ & $S_4$ & 3 (even) & 2,1,1 & 6 &1 \\
${\sf C_3(a_1)}$ & $S_2$ & 1 & 1 & 3 &$2^*$ \\
${\sf B_2}$ & $S_2$ & 1 (even)& 1 & 1 &1 \\
${\sf \widetilde{A}_2}$ & $1$ & 1 (even)& 1 & 1 &1 \\
${\sf A_2}$ & $S_2$ & 1 (even)& 1 & 1 &1 \\
\hline
\end{tabular}
\end{table}

\begin{table}[htb]
\caption{{\sc Data for the induced orbits in type  $\sf{G}_2$}}
\label{data-exc}
\begin{tabular}{|c|c|c|c|c|c|c|c|}
\hline\hline Dynkin label  & Type of $\Gamma$ &Number of sheets&
Ranks of sheets &  $\dim \mathfrak{c}_e$ & $\dim \,
\mathfrak{c}_e^\Gamma$  \\ \hline
${\sf G_2}$ & 1 & 1 (even)& 2 & 2 &2 \\
${\sf G_2(a_1)}$ & $S_3$ & 2 (even) & 1,1 & 3 &1 \\
\hline
\end{tabular}
\end{table}



\end{document}